\documentclass[12pt,english,fleqn,liststotoc,bibtotoc,idxtotoc,BCOR7.5mm,tablecaptionabove]{amsart}
\usepackage[T1]{fontenc}
\usepackage[latin1]{inputenc}
\usepackage{geometry}
\usepackage{color}
\usepackage{amsthm}
\usepackage{amstext}
\usepackage{amssymb}
\usepackage{amsmath}
\usepackage{graphicx}
\usepackage{young}
\usepackage{xypic}
\usepackage{amscd}
\usepackage[section]{placeins}
\usepackage[svgnames]{xcolor}
\usepackage[unicode=true,
bookmarks=true,bookmarksnumbered=true,bookmarksopen=false,
breaklinks=false,pdfborder={0 0 1},backref=false,colorlinks=true]
{hyperref}
\hypersetup{pdftitle={},
	pdfauthor={Nick Early},
	pdfsubject={},
	pdfkeywords={},
	linkcolor=black, citecolor=black, urlcolor=blue, filecolor=blue,  pdfpagelayout=OneColumn, pdfnewwindow=true,  pdfstartview=XYZ, plainpages=false, pdfpagelabels}
\usepackage{breakurl}

\makeatletter

\theoremstyle{plain}
\newtheorem{thm}{\protect\theoremname}
\theoremstyle{plain}
\newtheorem{conjecture}[thm]{\protect\conjecturename}
\theoremstyle{plain}

\theoremstyle{remark}

\theoremstyle{plain}
\newtheorem{lem}[thm]{\protect\lemmaname}
\theoremstyle{plain}
\newtheorem{prop}[thm]{\protect\propositionname}
\theoremstyle{remark}

\theoremstyle{remark}
\newtheorem{rem}[thm]{\protect\remarkname}
\theoremstyle{definition}
\newtheorem{defn}[thm]{\protect\definitionname}
\theoremstyle{definition}
\newtheorem{example}[thm]{\protect\examplename}
\theoremstyle{plain}
\newtheorem{cor}[thm]{\protect\corollaryname}
\theoremstyle{plain}

\numberwithin{thm}{section}

\usepackage{ifpdf}
\ifpdf

\IfFileExists{lmodern.sty}
{\usepackage{lmodern}}{}

\fi 

\usepackage{multicol}

\newenvironment{nouppercase}{%
	\renewcommand{\uppercasenonmath}[1]{}}{}

\makeatother

\providecommand{\claimname}{\inputencoding{latin9}Claim}
\providecommand{\conjecturename}{\inputencoding{latin9}Conjecture}
\providecommand{\corollaryname}{\inputencoding{latin9}Corollary}
\providecommand{\definitionname}{\inputencoding{latin9}Definition}
\providecommand{\examplename}{\inputencoding{latin9}Example}
\providecommand{\lemmaname}{\inputencoding{latin9}Lemma}
\providecommand{\notename}{\inputencoding{latin9}Note}
\providecommand{\propositionname}{\inputencoding{latin9}Proposition}
\providecommand{\questionname}{\inputencoding{latin9}Question}
\providecommand{\remarkname}{\inputencoding{latin9}Remark}
\providecommand{\theoremname}{\inputencoding{latin9}Theorem}
\providecommand{\problemname}{\inputencoding{latin9}Problem}

\newcommand{\symm}{\mathfrak{S}}

\newcommand\twoheaduparrow{\mathrel{\rotatebox{90}{$\twoheaduparrow$}}}
\newcommand\twoheaddownarrow{\mathrel{\rotatebox{270}{$\twoheaddownarrow$}}}

\makeatletter
\let\@wraptoccontribs\wraptoccontribs
\makeatother

\begin{document}
	\author{Nick Early}
	\contrib[with an appendix by]{Donghyun Kim}
		\thanks{Institute for Advanced Study. \\
		email: \href{mailto:earlnick@ias.edu}{earlnick@ias.edu}}
	\title[Honeycomb Tessellations and Graded Permutohedral Blades]{Honeycomb Tessellations and Graded Permutohedral Blades}
\begin{nouppercase}
	\maketitle
\end{nouppercase}
	\begin{abstract}
	This paper investigates enumerative aspects of permutohedral blades, which provide a generalization of the notion of the tropical hyperplane arrangement.  Blade provide the combinatorial underpinning of generalized biadjoint scalar scattering amplitudes in work of Cachazo, Early, Guevara and  Mizera (CEGM).  We construct a graded basis for a vector space of indicator functions of blades, with the grading determined by the dimension of the support of the function.  We prove a Minkowski sum decomposition law into lines and tripods and we explore connections to the cohomology ring of certain moduli spaces and a non-planar analog of the square move for plabic graphs.
	
	We give a closed formula for the graded dimension of the basis.  It is shown in an Appendix by Donghyun Kim that the coefficients appearing in the numerator of the generating function for the graded dimension are symmetric, and that they sum to $\frac{(2j)!}{j!}$.  Unimodality is still open.  
	\end{abstract}

	\begingroup
	\let\cleardoublepage\relax
	\let\clearpage\relax
	\tableofcontents
	\endgroup

\section{Introduction}\label{sec: blade honeycombs}
We initiate the study of permutohedral blades, which lie at the intersection of matroid theory, tropical geometry, topology and scattering amplitudes.  In this section we provide motivation for the problem which we study, by describing first the construction of tessellations of $\mathbb{R}^{n-1}$ with generic (weight) permutohedra, then connections to matroid theory, and finally  tropical geometry.

Throughout this paper, the field for the vector space generated by linear combinations of indicator functions will be the rational numbers $\mathbb{Q}$.

Consider the open simplex $W\subset \mathbb{R}^n\slash (1,1,\ldots, 1)\mathbb{R}$ which is characterized by the $n$ facet inequalities $ y_1<y_2<\cdots <y_n<y_1+1$; this is often called the (fundamental) \textit{Weyl alcove}.  Following (\cite{PostnikovPermutohedra}, Proposition 16.6), by reflecting a point $y\in W$ across the affine hyperplanes $y_i-y_j \in\mathbb{Z}$, one obtains the set of vertices of a honeycomb tessellation of $\mathbb{R}^n\slash (1,1,\ldots, 1)\mathbb{R}$ with generalized (weight) permutohedra; the property to note here is that the permutohedra have disjoint interiors and intersect only on common facets.  

\begin{figure}[h!]
	\centering
	\includegraphics[width=0.4\linewidth]{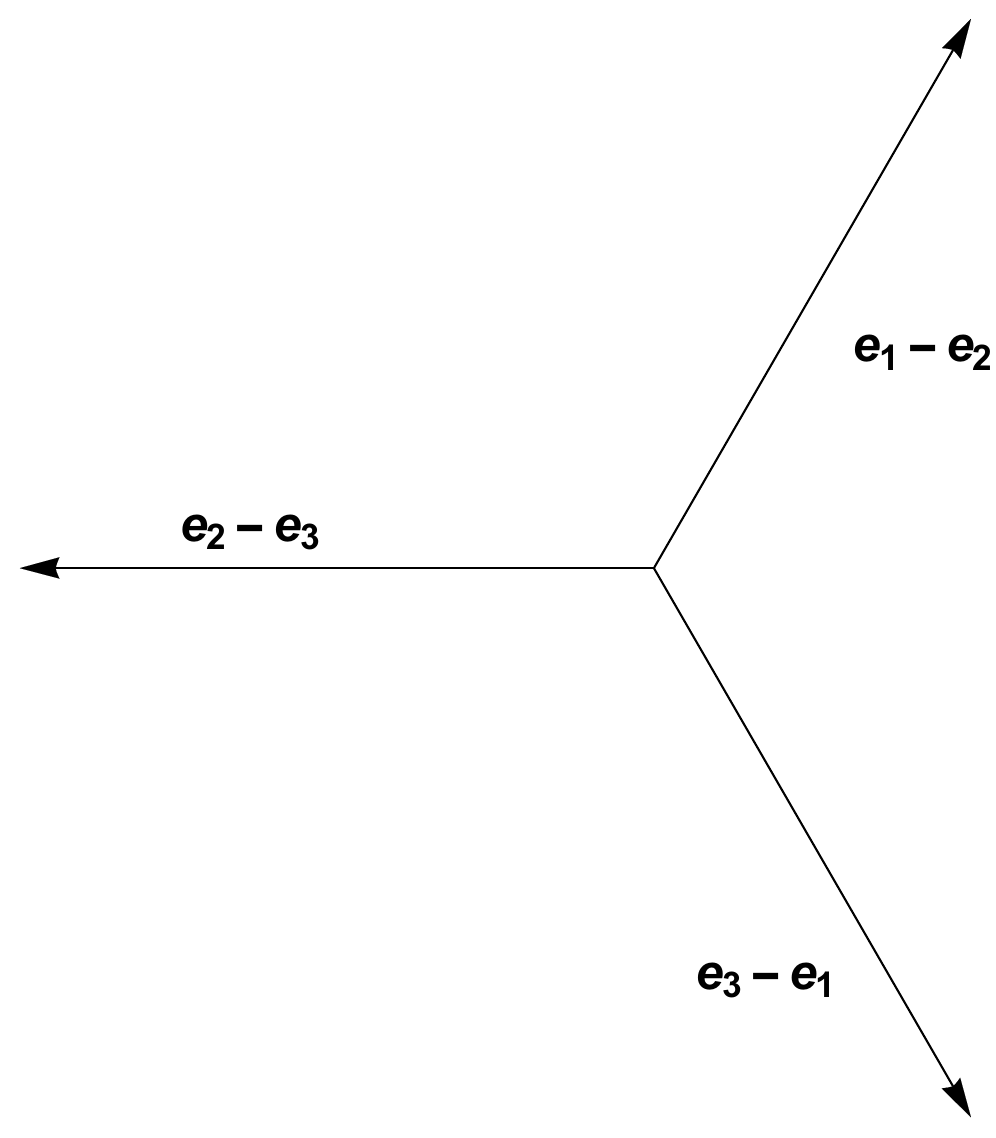}
	\includegraphics[width=0.5\linewidth]{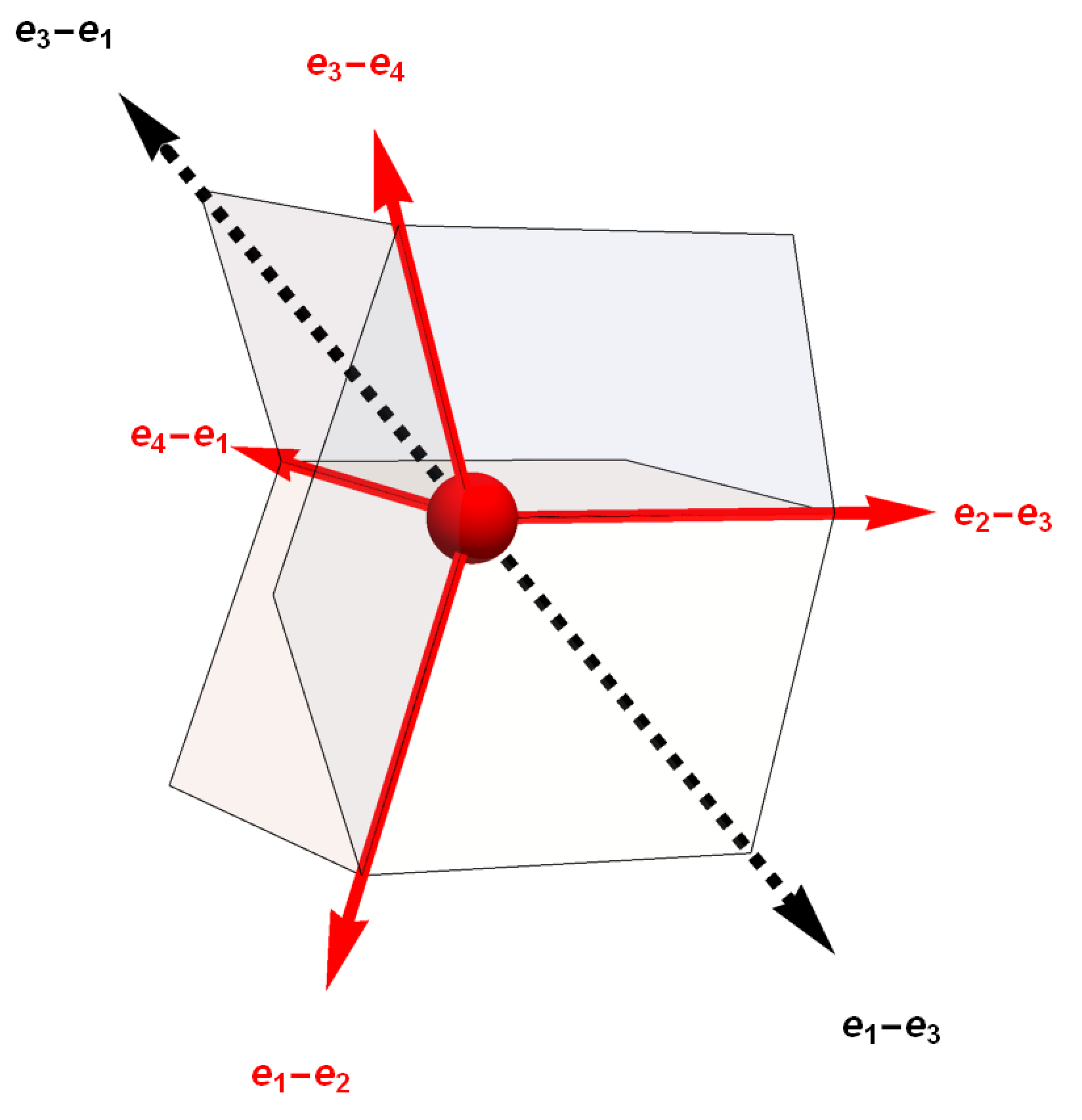}
	\caption{Left: the tripod $((1,2,3))$.  Right: the blade $((1,2,3,4))$ as a Minkowski sum of the two tripods $((1,2,3))$ and $((1,3,4))$.}
	\label{fig:blade3coordinates000}
\end{figure}
Let us pass to an (affine) section of the quotient $\mathbb{R}^n\slash (1,1,\ldots, 1)\mathbb{R}$, say $\left\{x\in\mathbb{R}^n: \sum_{i=1}^n x_i=r \right\}$ for some integer $r$.  In particular, let us introduce the notation $V_0^n = \{x\in \mathbb{R}^n: \sum_{i=1}^n x_i=0 \}$.

We are interested in the geometry of the affine ``shell'' $\mathcal{H}_y$ which consists of the union of the facets of the permutohedra in the tessellation; it partitions a neighborhood of a generic point $y$ into $n$ chambers whose tangent cones at $y$ are the $n$ cyclically related sets
$$\lbrack i,i+1,\ldots, i-1\rbrack_y :=\left\{y+\sum_{j=i,i+1,\ldots, i-2} t_i (e_i-e_{i+1}):t_i\ge 0\right\},$$
intersecting at their common point $y$, such that $\mathcal{H}_y$ coincides locally with the union of their facets.  See Figure \ref{fig:blade3coordinates000} for the neighborhoods of a vertex for $n=3,4$.  The tessellation itself is given below with several fundamental regions chosen, starting in Figure \ref{fig:weightpermtessellation}.  The set of these cones forms a complete \textit{simplicial fan} with base point at $y$ with $\mathcal{H}_y$ locally the (n-2)-skeleton.  The $(n-2)$-skeleton of the fan equals the union of the $\binom{n}{2}$ codimension-2 cones
$$((1,2,\ldots, n))_y := \bigcup_{1\le i< j\le n}\left\{y+ \left(\sum_{\ell\not\in\{i,j\}}t_\ell(e_\ell-e_{\ell+1})  \right): t_\ell\ge0  \right\}.$$

Now we come to the moral of the story about honeycombs:

\begin{quote}\label{quotation: permutohedral tessellation}
	\textit{Choose a small open set $\mathcal{Y}_y$ about the (generic) point $y$ such that, for any $x\in \mathcal{Y}_y$ we have $x_i-x_j\not\in\mathbb{Z}$ for all $i\not=j$; this guarantees that the image of $\mathcal{Y}_y$ under reflection across any affine hyperplane $x_i-x_j\in\mathbb{Z}$ is disjoint from $\mathcal{Y}_y$.  Then on the neighborhood $\mathcal{Y}_y$ of $y$, the sets $\mathcal{H}_y$ and $((1,2,\ldots, n))_y$ coincide.}
\end{quote}
For the case $n=4$ see Figure \ref{fig:BladePermutohedralSingularityInTetrahedron}, where $((1,2,3,4))_{(1,1,1,1)}$ is framed by the tetrahedron with vertices the four permutations of $(4,0,0,0)$.  

Figures \ref{fig:weightpermtessellation} and \ref{fig:weightpermtessellation2} describe the passage from the honeycomb tessellation to the fundamental parallelepiped, where the generalized permutohedra have two out $n=3$ possible edge lengths; see also \cite{OcneanuYoutubeVideos,PostnikovPermutohedra}.
\begin{figure}[!htb]
	\minipage{0.5\textwidth}
		\includegraphics[width=.75\linewidth]{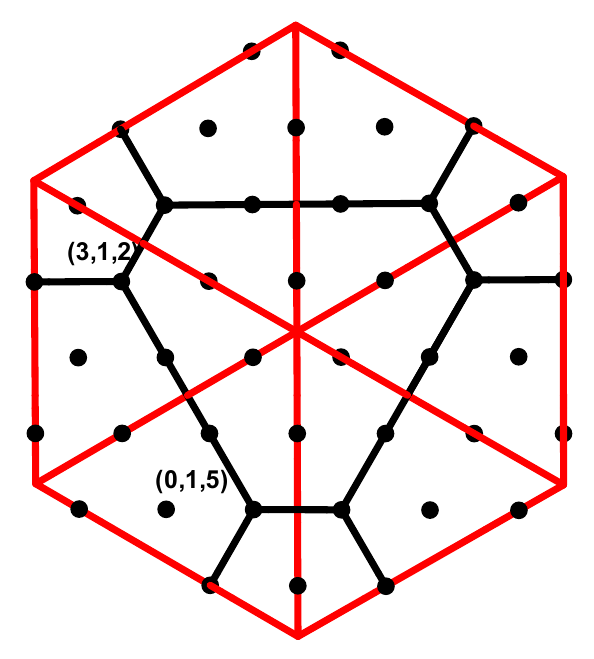}
			\label{fig:hexagonalbladetessellation}
			\caption{One choice of a (hexagonal) period; opposite edges of the red hexagon are identified. Red edges are segments of affine hyperplanes, placed at $x_1-x_2 \in 3+6\mathbb{Z},\ x_2-x_3 \in 1+6\mathbb{Z},\ x_3-x_1\in 2+6\mathbb{Z}$.  Black edges in the tiling are segments parallel to the three root directions $e_i-e_j$.  Near the vertices of the (three) weight permutohedra (with black edges of lengths respectively 1,2,3) the black shell coincides with a blade.  The blade $((1,2,3))$ is translated to for instance the point $(3,1,2)$, while the blade $((1,3,2))$ is translated to for instance $(0,1,5)$; compare to Figures \ref{fig:weightpermtessellation} and \ref{fig:weightpermtessellation2}.}
	\endminipage\hfill
	\minipage{0.4\textwidth}
	\includegraphics[width=\linewidth]{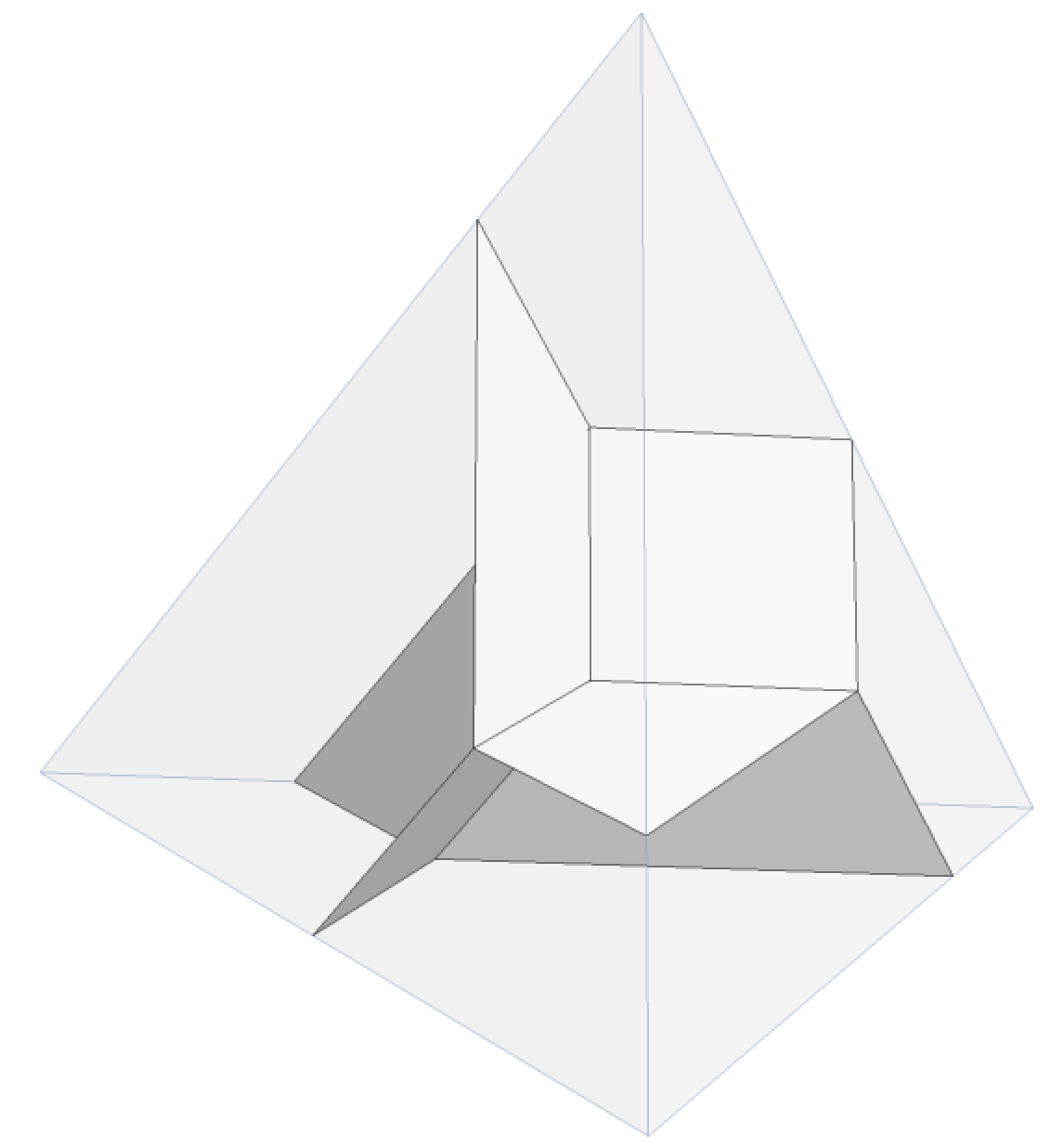}
	\caption{Part of the blade $((1,2,3,4))$.  It is the 2-skeleton of a complete fan, here viewed inside a tetrahedral frame.  Note that here the edges in the 1-skeleton of the blade are not perpendicular to the facets of the tetrahedron, but are rather parallel to its edges.}\label{fig: triangulation independence complete blade}
	\label{fig:BladePermutohedralSingularityInTetrahedron}
	\endminipage\hfill
\end{figure}

\begin{figure}[h!]
	\centering
	\includegraphics[width=0.45\linewidth]{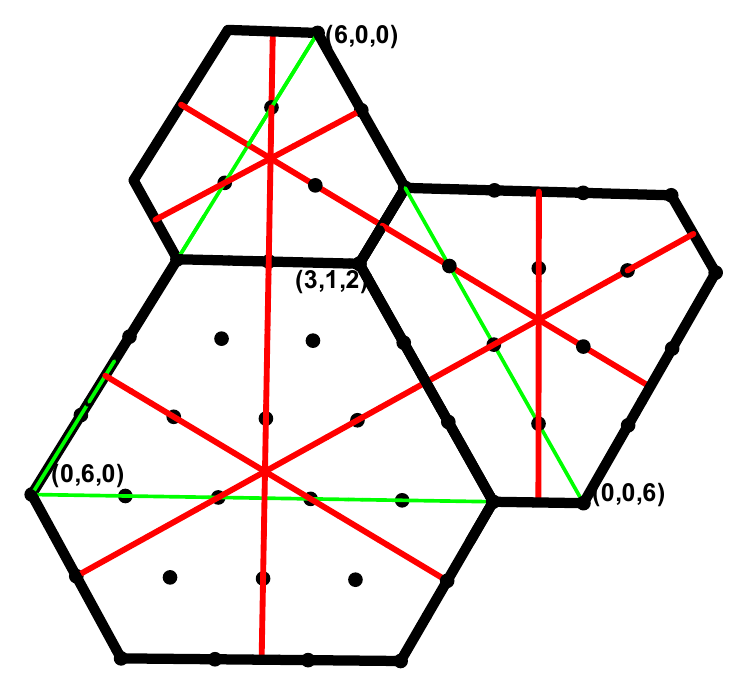}
	\caption{Tessellating with three generalized (weight) permutohedra.  Here rather than taking rational-valued coordinates for the point $y$, we have dilated the Weyl alcoves by a factor 6; then $y=(3,1,2)$ is reflected across the (red) affine reflection hyperplanes placed at $x_1-x_2 \in 3+6\mathbb{Z},\ x_2-x_3 \in 1+6\mathbb{Z},\ x_3-x_1\in 2+6\mathbb{Z}$.
  }
	\label{fig:weightpermtessellation}
		\includegraphics[width=0.6\linewidth]{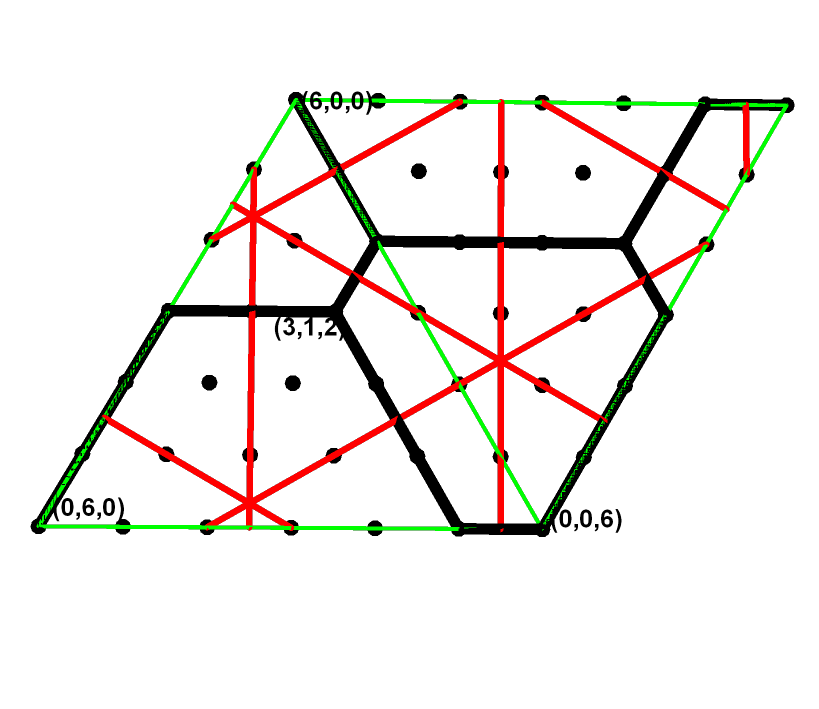}
	\caption{Same as in Figures \ref{fig:hexagonalbladetessellation} and \ref{fig:weightpermtessellation}, but shifted into the fundamental parallelepiped.  Then the shell (outlined in black) coincides locally with translations of either the blade $((1,2,3))$ at for instance the points $(3,1,2)$  or $(4,-3,5)$, or with the blade $((1,3,2))$ at for instance the point $(4,0,2)$. The lengths of the arms extending from the point $(3,1,2)$ add to the edge length $(=6)$ of the ambient simplex (outlined in green on bottom left).  Such subdivisions have appeared in the work of Ocneanu \cite{OcneanuVideo}.}
	\label{fig:weightpermtessellation2}
\end{figure}

\begin{rem}
	Since the first version of this preprint was posted, blades have appeared in various contexts, in the study of generalized scattering amplitudes and planar bases of kinematic invariants, which are dual to rays in the positive tropical Grassmannian $\text{Trop}^+G(k,n)$.  See \cite{Early19WeakSeparationMatroidSubdivision} for a connection between certain matroidal blade arrangements on hypersimplices and weakly separated collections.  For constructions of a distinguished set of codimension 1 singularities of scattering amplitudes in theoretical physics, see \cite{Early2019PlanarBasis}.  An embedding of the positive tropical Grassmannian $\text{Trop}^+G(k,n)$ into the set of weighted blade arrangements on the vertices of $\Delta_{k,n}$ was constructed in \cite{Early2020WeightedBladeArrangements}.  See also closely related work in \cite{CachazoBorges,CGUZ} on collections and arrays of Feynman diagrams, where each Feynman diagram is assigned a metric and the collections of metrics are glued together.  See also \cite{GuevaraZhang} for further study of rays and further progress in constructing of generalized biadjoint scalar amplitudes.
\end{rem}

\subsection{What is a blade?}\label{sec: what is a blade}  
Blades, defined by A. Ocneanu in \cite{OcneanuVideo}, are objects which are closely related to matroid polytopes and their subdivisions, and tropical geometry, with connections to the cohomology ring of certain configuration spaces \cite{EarlyReiner}.

Section \ref{sec: blades motivation matroid subdivisions} provides the link to matroid subdivisions: we show that blades in $n$ coordinates induce matroid subdivisions of the middle hypersimplex $\Delta(n,2n)$.  Here, by a \textit{matroid subdivision} we mean a decomposition of a hypersimplex\footnote{However, in general, the ambient polytope for a matroid subdivision does not have to be the hypersimplex, the matroid polytope for the uniform matroid with $d$ elements and rank $k$.}, into a collection of finitely many matroid polytopes $P_1,\ldots, P_d$ such that every (nonempty) intersection $P_{j_1} \cap \cdots \cap P_{j_\ell}$ is itself a face of $P_{j_1},\ldots, P_{j_\ell}$.

We move quickly to be explicit about the link to tropical geometry: in Section \ref{sec: bladesInTropicalGeometry} we exhibit blades as tropical hypersurfaces.

In Section \ref{sec: configuration space circle} we use what we learned about blades to define parametrizations of all possible sets of distinct collision classes of (labeled) particles on the circle modulo simultaneous rotation, in the configuration space $\mathcal{O}^n := U(1)^n \slash U(1)$.  

In general blades may be degenerate.  For this it is helpful to proceed using the Minkowski sum operation $\oplus$, defined for polyhedra $\pi_1,\pi_2$ by 
$$\pi_1\oplus \pi_2 = \{v_1+v_2:v_1\in \pi_1,\ v_2\in\pi_2\}.$$

Given disjoint (nonempty) subsets $S_1,S_2\subsetneq\{1,\ldots, n\}$, define the half space
$$\lbrack S_1,S_2\rbrack := \left\{x\in V_0^n: x_{S_1}\ge 0, x_i = 0\text{ for all }i\in (S_1\cup S_2)^c \right\}.$$
For any ordered set partition $(S_1,\ldots, S_k)$ of $\{1,\ldots, n\}$, define the cone
$$\lbrack S_1,\ldots, S_k\rbrack = \lbrack S_1,S_2\rbrack \oplus \lbrack S_2,S_3\rbrack \oplus \cdots \oplus \lbrack S_{k-1},S_k\rbrack.$$

Definition \ref{defn:blade}, reproduced here, is due to A. Ocneanu.
\begin{defn}[\cite{OcneanuVideo}]
	For an ordered set partition $(S_1,\ldots, S_k)$ of $\{1,\ldots, n\}$ with $k\ge 3$, the \textit{blade} $((S_1,\ldots, S_k))$ is the set theoretic \textit{union} of the Minkowski sums of cones,
	$$((S_1,\ldots, S_k))=\bigcup_{1\le i<j\le k} \lbrack S_1,S_2\rbrack\oplus \cdots\oplus \widehat{\lbrack S_i,S_{i+1}\rbrack} \oplus \cdots \oplus\widehat{\lbrack S_j,S_{j+1}\rbrack} \oplus \cdots\oplus \lbrack S_k,S_1\rbrack,$$
	where the hat means that that corresponding term has been omitted, and where we adopt the convention $S_{k+1}=S_1$.  When $k=2$ and $S_1\sqcup S_2 = \{1,\ldots, n\}$, then put 
	$$((S_1,S_2)) = \lbrack S_1\rbrack \oplus \lbrack S_2 \rbrack = \left\{\sum_{i\in S_1\cup S_2}t_i e_i: t_{S_1} = 0,\ t_{S_2}=0 \right\}.$$
	Finally, when $k=1$ we set $$((12\cdots n)) =V_0^n.$$
\end{defn}

To each blade $((S_1,\ldots, S_k))$ we shall assign two piecewise-constant (surjective) functions, namely its \textit{characteristic} function 

$$\Gamma_{S_1,\ldots, S_k}:\left\{x\in \mathbb{R}^n: \sum_{i=1}^n x_i=0 \right\}\rightarrow \{0,1\},$$
which takes the value 1 on $((S_1,\ldots, S_k))$ and zero on its complement, and a \textit{graduated} function
$$\lbrack (S_1,\ldots, S_k)\rbrack: \left\{x\in \mathbb{R}^n: \sum_{i=1}^n x_i=0 \right\}\rightarrow \{1,\ldots, k\},$$
where the level set corresponding to the height $j\in \{1,\ldots, k\}$ has \textit{co}dimension $j-1$ in $V_0^n$.

These two families of functions capture essentially complementary aspects of the blades $((S_1,\ldots, S_k))$; both will be useful in what follows and we believe deserve further study.  In Section \ref{sec: concluding remarks} we describe which properties of these two families remain to be established.

The main construction involves characteristic functions of certain embeddings of the type $A_2$ root system.  The characteristic functions of elementary (one-dimensional) tripods, are denoted by $\gamma_{i,j,k}$ and satisfy the cyclic index relation $\gamma_{i,j,k}=\gamma_{k,i,j}$.  More generally, lumped tripods are labeled by cyclic classes of 3-block ordered set partitions $(S_1,S_2,S_3)$, of the subset $S_1\cup S_2\cup S_3$ of $\{1,\ldots, n\}$.  Additionally we have characteristic functions $1_{S}$ of subspaces $\{x\in V_0^n: x_i=0\text{ whenever }i\not\in S\}$ for any proper nonempty subset $S$ of $\{1,\ldots, n\}$.  Together, these satisfy the fundamental relation
\begin{eqnarray*}
	\gamma_{S_1,S_2,S_3} + \gamma_{S_1,S_3,S_2} & = &  1_{S_1\cup S_2} + 1_{S_2\cup S_3} + 1_{S_3\cup S_1} - 1_{S_1}1_{S_2}1_{S_3}.
\end{eqnarray*}
\begin{figure}[h!]
	\centering
	\includegraphics[width=1\linewidth]{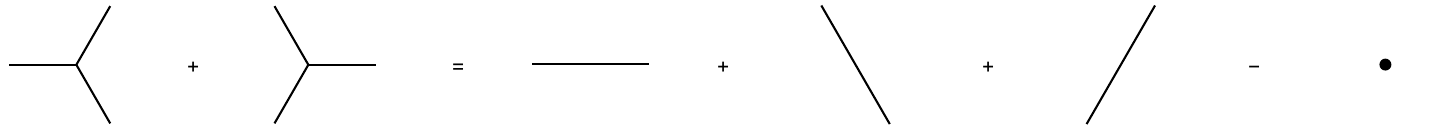}
	\caption{The blade relation
		$\gamma_{1,2,3} + \gamma_{1,3,2} = 1_{12}+1_{23}+1_{31} -1$.}
	\label{fig:bladerelationdim3}
\end{figure}

In the case of singlet blocks (as is the case in Figure \ref{fig:bladerelationdim3}), since all of the functions $1_{1},\ldots,  1_{n}$ are equal to the identity with respect to convolution, we could (and usually will) abuse notation write the fundamental blade relation as
$$\gamma_{1,2,3} + \gamma_{1,3,2} = 1_{12}+1_{23}+1_{31} -1,$$
rather than the full expression
$$\gamma_{1,2,3} + \gamma_{1,3,2} = 1_{12}1_3+1_1 1_{23}+1_{31}1_2 -1_1 1_2 1_3.$$

\subsection{Blades induce matroid subdivisions}\label{sec: blades motivation matroid subdivisions}

For integers $1\le d\le m-1$, denote by $\Delta(d,m) = \left\{x\in \lbrack 0,1\rbrack^m:\sum_{j=1}^m x_j=d \right\}$ the $d^\text{th}$ hypersimplex, considered here to be an affine section of the unit cube. 

In what follows, we shall give a natural lifting of blades $((S_1,\ldots, S_k))\subset V_0^n$ to matroid subdivisions (of the middle hypersimplex $\Delta(n,2n)$).  Matroid subdivisions have appeared prominently in algebraic geometry and configuration spaces (Hacking, Keel and Tevelev \cite{HackingKeelTevelev} and Kapranov \cite{Kapranov}), in tropical geometry \cite{SpeyerSturmfels}, and in particular more recently in (Fink \cite{Fink Thesis} and Schroter \cite{Schroter}).  The subdivisions which we obtain here specialize to those derived from Catalan matroid polytopes from \cite{Ardila2003,BidkoriSullivant}. 

Let us fix the affine-linear map $\psi: \mathbb{R}^{2n}\rightarrow\mathbb{R}^n$, defined componentwise by 
$$(x_1,\ldots, x_{2n}) \mapsto (x_1+x_2-1,x_3+x_4-1\ldots, x_{2n-1}+x_{2n}-1).$$
Schubert matroids were first studied by Crapo \cite{Crapo} under the name of nested matroids and have been rediscovered in various contexts, for example \cite{DerksenFink} in the context of Hopf algebras, in the context of permutations \cite{OcneanuPermutations}, and in the context of positivity and scattering amplitudes \cite{Early19WeakSeparationMatroidSubdivision}.  We record the definition here.


\begin{defn}
	A decorated ordered set partition $((S_1)_{s_1},\ldots, (S_k)_{s_k})$ of $(\{1,\ldots, m\},d)$ is an ordered set partition $(S_1,\ldots, S_k)$ of $\{1,\ldots, m\}$ together with a list of (not necessarily positive, in general) integers $(s_1,\ldots, s_k)$ with $\sum_{j=1}^k s_j=d$.  It is said to be of type $\Delta(d,n)$ in the case that $\sum_{j=1}^k s_j=d$ and $1\le s_j\le\vert S_j\vert-1 $, for $d$ a positive integer.  In this case we write $((S_1)_{s_1},\ldots, (S_k)_{s_k}) \in \text{OSP}(\Delta(d,m))$, and we denote by $\lbrack (S_1)_{s_1},\ldots, (S_k)_{s_k}\rbrack$ the convex polytope in $\Delta(d,m)$ that is cut out by the inequalities
	\begin{eqnarray}\label{eq:hypersimplexPlate}
		x_{S_1} & \ge & s_1 \nonumber\\
		x_{S_1\cup S_2} & \ge & s_1+s_2\nonumber\\
		& \vdots & \\
		x_{S_1\cup\cdots \cup S_{k-1}} & \ge & s_1+\cdots +s_{k-1}\nonumber\\
		\sum_{j=1}^m x_j & = & d.\nonumber
	\end{eqnarray}
\end{defn}
\begin{rem}
	It is easy to see by direct comparison of the equations that when $((S_1)_{s_1},\ldots, (S_k)_{s_k}) \in \text{OSP}(\Delta(d,m))$, then $\lbrack (S_1)_{s_1},\ldots, (S_k)_{s_k}\rbrack$ is the truncation of a (translated) permutohedral cone from Definition \ref{defn: permutohedral cone} (and \cite{EarlyCanonicalBasis}) to a unit cube, where $d\in\{1,\ldots, n-1\}$.
		
	That Equation \eqref{eq:hypersimplexPlate} defines a matroid polytope can be seen by reversing the directions of the inequalities in Equation \eqref{eq:hypersimplexPlate} and observing that we recover the matroid polytope for a nested matroid of rank $d$.  See for instance Example 3.3 of \cite{SchroterMultisplits}.
\end{rem}

There is a natural involution on the set of hypersimplices, $\iota:\Delta(d,m) \mapsto \Delta(m-d,m)$ given coordinate-wise by $x_i\mapsto 1-x_i$ for each $i=1,\ldots, m$.  In particular, it follows that $\iota$ preserves the middle hypersimplices $\Delta(n,2n)$ for $i=1,\ldots, 2n$.

Moreover, from equations \eqref{eq:hypersimplexPlate} we see that $\iota: \Delta(n,2n)\rightarrow \Delta(n,2n)$ acts by 
$$\lbrack (S_1)_{s_1},\ldots, (S_k)_{s_k}\rbrack \mapsto \lbrack (S_1)_{\vert S_1\vert - s_1},\ldots, (S_k)_{\vert S_k\vert - s_k}\rbrack. $$
Call $((S_1)_{s_1},\ldots, (S_k)_{s_k})$ (and respectively the hypersimplex plate $\lbrack(S_1)_{s_1},\ldots, (S_k)_{s_k}\rbrack$) \textit{central} if $\vert S_i\vert =2s_i$.

\begin{cor}
	We have
	$$\iota\left(\lbrack (S_1)_{s_1},\ldots, (S_k)_{s_k}\rbrack\right) = \lbrack (S_1)_{s_1},\ldots, (S_k)_{s_k}\rbrack$$
	if and only if $\lbrack  (S_1)_{s_1},\ldots, (S_k)_{s_k}\rbrack $ is central.
\end{cor}
 It follows that central hypersimplex plates are exactly those obtained as inverse images under $\psi$ of plates in $V_0^n$.  Now, as shown in Corollary \ref{cor: cyclic sum identity}, the set of $k$ cyclic block rotations of a plate $\lbrack T_1,\ldots, T_k\rbrack$ have union all of $V_0^n$ and intersect only on common facets, and it follows that together they lift via $\psi$ to induce a matroid subdivision of $\Delta(n,2n)$, partitioning $\Delta(n,2n)$ into $k$ matroid polytopes.  We say that the blade $((T_1,\ldots, T_\ell))$, which is the union of the facets of the cyclic block rotations of $\lbrack T_1,\ldots, T_\ell\rbrack$ which are in the interior of $\Delta(n,2n)$, induces the matroid subdivision 
$$\left\{\psi^{-1}\left(\lbrack T_1,\ldots, T_\ell\rbrack\right),\psi^{-1}\left(\lbrack T_2,T_3,\ldots, T_1\rbrack\right),\psi^{-1}\left(\lbrack T_\ell, T_1,\ldots, T_{\ell-1}\rbrack\right)\right\}$$
of $\Delta(n,2n)$.

Let us specialize still further.

In the case when all integers $s_i=1$, then after reversing the inequalities the polytope $\lbrack (12)_1,(34)_1,\ldots, (2n-1\ 2n)_1\rbrack$
is the same as the Catalan matroid polytope from \cite{Ardila2003,BidkoriSullivant}.  Indeed, again invoking Corollary \ref{cor: cyclic sum identity} below, we obtain Theorem 2.2 of \cite{BidkoriSullivant}, according to which the $n$ cyclically block-rotated matroid polytopes 
$$\lbrack(12)_1,(34)_1,\ldots, (2n-1\ 2n)\rbrack, \lbrack(34)_1,(56)_1,\ldots, (12)\rbrack,\ldots, \lbrack (2n-1\ 2n),(12)_1,\ldots,(2n-3\ 2n-2)\rbrack$$
have disjoint interiors and have union the whole hypersimplex $\Delta(n,2n)$.  In other words, they form a matroid subdivision of the hypersimplex $\Delta(n,2n)$.  

Now taking the union of the $\binom{n}{2}$ facets of the block rotations of $\lbrack 12_1 34_1\cdots (2n-1\ 2n)_1\rbrack$ which are interior to $\Delta(n,2n)$ and projecting the result with $\psi$, we obtain (a finite truncation of) the blade $((1,2,\ldots, n))$.
\begin{example}
	The tripod $((1,2,3))$ corresponds (after switching the directions of the inequalities) to the image under $\psi$ of the $4$-skeleton of the Catalan matroid subdivision from \cite{BidkoriSullivant}; that is, the tripod $((1,2,3))$ induces a subdivision of $\Delta(3,6)$ into three matroid polytopes
	$$\lbrack 12_1 34_1 56_1\rbrack,\lbrack 34_1 56_1 12_1 \rbrack,\lbrack 56_1 12_1 34_1 \rbrack.$$
\end{example}
	\begin{figure}[h!]
	\centering
	\includegraphics[width=0.4\linewidth]{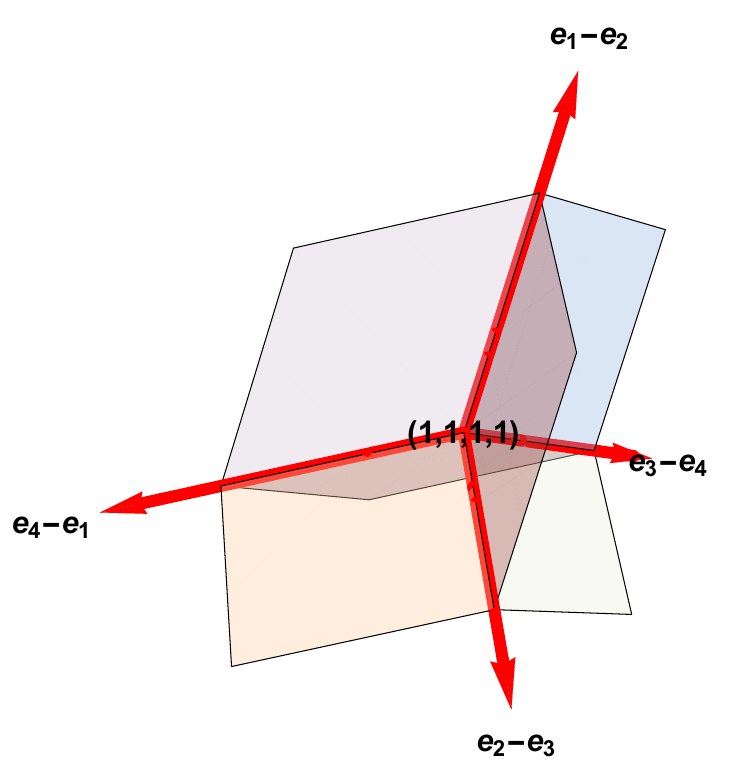}
	\includegraphics[width=0.4\linewidth]{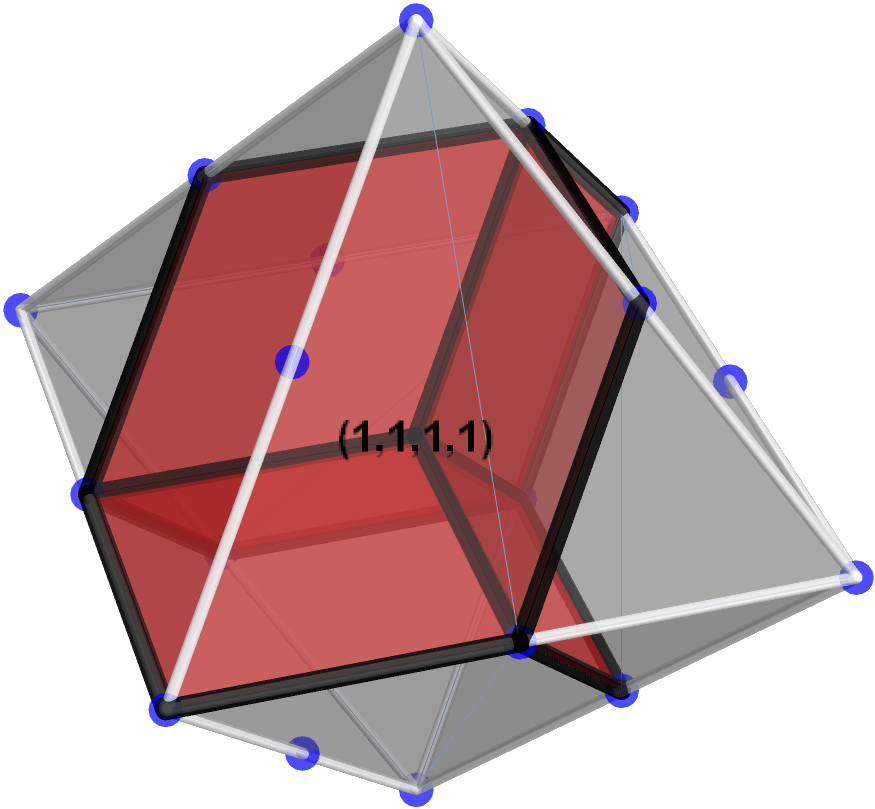}	
	\caption{For Example \ref{example: matroid subdivision 4 coords}: the blade $((1,2,3,4))$ (left) induces the matroid subdivision $((12_1 34_1 56_178_1))\in\text{OSP}(\Delta(4,8))$ (seen projected into the second dilation of the octahedron $\Delta(2,4)$) of the hypersimplex $\Delta(4,8)$  (right).}
	\label{fig:4-9-2019-octahedron-in-blade}
\end{figure}
\begin{example}\label{example: matroid subdivision 4 coords}
	The four chambers of the matroid subdivision induced in $\Delta(4,8)$ by the hypersimplicial blade $((12_1 34_1 56_1 78_1 ))$ are cut out by the facet inequalities respectively
	$$y_{12}\ge 1,\ y_{1234}\ge 1+1,\ y_{123456}\ge 1+1+1$$
	$$y_{34}\ge 1,\ y_{3456}\ge 1+1,\ y_{345678}\ge 1+1+1$$
	$$y_{56}\ge 1,\ y_{5678}\ge 1+1,\ y_{567812}\ge 1+1+1$$
	$$y_{78}\ge 1,\ y_{7812}\ge 1+1,\ y_{781234}\ge 1+1+1.$$
	These chambers project onto the four chambers in the subdivision $((1_1 2_1 3_1 4_1))$ of the second dilation of the octahedron $\left\{x\in \lbrack 0,2\rbrack^4: \sum_{i=1}^4 x_i=4 \right\}$ in Figure \ref{fig:4-9-2019-octahedron-in-blade}, with facet inequalities respectively
	
	$$x_{1}\ge 1,\ x_{12}\ge 1+1,\ x_{123}\ge 1+1+1$$
	$$x_{2}\ge 1,\ x_{23}\ge 1+1,\ x_{234}\ge 1+1+1$$
	$$x_{3}\ge 1,\ x_{34}\ge 1+1,\ x_{341}\ge 1+1+1$$
	$$x_{4}\ge 1,\ x_{41}\ge 1+1,\ x_{412}\ge 1+1+1.$$
	Here we have used different variables to emphasize that one set of polytopes is in $\Delta(4,8)$, while the other set is in the second dilation of the octahedron $\Delta(2,4)$.
\end{example}

\subsection{Blades in tropical geometry}\label{sec: bladesInTropicalGeometry}
Following \cite{Sturmfels Speyer Tropical intro}, a \textit{tropical polynomial function} $p:\mathbb{R}^n\rightarrow\mathbb{R}$ is the minimum of a finite set of linear functions,
$$p(x) = \text{min} \left\{f_1(x),f_2(x),\ldots, f_m(x)\right\},$$
whose graph $Z(p)$ is a piecewise-linear hypersurface in $\mathbb{R}^{n+1}$.  As such, the normal vector to $Z(p)$ changes direction only across a set of codimension 1 ``cracks'' in $\mathbb{R}^n$.  Translating from \cite{Sturmfels Speyer Tropical intro}, the collection of such cracks is called a \textit{tropical hypersurface} $\mathcal{H}_p$, consisting of points in $\mathbb{R}^n$ such that the minimum value is achieved by at least two of the linear forms $f_i$.

Here we show that the blade $((S_1,\ldots, S_k))\subset V_0^n$ is equal to the tropical hypersurface defined by the tropical polynomial
$$\text{min} \left\{L_1(x),\ldots, L_k(x)\right\},$$
where $L_i(x) = -(x_{S_i} +x_{S_i\cup S_{i+1}}+\cdots +x_{S_i\cup S_{i+1}\cup \cdots \cup S_{i-2}})$, with the convention that its indices are cyclic with $S_{i+k}=S_i$.  We remind that $x_{S_i\cup S_{i+1}\cup\cdots\cup S_{i-1}}=0$.
\begin{prop}
		The blade $((S_1,S_2,\ldots,S_k))$ is a tropical hypersurface in $V_0^n$.
\end{prop}

\begin{proof}
Supposing we are given a point $x\in V_0^n$ that is simultaneously a minimum of some $L_i$ and $L_j$, with value say $c_x$.  Taking without loss of generality $(1 = i)<j\le n$, then 
	$$c_x=-(x_{S_1} +x_{S_1\cup S_{2}}+\cdots +x_{S_1\cup S_{2}\cup \cdots \cup S_{k-1}})=-(x_{S_j} +x_{S_j\cup S_{j+1}}+\cdots +x_{S_j\cup S_{j+1}\cup \cdots \cup S_{j-2}}),$$
	or equivalently
	$$(k-(j-1))x_{S_1\cup \cdots \cup S_{j-1}} = (j-1)x_{S_j\cup\cdots \cup S_k}.$$
	But since in $V_0^n$ we have $x_{S_1\cup\cdots \cup S_{j-1}} + x_{S_{j}\cup\cdots \cup S_k} = \sum_{i=1}^n x_i=0$, it follows that
	$$(x_{S_1\cup S_{2}\cup\cdots \cup S_{j-1}})=0,\  \text{ and }\ (x_{S_{j}\cup S_{j+1}\cup\cdots \cup S_{k}})=0.$$
Now, supposing that $c_x = L_i(x) = L_j(x)$ is a minimum value, then for all $p \in\{1,\ldots, n\}\setminus\{i,j\}$,
\begin{eqnarray*}
	x_{S_p} +x_{S_p \cup S_{p+1}}+\cdots + x_{S_p \cup S_{p+1}\cup\cdots \cup S_{p -1}} & \le & -c_x.
\end{eqnarray*}
In particular, for each $p=1,2,\ldots, j-1$, 
\begin{eqnarray*}
(x_{S_1} +x_{S_1\cup S_{2}}+\cdots +x_{S_1\cup S_{2}\cup \cdots \cup S_{k-1}}) - (x_{S_p} +x_{S_p \cup S_{p+1}}\cdots + x_{S_p \cup S_{p+1}\cup\cdots \cup S_{p -2}}) & \ge & 0\\
 (k-(p-1))x_{S_1\cup \cdots \cup S_{p-1}} - (p-1)x_{S_{p}\cup\cdots \cup S_{k}}& \ge  & 0,
\end{eqnarray*}
and adding $(p-1)x_{S_1\cup\cdots \cup S_k}(\equiv 0)$ to both sides we find 
$$x_{S_1\cup\cdots \cup S_{p-1}}\ge 0.$$
Similarly we find $x_{S_j\cup\cdots \cup S_{q-1}}\ge 0$ for each $q=j,j+1,\ldots, k$.  Repeating for each $1\le i<j\le n$ and putting everything together, we recover the blade $((S_1,\ldots, S_k))$, as 
\begin{eqnarray*}
&&\bigcup_{1\le i<j\le k} \left\{x: x_{S_i}\ge 0,x_{S_i\cup S_{i+1}}\ge 0,\ldots, \ x_{S_j}\ge 0,x_{S_j\cup S_{j+1}}\ge 0,\ldots; \ x_{S_i\cup\cdots \cup S_{j-1}} = 0= x_{S_j\cup\cdots \cup S_{i-1}} \right\}\\
& = & \bigcup_{1\le i<j\le k} \lbrack S_i,S_{i+1},\ldots, S_{j-1}\rbrack \oplus \lbrack S_{j},S_{j+1},\ldots, S_{i-1}\rbrack\\
& = & \bigcup_{1\le i<j\le k} \lbrack S_1,S_{2}\rbrack \oplus \cdots  \lbrack \widehat{S_{i},S_{i+1}} \rbrack \oplus \cdots \oplus \widehat{\lbrack S_{j},S_{j+1\rbrack}}\oplus \cdots \oplus \lbrack S_k,S_1\rbrack,
\end{eqnarray*}
where we remind that here $\oplus$ denotes the Minkowski sum.  This last line is in agreement with Definition \ref{defn:blade}.
\end{proof}

\subsection{Discussion of results}

Our main technical results are Theorem \ref{thm: blade flag factorization} and Theorem \ref{thm: canonical basis cyclic sum plate}; we use the latter to introduce what we call the canonical\footnote{As opposed to the given basis of graduated functions, where all blades (except the whole space) are unions of cones of dimension $n-2$, the canonical basis is graded naturally by dimension.} basis for the graduated blades.  In Theorem \ref{thm: blade flag factorization}, we prove that the characteristic function $\Gamma_{S_1,S_2,\ldots, S_k}$ of the blade $((S_1,\ldots, S_k))$ with $k$ blocks factors as a convolution product of $k-2$ tripods, using the \textit{flag} factorization: $$\Gamma_{S_1,S_2,\ldots, S_k} = \Gamma_{S_1,S_2,S_3} \Gamma_{S_1,S_3,S_4}\cdots \Gamma_{S_1,S_{k-1},S_k}.$$
We deduce in Corollary \ref{cor: blade Minkowski sum decomposition} the set-theoretic identity for the Minkowski sum: we prove that any blade can be expressed (not uniquely) as a Minkowski sum of tripods and one-dimensional subspaces.  The proof is algebraic: we prove that characteristic functions of blades factorize with respect to the convolution product; however graduated functions of blades do not have the factorization property!

\begin{figure}[h!]
	\centering
	\includegraphics[width=0.5\linewidth]{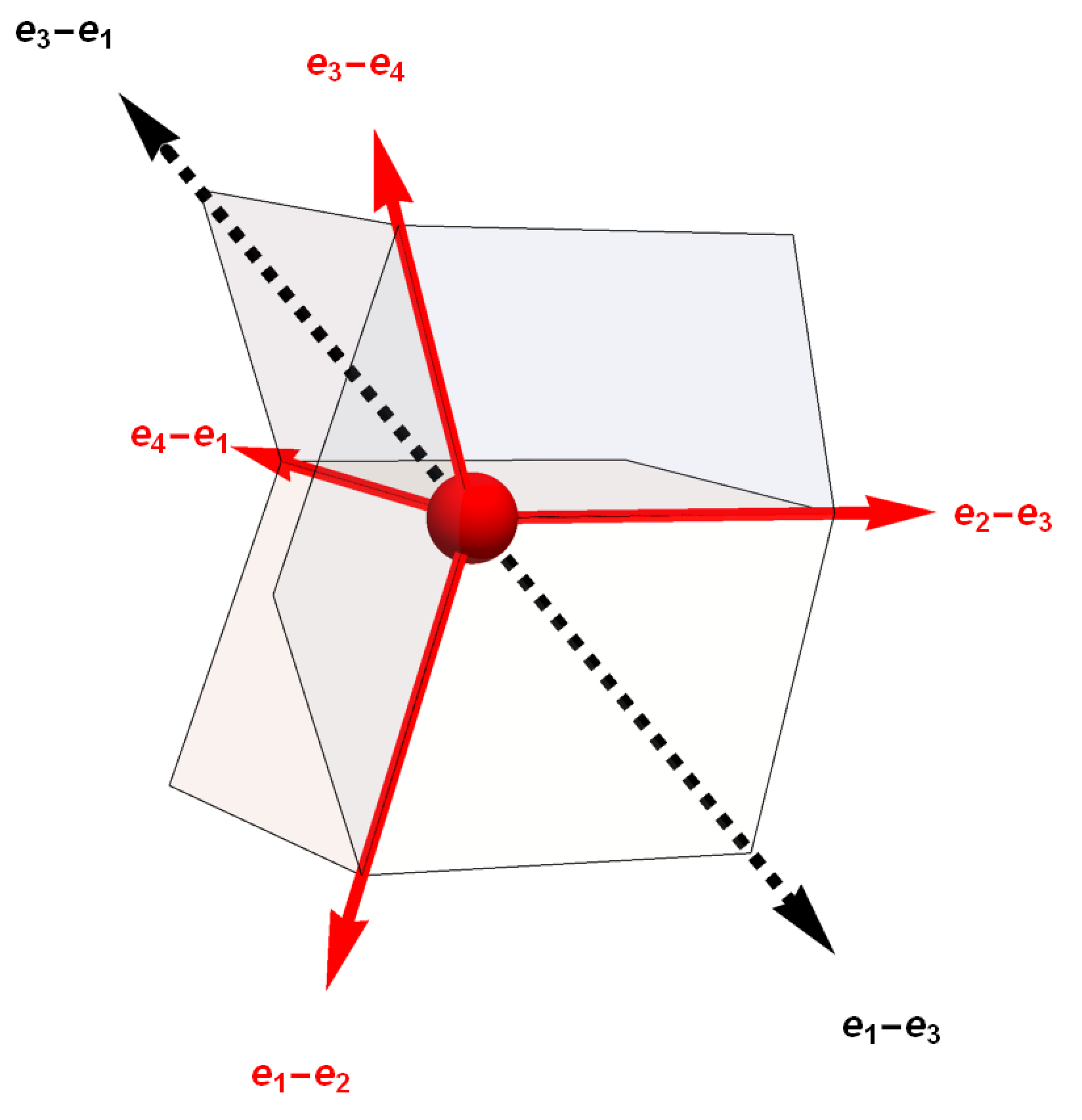}
	\caption{The $n=4$ blade $((1,2,3,4))$ as the Minkowski sum of two tripods $((1,2,3)),\ ((1,3,4))$.  The dotted black line has an apparent multiplicity in the computation of the Minkowski sum; however, upon inspection, for the convolution of characteristic functions of blades the line is smoothed away and has multiplicity 1.  This does not happen when the corresponding graduated functions of tripods are multiplied.
	}
\end{figure}

We prove that the factorization for characteristic functions of is independent of the triangulation; this shows, for example, that factorizations of $((1,2,\ldots, n))$ as Minkowski sums of tripods are in bijection with the set of triangulations of a cyclically oriented $n$-gon with vertex labels $1,2,\ldots n$.  

Theorem \ref{thm: canonical basis cyclic sum plate} derives the expansion of an element in the canonical (candidate) basis as a signed sum of \textit{graduated} functions of blades.  In Theorem \ref{thm: cyclic sum canonical basis} we remove the word ``candidate'' by showing that each can be obtained from an upper-unitriangular transformation from a set of graduated blades which according to Proposition \ref{prop: linear independence cyclic plate sum} is linearly independent.

Remark that there may be other interesting bases for the space spanned by characteristic functions of root cones $\lbrack S_1,\ldots S_k\rbrack$ which are labeled not by ordered set partitions, but by more general directed graphs with nodes $\{S_1,\ldots, S_k\}$; for example, there is the \textit{forkless} monomial basis from \cite{Grinberg} for the so-called subdivision algebra \cite{Meszaros}.  Forkless monomials are in bijection with directed forests $\{(i_1,j_1),\ldots, (i_m,j_m)\}$ with $i_a<j_a$, having the property that no node has two outgoing branches: $i_a\not=i_b$ whenever $a\not=b$.

In Section \ref{section: canonical basis enumeration}, Proposition \ref{prop: graded dimension blades}, we enumerate the canonical bases as $n$ increases, and then ask about possible geometric interpretations of the coefficients of the generating functions for the diagonals.  Indeed, by way of a simple counting argument we have for the enumeration of the canonical basis of graduated functions of blades,
$$T_{n,k}=\sum_{i=1}^n S(n,i)s(i-1,k-1),$$
where $S(n,i)$ is the Stirling number of the second kind, and $s(i,k)$ is the (unsigned) Stirling number of the first kind. 

For $n=1,2,\ldots, 6$ we have
$$\begin{array}{cccccc}
1 & \text{} & \text{} & \text{} & \text{} & \text{} \\
1 & 1 & \text{} & \text{} & \text{} & \text{} \\
1 & 4 & 1 & \text{} & \text{} & \text{} \\
1 & 15 & 9 & 1 & \text{} & \text{} \\
1 & 66 & 66 & 16 & 1 & \text{} \\
1 & 365 & 500 & 190 & 25 & 1 \\
\end{array}$$
Note that the rows sum (as they should) to the necklace numbers, which count ordered set partitions up to cyclic block rotation.

The generating function numerators for its diagonals appear to be symmetric and unimodal, and to have coefficients that sum to $(2n)!/n!$, O.E.I.S. A001813, \cite{oeis}.  The coefficients themselves agree with A142459 for the first four rows, and their sums agree, $(2n)!/n!$, but for $n\ge 5$ the coefficients are different.  See Section \ref{section: canonical basis enumeration} below for more details, as well as a conjecture.

Finally, it turns out that there is an interesting second grading on cyclic ordered set partitions which is in a sense complementary to the grading coming from the canonical blade basis.  Grouping cyclic ordered set partitions by the number of blocks one obtains the sequence given by O.E.I.S. A028246 \cite{oeis}:
$$a(n, k) = S(n, k)*(k-1)!.$$
This gives for $n=1,2,\ldots, 6$, 
$$\begin{array}{cccccc}
1 & \text{} & \text{} & \text{} & \text{} & \text{} \\
1 & 1 & \text{} & \text{} & \text{} & \text{} \\
1 & 3 & 2 & \text{} & \text{} & \text{} \\
1 & 7 & 12 & 6 & \text{} & \text{} \\
1 & 15 & 50 & 60 & 24 & \text{} \\
1 & 31 & 180 & 390 & 360 & 120 \\
\end{array}$$
We shall encounter this enumeration quite naturally in Section \ref{sec: configuration space circle}, where it is observed to count the set of cyclic configurations of $n$ colliding labeled particles on a circle.

The entries of this number triangle also appeared in connection with a certain word Hopf algebra, counting numbers of packed words of length $n$ and supremum $k$, see Remark 4.2 in \cite{Duchamp}.

\subsection{Additional Connections}

We relate in Appendix \ref{sec: combinatorial scattering equations} the linear relations which characterize the subspace spanned by characteristic functions of blades in dimension $1$ to a balancedness condition on weighted directed graphs, as studied by T. Lam and A. Postnikov in their work on polypositroids \cite{LamPostnikov}.  By comparison with the results of \cite{Berget}, blades are also related to zonotopal algebras \cite{ArdilaPostnikov,Berget,HoltzAmos}, through Gale duality for the type $A$ reflection arrangement.  

In terms of representation theory of the symmetric group, modules spanned by characteristic functions of blades are constructed plethystically using the higher Eulerian representations, which arise as restrictions to $\symm_n$ of representations of $\symm_{n+1}$ of dimensions the Stirling numbers of the first kind, as studied by S. Whitehouse in \cite{WhitehouseInduced}.  For further historical discussion we refer to the introduction  of \cite{EarlyReiner}, where together with V. Reiner we studied the cohomology ring, denoted there $\mathcal{V}^n$, of the configuration space of $n$ points in $SU(2)$, modulo the diagonal action of $SU(2)$.  In fact, it turns out that this cohomology ring is related to (a degeneration of) the Minkowski algebra of blades.  The discovery of this geometric connection to blades helped to determine the presentation of  $\mathcal{V}^n$ in \cite{EarlyReiner}.  See also \cite{Brauner}.

A key observation is that as these $n$ (closed) cones can be seen to cover the whole ambient space (see Corollary \ref{cor: cyclic sum identity}), coinciding only on common facets, they form what is called a complete fan.  In the case of the blade $((1,2,\ldots, n))$ it can be obtained as the normal fan to the simplex with vertices $e_1-e_n,e_n-e_{n-1},\ldots, e_2-e_1$.  

In \cite{EarlyCanonicalBasis} and here we study spaces of characteristic functions of plates and blades \textit{modulo} higher codimension faces; it is tempting to ask if there is there a compatible Lie algebra in the direction of \cite{Racinet}, see also \cite{FrancisBrown}, of some rational functions modulo products.  Such products would correspond in our setting to higher codimension faces of cones.

We point out an interesting apparent connection between blades and certain elliptic functions which appear in string theory \cite{MafraSchlotterer}, where the blade relations are analogous to certain limits of the so-called Fay identities\footnote{We thank Oliver Schlotterer for this observation.}.  We would like also to mention another connection, with the related work \cite{MafraSchlotterer2014}.  Here, modulo the Fay identities, the  \textit{canonicalization} relations hold among pseudoinvariants.  These will correspond in our setting to straightening relations into the canonical basis for graduated functions of blades.

The beauty and utility of the exponential map used for the cohomology ring of the configuration space of points in $SU(2)$, in Section \ref{sec: graded cohomology ring}, suggests a connection to some of the combinatorial properties of certain exponential generating functions from \cite{BMNS}.  Finally, the examples and discussion in Appendix \ref{sec:Leading singularity} suggest a new class of identifications for non-planar on-shell diagrams beyond the well-known square move which deserves further attention; see \cite{DeltaAlgebra2019}, where the study was initiated.

	\section{Permutohedral Plates}
	Let us fix some notation.

We denote $x_S = \sum_{i\in S} x_i$ for any proper nonempty subset $S$ of $\{1,\ldots, n\}$, and abbreviate $x_{S_1\cup \cdots \cup S_k}$ as $x_{S_1\cdots S_k}$ for $k$ disjoint subsets $S_1,\ldots, S_k$ of $\{1,\ldots, n\}$.  Set $V_0^n = \{x\in \mathbb{R}^n: \sum_{i=1}^nx_i=0\}$.  Let  $e_1,\ldots, e_n$ the usual basis for $\mathbb{R}^n$.  With $J$ a nonempty subset of $\{1,\ldots, n\}$, set $e_J=\sum_{j\in J} e_j$, and let $\bar{e}_J = \sum_{j\in J} e_j - \frac{\vert J\vert }{n}(\sum_{i=1}^n e_i)$, be the projection of $e_J$ onto the plane $V_0^n$.  Denote by $\pi_1 \oplus \pi_2 = \{u+v:u\in \pi_1,\ v\in \pi_2\}$ the Minkowski sum of the polyhedra $\pi_1,\pi_2\subseteq V_0^n$.  Put
$$\lbrack S_i,S_j\rbrack = \left\{\sum_{a\in S_i\cup S_j} t_a\bar{e}_a\in V_0^n: t_{S_i}\ge 0\right\},$$
that is the upper half of the subspace spanned by $\{\bar{e}_a:a\in S_i\cup S_j\}$.

A polyhedral cone is a subset of some $\mathbb{R}^n$ which is cut out by a set of inequalities of the form $\{x\in V_0^n: \sum_{j=1}^n a_{ij} x_j\ge 0,\ i=1,\ldots, m\}$ for a given coefficient matrix $(a_{ij})\in\mathbb{R}^{m\times n}$.  For a polyhedral cone $\pi$, denote by $\lbrack \pi\rbrack $ its characteristic function, that is, it evaluates to $\lbrack \pi\rbrack(x) = 1$ if $x\in\pi$ and otherwise $\lbrack \pi\rbrack(x) = 0$.  

Recall that an ordered set partition (OSP) is an ordered list $\mathbf{S}=(S_1,\ldots, S_k)$ of disjoint subsets of $\{1,\ldots, n\}$ such that $\bigsqcup_{i=1}^k S_i = \{1,\ldots, n\}$.

\begin{defn}\label{defn: permutohedral cone}
	A polyhedral cone $\Pi$ in $V_0^n$ is \textit{permutohedral} if it can expressed as a Minkowski sum of subspaces $\lbrack T_i\rbrack$ and a sequence of half subspaces $\lbrack S_{a},S_{{a+1}}\rbrack$,
	$$\Pi=\lbrack T_1\rbrack \oplus \cdots \oplus  \lbrack T_\ell\rbrack\oplus \lbrack\mathbf{S}_1\rbrack \oplus \cdots \oplus \lbrack \mathbf{S}_k\rbrack,$$
	for disjoint subsets $T_1,\ldots, T_\ell$ and ordered set partitions $\mathbf{S}_1,\ldots,\mathbf{S}_k$ of respectively disjoint subsets $U_1,\ldots, U_k$ of $\{1,\ldots, n\}\setminus \left(\bigcup_{i=1}^\ell T_i\right)$.  In the terminology from bel\textit{}ow, the cone $\Pi$ is a Minkowski sum of \textit{tangent cones} to the regular permutohedron.
	
	Further, $\Pi$ is a \textit{generalized permuhedral} cone if it is a Minkowski sum 
	$$\Pi=\lbrack T_1\rbrack \oplus \cdots \oplus  \lbrack T_\ell\rbrack\oplus \lbrack S_{i_1},S_{j_1}\rbrack \oplus \cdots \oplus \lbrack S_{i_k},S_{j_k}\rbrack,$$
	where $S_{i_a}\cap S_{j_a} = \emptyset$ for $a=1,\ldots, k$ and where $\{T_1,\ldots, T_\ell\}$ is a set partition of the complement in $\{1,\ldots, n\}$ of the union of all subsets $S_{i_a}, S_{j_a}$,
	$$\{1,\ldots, n\} \setminus \left(\bigcup_{a=1}^k S_{i_a}\cup S_{j_a}\right).$$
\end{defn}
In Definition \ref{defn: permutohedral cone}, without loss of generality we could further require that the Minkowski sum be \textit{reduced}, that is, we have the property that for each pair $\{S_{\ell},S_{m}\}$, the sets $S_{\ell},S_{m}$ are either disjoint or equal.  For without this assumption, we would have the simplification for example 
$$\lbrack 135,246\rbrack \oplus \lbrack 127,345\rbrack = \lbrack 1234567\rbrack.$$
For detailed examples of generalized permutohedral cones and their functional representations in this context, see Appendix A.2 in \cite{EarlyCanonicalBasis}.

Recall that the tangent cone $C$ to a face $F$ of a polyhedron $P$ is defined by
$$C = \{x+\lambda(y-x):x\in F,y\in P, \text{ and }\lambda\ge0\}.$$

We single out the set of tangent cones to faces of the regular permutohedron, which were studied as \textit{plates} by Ocneanu.  See \cite{EarlyCanonicalBasis} for details of our construction and related results for permutohedral cones, as well as a canonical basis for the space spanned by characteristic functions of such tangent cones.  Let $(S_1,\ldots, S_k)$ be an ordered set partition of $\{1,\ldots, n\}$. Recall the notation $\lbrack S_1,\ldots, S_k\rbrack$ for the plate, which is the cone in $V_0^n$ cut out by the system of facet inequalities
\begin{eqnarray*}
	x_{S_1} & \ge & 0\\
	x_{S_1S_2} & \ge & 0\\
	& \vdots &\\	
	x_{S_1S_2\cdots S_{k-1}} & \ge & 0\\
	\sum_{i=1}^n x_i& = & 0.
\end{eqnarray*}
This is a permutohedral cone, since it can be expressed as the Minkowski sum of half spaces
$$\lbrack S_1,S_2\rbrack \oplus \lbrack S_2,S_3\rbrack \oplus\cdots\oplus \lbrack S_{k-1},S_k\rbrack.$$

It is a standard result that this is the tangent cone to the face of the usual permutohedron labeled by the ordered set partition $(S_1,\ldots, S_k)$ of $\{1,\ldots, n\}$, see for example \cite{PostnikovPermutohedra}.  In \cite{EarlyCanonicalBasis} the same was proved for tree graphs.

\begin{prop}
	Let $\Pi$ be a generalized permutohedral cone which is labeled by a directed tree graph with edge set $\{(i_1,j_1),\ldots, (i_{k},j_k)\}$ for some $k\le $n, with Minkowski sum decomposition
	$$\Pi  = \lbrack S_{i_1},S_{j_1}\rbrack \oplus \cdots \oplus \lbrack S_{i_k},S_{j_k}\rbrack.$$
	Then it can be expanded as a signed sum of characteristic functions of tangent cones to faces of the usual permutohedron.
\end{prop}
In fact the expansion result can be extended to the expansion of generalized permutohedral cones as signed sums of characteristic functions of permutohedral cones.

Recall from \cite{EarlyCanonicalBasis} the notation $\hat{\mathcal{P}}^n$ for the $\mathbb{K}$-linear span of all characteristic functions $\lbrack\lbrack S_1,\ldots, S_k\rbrack\rbrack$ of plates $\lbrack S_1,\ldots, S_k\rbrack$, as $(S_1,\ldots, S_k)$ ranges over all ordered set partitions of $\{1,\ldots, n\}$.  Here we may take $\mathbb{K} = \mathbb{Q}$ or $\mathbb{K} = \mathbb{C}$.

\begin{defn}\label{defn: duality for cones}
	Let $C$ be a polyhedral cone in $V_0^n$.  The dual cone to $C$, denoted $C^\star$, is defined by the equation 
	$$C^\star=\{y\in V_0^n:y\cdot x\ge 0\text{ for all } x\in C\}.$$
\end{defn}

In \cite{EarlyCanonicalBasis} it was an essential property that the plate $\lbrack S_1,\ldots, S_k\rbrack$ is \textit{dual} (in the sense of convex geometry, see \cite{BarvinokPommersheim}) to the face of the arrangement of type $A_{n-1}$ reflection hyperplanes, given by
$$\{x\in V_0^n: x_{(S_1)}\ge \cdots\ge x_{(S_k)} \},$$
where $x_{(S)}$ is shorthand notation for $x_{i_1}=\cdots =x_{i_{\vert S\vert}}$, given that $S=\{i_1,\ldots, i_{\vert S\vert}\}$

Let us collect two important results from \cite{BarvinokPommersheim}.  It is immediate from Definition \ref{defn: permutohedral cone} that the conical hull of two permutohedral cones is a permutohedral cone; therefore in Theorem \ref{thm: valuation and convolution} the convolution product $\bullet $ in (2) is defined on the subring of characteristic functions of permutohedral cones, $\hat{\mathcal{P}}^n$.

\begin{defn}\label{defn: valuation}
	A linear transformation $\hat{\mathcal{P}}^n\rightarrow V$, where $V$ is a vector space, is called a \textit{valuation}.  
\end{defn}

We shall need Theorem 2.5 of \cite{BarvinokPommersheim}, which defines the convolution of two polyhedral cones with respect to the Euler characteristic, to prove relations with respect to the convolution product $\bullet$.
\begin{thm}[\cite{BarvinokPommersheim}, Theorem 2.5] \label{thm: valuation and convolution}
		There is a unique bilinear operation, denoted here $\bullet:\hat{\mathcal{P}}^n\times \hat{\mathcal{P}}^n\rightarrow \hat{\mathcal{P}}^n$, called convolution, such that 
		$$\lbrack \pi_1\rbrack \bullet \lbrack \pi_2\rbrack = \lbrack \pi_1\oplus \pi_2\rbrack$$
		for any two permutohedral cones $\pi_1,\pi_2$.
\end{thm}

The proof given in \cite{BarvinokPommersheim} of Theorem \ref{thm: valuation and convolution} relies a linear map called the \textit{Euler characteristic}, which is proven there to be the unique valuation 
$\mu:\hat{\mathcal{P}}^n\rightarrow \mathbb{Q}$, such that $\mu(\lbrack\pi\rbrack)=1$ for any nonempty polyhedron $\pi$.

Finally, we specialize a result from \cite{BarvinokPommersheim}, which is a statement about the wider class of all polyhedral cones, to permutohedral cones.

\begin{thm}[\cite{BarvinokPommersheim}, Theorem 2.7]\label{thm: duality operator}
	There exists a valuation $\mathcal{D}$ on the space of characteristic functions of polyhedral cones such that 
	$$\mathcal{D}(\lbrack \pi\rbrack) = \lbrack \pi^\star\rbrack.$$
\end{thm}

We shall usually abuse notation and use $\star$ for the valuation $\mathcal{D}$, so the relation of Theorem \ref{thm: duality operator} becomes
$$\lbrack \pi\rbrack^\star = \lbrack \pi^\star\rbrack.$$

The essential property here is that $\star$ interchanges convolution with pointwise product:
$$(\lbrack C_1\rbrack \bullet \lbrack C_2\rbrack \bullet \cdots\bullet \lbrack C_\ell\rbrack)^\star =(\lbrack C_1\oplus \cdots\oplus C_\ell \rbrack)^\star = \lbrack C_1^\star\cap\cdots \cap  C_\ell^\star\rbrack= \lbrack C_1 \rbrack^\star \cdot \lbrack C_2\rbrack^\star \cdots \lbrack C_\ell\rbrack^\star,$$
where $\bullet$ is convolution and $\cdot$ is the pointwise product of characteristic functions.

We have the following easy but important identity.
\begin{prop}\label{prop:cyclic Minkowski sum identity closed plates}
	For any ordered set partition $(S_1,S_2,\ldots, S_k)$ of $\{1,\ldots, n\}$, we have
	$$\lbrack\lbrack S_1,S_2\rbrack\rbrack \bullet \lbrack\lbrack S_2,S_3\rbrack\rbrack \bullet \cdots \bullet \lbrack\lbrack S_k,S_1\rbrack\rbrack=\lbrack\lbrack S_1\cup \cdots\cup S_k\rbrack\rbrack.$$
\end{prop}

\begin{proof}
It is informative to check this by dualizing.  Then, applying $\star$ gives $\lbrack S_i,S_{i+1}\rbrack^\star = \{x\in V_0^n: x_{(S_i)}\ge x_{(S_{i+1})}\}$, where $x_{(S)}$ is shorthand for $x_{i_1}=\cdots  = x_{i_{\ell}}$ if $S=\{i_1,\ldots, i_\ell\}$.  As $\star$ interchanges convolution and pointwise product of characteristic functions, we have
\begin{eqnarray*}
	\left(\lbrack\lbrack S_1,S_2\rbrack\rbrack \bullet \lbrack\lbrack S_2,S_3\rbrack\rbrack\bullet \cdots \bullet \lbrack\lbrack S_k,S_1\rbrack\rbrack\right)^\star & = & \left\{x\in V_0^n: x_{(S_1)}\ge x_{(S_2)}\ge \cdots \ge x_{(S_k)} \ge x_{(S_1)} \right\}\\
	& = & \lbrack \{(0,\ldots, 0)\}\rbrack\\
	& = & \lbrack \lbrack S_1\cup\cdots\cup S_k\rbrack\rbrack^\star,
\end{eqnarray*}
and since $\star$ is an involution for closed convex cones, applying $\star$ again completes the proof.	
\end{proof}

\section{Open plates}

For any ordered pair of disjoint subsets $(S_1,S_2)$ of $\{1,\ldots, n\}$, define 
$$\mu_{S_1,S_2} = \lbrack \lbrack S_1,S_2\rbrack\rbrack - \lbrack \lbrack S_1\rbrack\rbrack \bullet \lbrack\lbrack S_2\rbrack\rbrack,$$
the characteristic function of the set
$$\left\{\sum_{i\in S_1\cup S_2}t_i e_i\in V_0^n: t_{S_1}>0  \right\},$$
where the inequality is strict.  It will be understood that unless indicated otherwise, all products of $\mu$'s are convolution, using the Euler characteristic as the measure, see Theorem \ref{thm: valuation and convolution}.

Further denote by $1_{S}$ the characteristic function of the subspace
$$ \lbrack S\rbrack = \left\{\sum_{i\in S}t_i e_i\in V_0^n: t_i\in\mathbb{R}\right\}.$$
Then for any disjoint nonempty subsets $S_1,\ldots, S_\ell$ of $\{1,\ldots, n\}$, $1_{S_1}1_{S_2}\cdots 1_{S_\ell}$ is the characteristic function of the subspace 
$$\{x\in V_0^n: x_{S_i}=0 \text{ and } x_j=0 \text{ whenever } j\not\in S_1\cup\cdots\cup S_\ell \}.$$

\begin{prop}\label{prop: open plates properties}
	We have
	$$\mu_{S_1,S_2} ^2 = -\mu_{S_1,S_2},\ \text{and }\ \mu_{S_1,S_2} \mu_{S_2,S_1} = 0.$$
	More generally, we have the cycle identities
	$$\mu_{S_1,S_2}\mu_{S_2,S_3} \cdots \mu_{S_{k-1},S_k} \mu_{S_k,S_1} = 0$$
	for any ordered set partition $(S_1,\ldots, S_k)$ of a subset of $\{1,\ldots, n\}$.
\end{prop}

\begin{proof}

	We supply the algebraic proofs for the identities $\mu_{S_1,S_2}^2 = -\mu_{S_1,S_2}$ and $\mu_{S_1,S_2}\mu_{S_2,S_1} = 0$.  We have
	\begin{eqnarray*}
		\mu_{S_1,S_2}^2 & = & (\lbrack\lbrack S_1,S_2\rbrack\rbrack -\lbrack\lbrack S_1\rbrack\rbrack \bullet \lbrack\lbrack S_2\rbrack\rbrack)^2 = (\lbrack\lbrack S_1,S_2\rbrack\rbrack^2 -2\lbrack\lbrack S_1,S_2\rbrack\rbrack +\lbrack\lbrack S_1\rbrack\rbrack^2 \bullet \lbrack\lbrack S_2\rbrack\rbrack^2)\\
		& = & (\lbrack\lbrack S_1,S_2\rbrack\rbrack -2\lbrack\lbrack S_1,S_2\rbrack\rbrack +\lbrack\lbrack S_1\rbrack\rbrack \bullet \lbrack\lbrack S_2\rbrack\rbrack) = -(\lbrack\lbrack S_1,S_2\rbrack\rbrack - \lbrack\lbrack S_1\rbrack\rbrack \bullet \lbrack\lbrack S_2\rbrack\rbrack)\\
		& = & -\mu_{S_1,S_2},
	\end{eqnarray*}
	and
	\begin{eqnarray*}
		\mu_{S_1,S_2}\mu_{S_2,S_1} & = & (\lbrack\lbrack S_1,S_2\rbrack\rbrack -\lbrack\lbrack S_1\rbrack\rbrack \bullet \lbrack\lbrack S_2\rbrack\rbrack)(\lbrack\lbrack S_2,S_1\rbrack\rbrack -\lbrack\lbrack S_1\rbrack\rbrack \bullet \lbrack\lbrack S_2\rbrack\rbrack) \\
		& = & \lbrack\lbrack S_1,S_2\rbrack\rbrack \bullet \lbrack\lbrack S_2,S_1\rbrack\rbrack - \left(\lbrack\lbrack S_1,S_2\rbrack\rbrack + \lbrack\lbrack S_2,S_1\rbrack\rbrack\right) + \lbrack\lbrack S_1\rbrack\rbrack \bullet \lbrack\lbrack S_2\rbrack\rbrack\\
		& = & \lbrack\lbrack S_1 \cup S_2\rbrack\rbrack -\left(\lbrack\lbrack S_1\cup S_2\rbrack\rbrack + \lbrack\lbrack S_1\rbrack\rbrack \bullet \lbrack\lbrack S_2\rbrack\rbrack\right) + \lbrack\lbrack S_1\rbrack\rbrack \bullet \lbrack\lbrack S_2\rbrack\rbrack\\
		& = & 0.
	\end{eqnarray*}
	The cycle identities can be seen in general through an application of duality for polyhedral cones.  Then, as the dual $\mu_{S_i,S_j}^\star$'s are open half-spaces through the origin in $V_0^n$, and the intersection of the supports of the characteristic functions of the duals $\mu_{S_1,S_2}^\star ,\ldots, \mu_{S_{k},S_1}^\star $'s is empty, their (pointwise) product is identically zero.  Further, as $\star$ is an involution on the characteristic functions  $\lbrack\lbrack S_1,S_2\rbrack\rbrack$ and $1_S$, we have
	\begin{eqnarray*}
		(\mu_{S_1,S_2}^\star)^\star & = & (\lbrack\lbrack S_1,S_2\rbrack\rbrack^\star - 1_{S_1}^\star1_{S_2}^\star)^\star \\
		& = & (\lbrack\lbrack S_1,S_2\rbrack\rbrack^\star)^\star - (1_{S_1}^\star)^\star(1_{S_2}^\star)^\star\\
		& = &\lbrack\lbrack S_1,S_2\rbrack\rbrack - 1_{S_1}1_{S_2}\\
		& = & \mu_{S_1,S_2}.
	\end{eqnarray*}
	It follows that
	\begin{eqnarray*}
	\mu_{S_1,S_2}\cdots\mu_{S_{k},S_1} & = & (\mu_{S_1,S_2}^\star \cdots\mu_{S_{k},S_1}^\star)^\star\\
	& = & 0.
	\end{eqnarray*}
\end{proof}
	Denote by $e_j(x_1,\ldots, x_k)$ the $j$th elementary symmetric function in the variables $x_1,\ldots, x_k$, obtained from the generating function
	$$(1+tx_1)(1+tx_2)\cdots (1+tx_k) = \sum_{j=0}^k e_j(x_1,\ldots, x_k)t^j.$$
	\begin{cor}\label{cor: alternating expansion graduated blade}
		We have
		$$\lbrack(S_1,\ldots, S_k)\rbrack = \lbrack\lbrack S_1\cup \cdots\cup S_k\rbrack\rbrack  + \sum_{j=0}^{k-1} (-1)^{(k-2)-j} e_j(\lbrack\lbrack S_1,S_2\rbrack\rbrack , \lbrack\lbrack S_2,S_3\rbrack\rbrack,\ldots, \lbrack\lbrack S_k,S_1\rbrack\rbrack).$$
	\end{cor}
	
	\begin{proof}
		By Proposition \ref{prop: open plates properties}, we have the cycle identity $\mu_{12}\mu_{23}\cdots\mu_{k1}=0$, hence
		\begin{eqnarray*}
			0 & = & \mu_{S_1S_2}\mu_{S_2S_3}\cdots \mu_{S_kS_1}\\
			& = & (\lbrack\lbrack  S_1,S_2\rbrack\rbrack - 1)(\lbrack\lbrack  S_2,S_3\rbrack\rbrack - 1)\cdots (\lbrack\lbrack  S_k,S_1\rbrack\rbrack - 1),
		\end{eqnarray*}
		which expands to an alternating sum of the elementary symmetric functions in the variables 
		$$\lbrack\lbrack S_1,S_2\rbrack\rbrack, \ldots, \lbrack\lbrack S_k,S_1\rbrack\rbrack.$$
		Noting that the piecewise constant function $\lbrack(S_1,\ldots, S_k)\rbrack$ can be expressed as an elementary symmetric function of degree $k-1$, as 
		$$\lbrack (S_1,\ldots, S_k)\rbrack = e_{k-1}(\lbrack\lbrack S_1,S_2\rbrack\rbrack,\ldots, \lbrack\lbrack S_k,S_1\rbrack\rbrack),$$ we conclude the proof by solving for $e_{k-1}(\lbrack\lbrack S_1,S_2\rbrack\rbrack,\ldots, \lbrack\lbrack S_k,S_1\rbrack\rbrack)$. 
	\end{proof}

In particular, for any three disjoint nonempty subsets $S_1,S_2,S_3\subset\{1,\ldots, n\}$ we have the fundamental relation 
$$\lbrack\lbrack S_1,S_2,S_3\rbrack\rbrack + \lbrack\lbrack S_2,S_3,S_1\rbrack\rbrack + \lbrack\lbrack S_3,S_1,S_2\rbrack\rbrack = 1_{S_1\cup S_2\cup S_3} + \mu_{S_1,S_2}+ \mu_{S_2,S_3} + \mu_{S_3,S_1}-1_{S_1}1_{S_2}1_{S_3} .$$
\begin{example}\label{example: rational function representation blade}
	In the cases $n=3,4$ we have the functional representations of characteristic functions of plates, respectively
	\begin{eqnarray*}
		\lbrack\lbrack 1,2,3\rbrack\rbrack & \mapsto & \frac{1}{x_1 \left(x_1+x_2\right) \left(x_1+x_2+x_3\right)}\\
	\lbrack\lbrack 1,2,3,4\rbrack\rbrack	&\mapsto & \frac{1}{x_1 \left(x_1+x_2\right) \left(x_1+x_2+x_3\right) \left(x_1+x_2+x_3+x_4\right)}
	\end{eqnarray*}
	while for blades we obtain, after partial fraction identities,
	\begin{eqnarray*}
	\lbrack (1,2,3)\rbrack & = & \lbrack\lbrack 1,2,3\rbrack\rbrack + \lbrack\lbrack 2,3,1\rbrack\rbrack + \lbrack\lbrack 3,1,2\rbrack\rbrack\\
	& \mapsto  & \frac{1}{y_1 \left(y_1+y_2\right) \left(y_1+y_2+y_3\right)}+\frac{1}{y_2 \left(y_2+y_3\right) \left(y_1+y_2+y_3\right)}+\frac{1}{y_3 \left(y_1+y_3\right) \left(y_1+y_2+y_3\right)}\\
	& = & \frac{1}{y_1 \left(y_1+y_2\right) y_3} +\frac{1}{y_2 \left(y_2+y_3\right)y_1}+ \frac{1}{y_3 \left(y_3+y_1\right)y_2}-\frac{1}{y_1y_2y_3},
\end{eqnarray*}
and
	\begin{eqnarray*}
	&&\lbrack (1,2,3,4)\rbrack = \lbrack\lbrack 1,2,3,4\rbrack\rbrack + \lbrack\lbrack 2,3,4,1\rbrack\rbrack +\lbrack\lbrack 3,4,1,2\rbrack\rbrack +\lbrack\lbrack 4,1,2,3\rbrack\rbrack \\
	&\mapsto &\frac{1}{y_1 \left(y_1+y_2\right) \left(y_1+y_2+y_3\right) \left(y_1+y_2+y_3+y_4\right)}+\frac{1}{y_2 \left(y_2+y_3\right) \left(y_2+y_3+y_4\right) \left(y_1+y_2+y_3+y_4\right)}\\
	&+ & \frac{1}{y_3 \left(y_3+y_4\right) \left(y_1+y_3+y_4\right) \left(y_1+y_2+y_3+y_4\right)} + \frac{1}{y_4 \left(y_1+y_4\right) \left(y_1+y_2+y_4\right) \left(y_1+y_2+y_3+y_4\right)}\\
	&=&\frac{1}{y_1 \left(y_1+y_2\right) \left(y_1+y_2+y_3\right) y_4} + \frac{1}{y_2 y_3 \left(y_3+y_4\right) \left(y_1+y_3+y_4\right)} + \frac{1}{y_3 y_4 \left(y_1+y_4\right) \left(y_1+y_2+y_4\right)}\\
	&+&\frac{1}{y_3 y_4 \left(y_1+y_4\right) \left(y_1+y_2+y_4\right)} + \frac{1}{y_1 \left(y_1+y_2\right) y_3 \left(y_3+y_4\right)} + \frac{1}{y_2 \left(y_2+y_3\right) y_4 \left(y_1+y_4\right)}\\
	&-&\frac{1}{y_1 \left(y_1+y_2\right) y_3 y_4}-\frac{1}{y_1 y_2 \left(y_2+y_3\right) y_4} - \frac{1}{y_1 y_2 y_3 \left(y_3+y_4\right)}-\frac{1}{y_2 y_3 y_4 \left(y_1+y_4\right)}\\
	&+&\frac{1}{y_1 y_2 y_3 y_4}
\end{eqnarray*}
\end{example}

	\begin{cor}\label{cor: cyclic sum identity}
		Let $(S_1,\ldots, S_k)$ be an ordered set partition.  Then, modulo characteristic functions of cones of codimension $\ge 1$, we have the cyclic sum relation
		$$\lbrack\lbrack S_1,S_2,\ldots, S_k\rbrack\rbrack + \lbrack\lbrack S_2,S_3,\ldots, S_1\rbrack\rbrack + \cdots + \lbrack\lbrack S_k,S_1,\ldots, S_{k-1}\rbrack\rbrack \equiv\lbrack\lbrack S_1\cdots S_k\rbrack\rbrack,$$
		that is the union of the cyclic block rotations of the plate $\lbrack S_1,\ldots, S_k\rbrack$ is the whole ambient space $V_0^n$.
	\end{cor}
	
	\begin{proof}
		We have
		\begin{eqnarray*}
			0 & = & \mu_{S_1,S_2} \cdots \mu_{S_k,S_1}\\
			& = & (\lbrack\lbrack S_1,S_2\rbrack\rbrack-1)(\lbrack\lbrack S_2,S_3\rbrack\rbrack-1) \cdots (\lbrack\lbrack S_k,S_1\rbrack\rbrack-1)\\
			& = & \lbrack\lbrack S_1 S_2\cdots S_k\rbrack\rbrack -\left(\lbrack\lbrack S_1,S_2,\ldots, S_k\rbrack\rbrack + \lbrack\lbrack S_2,S_3,\ldots, S_1\rbrack\rbrack + \lbrack\lbrack S_k,S_1,\ldots, S_{k-1}\rbrack\rbrack\right) + \mathcal{O}(k-2),
		\end{eqnarray*}
		and it follows that 
		$$
		\lbrack\lbrack S_1 S_2\cdots S_k\rbrack\rbrack	\equiv \lbrack\lbrack S_1,S_2,\ldots, S_k\rbrack\rbrack + \lbrack\lbrack S_2,S_3,\ldots, S_1\rbrack\rbrack +\cdots + \lbrack\lbrack S_k,S_1,\ldots, S_{k-1}\rbrack\rbrack,$$
		where we have modded out by $\mathcal{O}(k-2) $, which is an alternating sum of only characteristic functions of cones of dimension $k-2\le n-2$, those with codimension at least 1 in $V_0^n$.
	\end{proof}

The identity of Proposition \ref{prop: triangulation identity higher codim} can be recognized as a deformation of the fundamental identity in the so-called subdivision algebra, see \cite{Grinberg, Meszaros}.
\begin{prop}\label{prop: triangulation identity higher codim}
	We have the triangulation identity for closed cones
	$$(\lbrack\lbrack S_1,S_2\rbrack\rbrack + \lbrack\lbrack S_2,S_3\rbrack\rbrack) \bullet \lbrack\lbrack S_1,S_3\rbrack\rbrack = \lbrack\lbrack S_1,S_2\rbrack\rbrack \bullet \lbrack\lbrack S_2,S_3\rbrack\rbrack + \lbrack\lbrack S_1,S_3\rbrack\rbrack,$$
	while for open cones we have
	$$(\mu_{S_1,S_2}+\mu_{S_2,S_3})\mu_{S_1S_3} =\mu_{S_1S_2}\mu_{S_2S_3} -\mu_{S_1S_3}.$$
\end{prop}

\begin{proof}
	It suffices to verify the identity in the plane $V_0^3$, where the two cones $\langle e_1-e_3,e_2-e_3\rangle_+$ and $\langle e_1-e_2,e_1-e_3\rangle_+$ intersect on the common line $\langle e_1-e_3\rangle_+$.  Therefore by inclusion/exclusion we have for their characteristic functions the identity 
	$$\lbrack \langle e_1-e_3,e_2-e_3\rangle_+ \rbrack + \lbrack \langle e_1-e_2,e_1-e_3\rangle_+ \rbrack = \lbrack \langle e_1-e_2,e_2-e_3\rangle_+\rbrack + \lbrack \langle e_1-e_3\rangle_+\rbrack,$$
	or in the bracket notation,
	$$\lbrack\lbrack 1,3\rbrack\rbrack \bullet \lbrack\lbrack 2,3\rbrack\rbrack + \lbrack\lbrack 1,2\rbrack\rbrack \bullet \lbrack\lbrack 1,3\rbrack\rbrack = \lbrack\lbrack 1,2\rbrack\rbrack \bullet \lbrack\lbrack 2,3\rbrack\rbrack + \lbrack\lbrack 1,3\rbrack\rbrack.$$
	For characteristic functions of cones generated by open half lines, $\mu_{i,j} = \lbrack\lbrack i,j\rbrack\rbrack - 1$, where $1$ is the characteristic function of the point at the origin, the identity can similarly be seen to be
	$$(\mu_{ij}+\mu_{jk})\mu_{ik} =\mu_{ij}\mu_{jk} -\mu_{ik}.$$
\end{proof}

\section{Blades from symmetric functions on edges of a type $A$ affine Dynkin graph}
Recall that for any nonempty subsets $S_1,S_2\subset\{1,\ldots, n\}$ with $S_1\cap S_2 = \emptyset$, we denote by 
$$\mu_{S_1,S_2} = \lbrack\lbrack S_1,S_2\rbrack\rbrack - \lbrack\lbrack S_1\rbrack \cap \lbrack S_2\rbrack \rbrack$$
the characteristic function of the set 
$$\{x\in V_0^n: x_{S_1}>0,\ x_{S_1\cup S_2}=0,\ x_i=0\text{ for } i\not\in S_1\cup S_2 \},$$
where the equality is strict, and we define the convolution product 
$$\mu_{S_1,\ldots, S_k} = \mu_{S_1,S_2}\cdots \mu_{S_{k-1},S_k}.$$

If $(S_1,\ldots, S_k)$ is an ordered set partition, the blade $((S_1,S_2,\ldots,S_k))$ is the \textit{complement} in $V_0^n$ of the union of the interiors of the plates
$$\lbrack S_1,\ldots, S_k\rbrack, \lbrack S_2,\ldots, S_{k}, S_1\rbrack,\ldots, \lbrack S_k,S_1\ldots, S_{k-1}\rbrack.$$
This union of interiors has characteristic function
$$\mu_{S_1,S_2}\mu_{S_2,S_3}\cdots \mu_{S_{k-1},S_k}+\mu_{S_2,S_3}\mu_{S_3,S_4}\cdots \mu_{S_{k},S_1} + \cdots \mu_{S_k,S_1}\mu_{S_1,S_2}\cdots \mu_{S_{k-2},S_{k-1}}.$$

Definition \ref{defn:blade} is due to A. Ocneanu.
\begin{defn}[\cite{OcneanuVideo}]\label{defn:blade}
	For an ordered set partition $(S_1,\ldots, S_k)$ of $\{1,\ldots, n\}$ with $k\ge 3$, the \textit{blade} $((S_1,\ldots, S_k))$ is the set theoretic \textit{union} of the Minkowski sums of cones,
	$$((S_1,\ldots, S_k))=\bigcup_{1\le i<j\le k} \lbrack S_1,S_2\rbrack\oplus \cdots\oplus \widehat{\lbrack S_i,S_{i+1}\rbrack} \oplus \cdots \oplus\widehat{\lbrack S_j,S_{j+1}\rbrack} \oplus \cdots\oplus \lbrack S_k,S_1\rbrack,$$
	where the hat means that that corresponding term has been omitted, and where we adopt the convention $S_{k+1}=S_1$.  When $k=2$ and $S_1\sqcup S_2 = \{1,\ldots, n\}$, then put 
	$$((S_1,S_2)) = \lbrack S_1\rbrack \oplus \lbrack S_2 \rbrack = \left\{\sum_{i\in S_1\cup S_2}t_i e_i: t_{S_1} = 0,\ t_{S_2}=0 \right\}.$$
	Finally, when $k=1$ we set $$((12\cdots n)) =V_0^n.$$
\end{defn}
	When the blocks in the ordered set partition are singlets, we have the following two interesting interpretations.
	
	\begin{prop}
		The (nondegenerate) blade $((\sigma_1,\sigma_2,\ldots, \sigma_n))$ is the union of the codimension 1 facets to the normal fan to the simplex with vertices
		$$e_{\sigma_1} - e_{\sigma_{n}},e_{\sigma_2}-e_{\sigma_1},\ldots, e_{\sigma_n}-e_{\sigma_{n-1}}.$$
		Moreover, as illustrated in Figure \ref{fig:bladesBijectionPathDashed}, the (one-skeleton of the) blade $((\sigma_1,\sigma_2,\ldots, \sigma_n))$ can be interpreted as a Hamiltonian cycle on the simplex with vertices $e_1,\ldots, e_n$.
	\end{prop}
\begin{figure}[h!]
	\centering
	\includegraphics[width=1\linewidth]{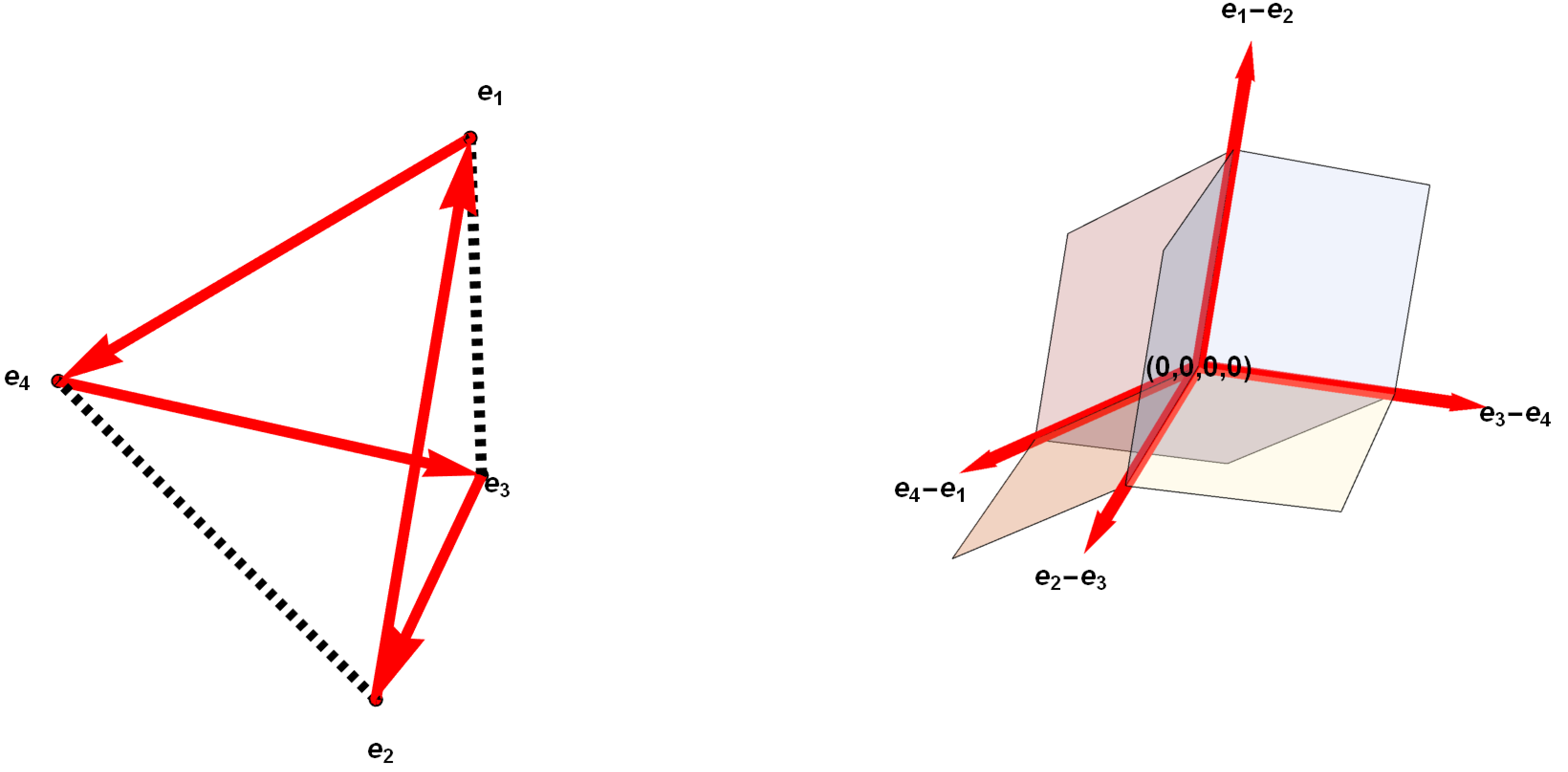}
	\caption{Alternate representation of the $n=4$ blade $((1,2,3,4))$. Left: as a Hamiltonian cycle on the tetrahedron, and Right: as the union of the conical hulls of pairs of roots extending from the origin, selected from $e_1-e_2,e_2-e_3,e_3-e_4,e_4-e_1$.  This is easily seen to extend to a bijection between Hamiltonian paths on a simplex and nondegenerate blades $((i_1,\ldots, e_n))$.
	}
	\label{fig:bladesBijectionPathDashed}
\end{figure}
	Denote by $\Gamma_{S_1,\ldots, S_k}$ the characteristic function of the blade $((S_1,\ldots, S_k))$.
	
	For compactness we further denote 
	$$\Gamma_{S} = 1_S \text{ and }  \Gamma_{S_1,S_2} = \lbrack(S_1,S_2)\rbrack$$

We distinguish the case of blades which are labeled by 3-block ordered set partitions.

\begin{defn}
	For a triple of disjoint subsets $(S_1,S_2,S_3)$ of $\{1,\ldots, n\}$, the (characteristic) function
	$$\gamma_{S_1,S_2,S_3}=1_{S_1}1_{S_2}1_{S_3} + \mu_{S_1,S_2} + \mu_{S_2,S_3} + \mu_{S_3,S_1}$$
	is called a \textit{tripod}.  
\end{defn}

\begin{figure}[h!]
	\centering
	\includegraphics[width=0.4\linewidth]{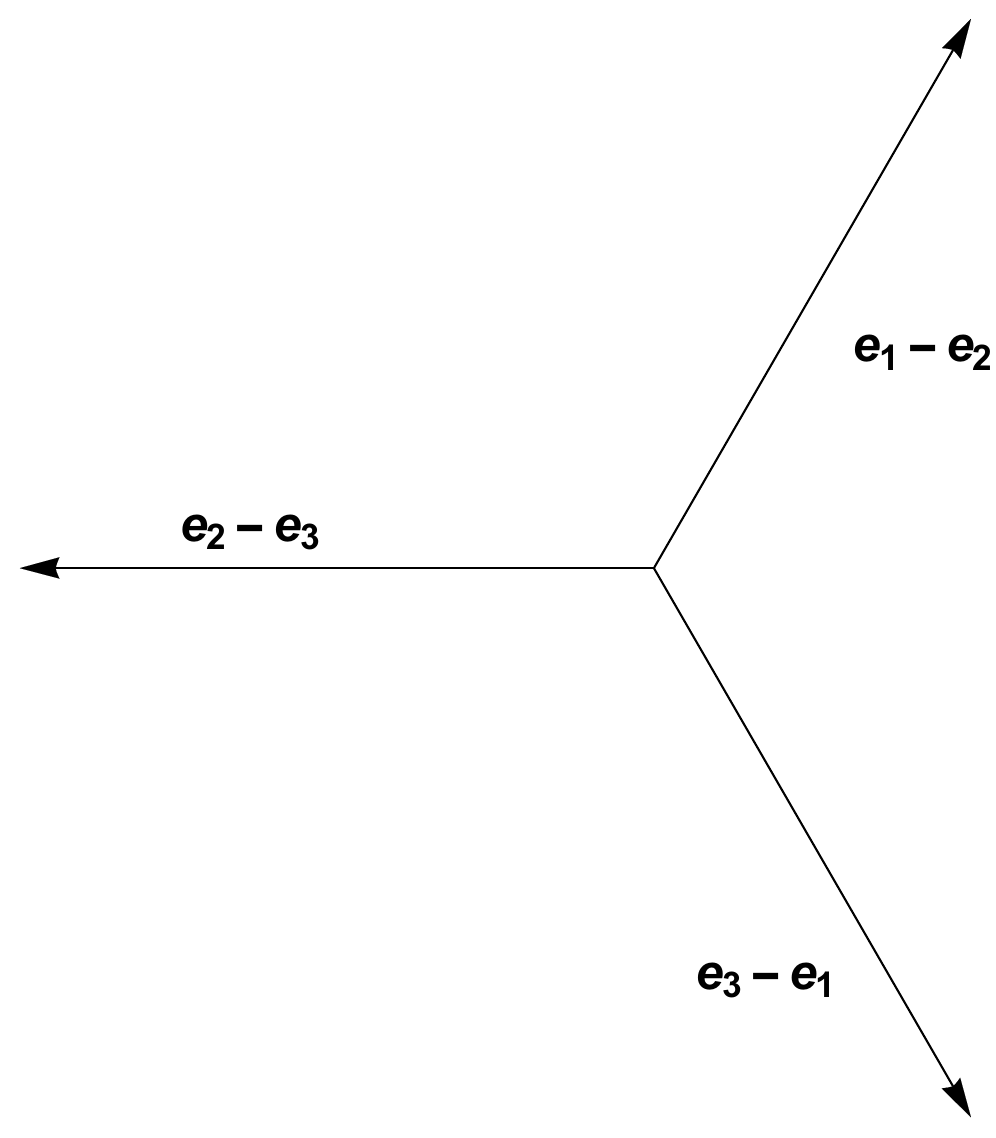}
	\caption{Blade in three coordinates, characteristic function $\gamma_{1,2,3} = 1+\mu_{12} + \mu_{23} + \mu_{31}$.  Arrows indicate that the rays extend to infinity.  The $\mu_{ij}$'s are characteristic functions of open rays extending from $(0,0,0)$, while ``1'' is the characteristic function of the point $(0,0,0)$ itself.}
	\label{fig:blade3coordinates0}
\end{figure}

Informally, a blade is the union of the facets of the complete fan built from the $k$ cyclic rotations of the plate $\lbrack S_1,\ldots, S_k\rbrack.$  Note that when the blocks $S_i$ are not singlets, the cones in the blade are not in general pointed.

We show now that lumped blades, labeled by standard ordered set partitions with at least one block of size $\ge2$, reduce naturally.
\begin{prop}
	The characteristic function $\Gamma_{S_1,\ldots, S_k}$ of any blade $((S_1,\ldots, S_k))$ with $k\ge 3$ is a convolution product of characteristic functions of (one-dimensional) tripods $\gamma_{i,j,k}$ and sticks $1_{ab}$ labeled by respectively triples $(i,j,k)$ of integers with $i<j<k$ and $a<b$, where $i,j,k,a,b\in\{1,\ldots, n\}$.
\end{prop}
\begin{proof}
	Let $(S_1,\ldots, S_k)$ be any ordered set partition of $\{1,\ldots, n\}$. Choose elements $i_j\in S_j$ for $j=1,\ldots, k$.  Then, applying the identity $1_T\gamma_{i,j,k} = \gamma_{T,j,k}$ whenever $i\in T$ but $j,k\not\in T$, we have
	$$\Gamma_{S_1,\ldots, S_k} = 1_{S_1}1_{S_2}\cdots 1_{S_k}\gamma_{i_1,i_2,i_3}\gamma_{i_1,i_3,i_4}\cdots \gamma_{i_1,i_{k-1},i_k}$$
\end{proof}

\begin{prop}\label{prop: blade as difference}
	The characteristic function of the blade $((S_1,\ldots, S_k))$ has the expansion in terms of convolution products, as 
\begin{eqnarray*}
	\Gamma_{S_1,\ldots, S_k} & = & 1_{S_1\cup\cdots \cup S_k} - \left(\mu_{S_1,S_2,\ldots, S_k} + \mu_{S_2,S_3,\ldots, S_1} + \cdots + \mu_{S_k,S_1,\ldots, S_{k-1}}\right)\\
	& = & 1_{S_1\cup\cdots \cup S_k} -\left((\mu_{S_1,S_2}\cdots \mu_{S_{k-1},S_k})+(\mu_{S_2,S_3}\cdots \mu_{S_{k},S_1})+\cdots + (\mu_{S_k,S_1}\cdots \mu_{S_{k-2},S_{k-1}})\right)
\end{eqnarray*}
\end{prop}

\begin{figure}[h!]
	\centering
	\includegraphics[width=0.5\linewidth]{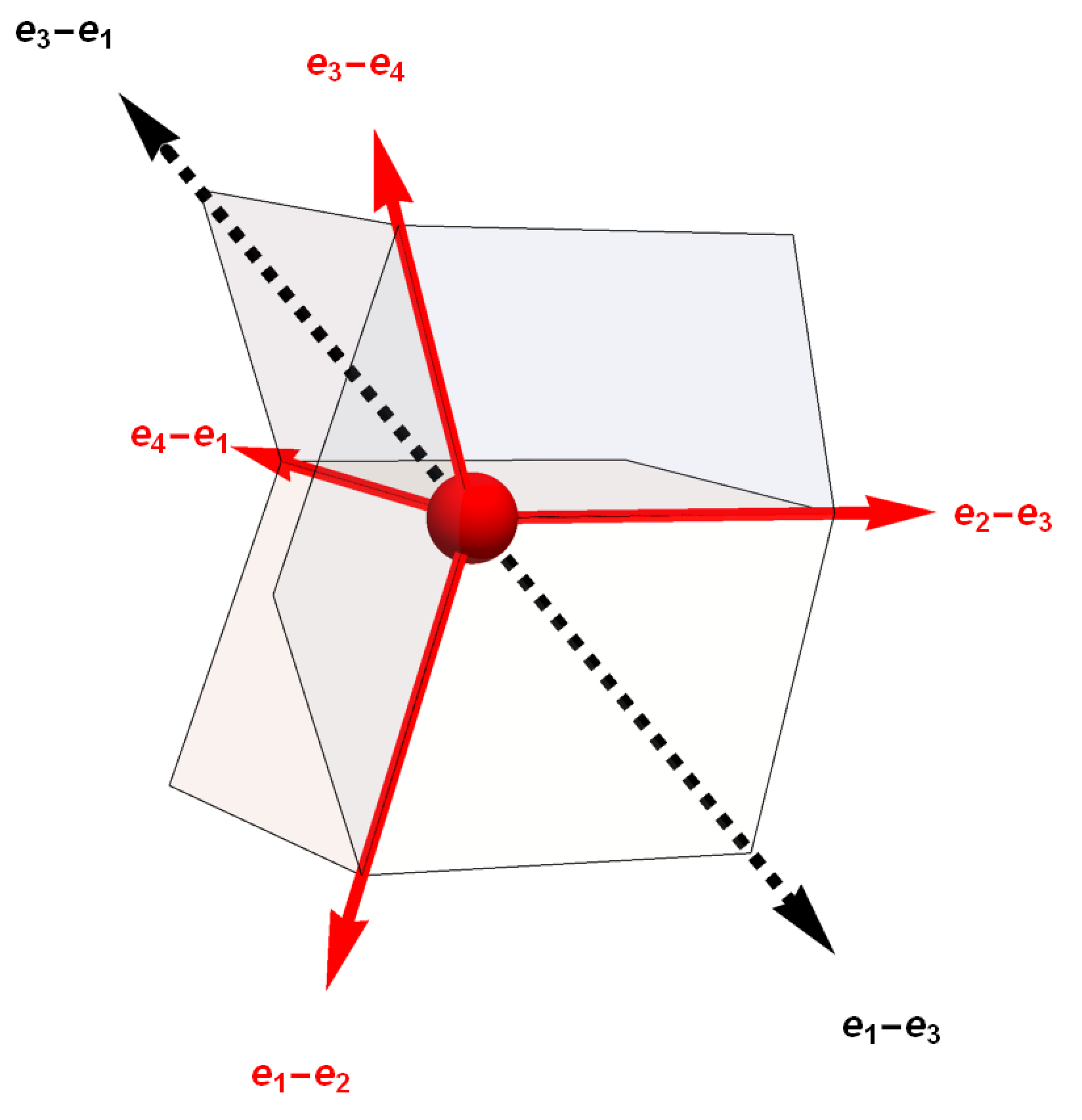}
	\caption{The $n=4$ blade $((1,2,3,4))$ from the convolution of two 3-coordinate blades $\gamma_{1,2,3}, \gamma_{1,3,4}$. This convolution product has characteristic function
		$$\Gamma_{(1,2,3,4)} = \gamma_{1,2,3}\gamma_{1,3,4} = (1+\mu_{12}+\mu_{23}+\mu_{31})(1+\mu_{13}+\mu_{34}+\mu_{41}).$$
		The dotted rays cancel in the convolution product and the red dot and rays remain, together with the $6=\binom{4}{2}$ sheets, and we obtain the characteristic function
		$$\Gamma_{(1,2,3,4)}=1+(\mu_{12}+\mu_{23} + \mu_{3 4} + \mu_{41}) + (\mu_{12}\mu_{23}+\mu_{12}\mu_{34}+\mu_{12}\mu_{41} + \mu_{23}\mu_{34} + \mu_{23}\mu_{41} + \mu_{34} \mu_{41})$$
		$$=\sum_{j=0}^2e_j(\mu_{12},\mu_{23},\mu_{34},\mu_{41}),$$
		where $e_j$ is the $j^\text{th}$ elementary symmetric function.
	}
	\label{fig:bladesstickframed3a}
\end{figure}

\begin{prop}
	We have
		$$\Gamma_{S_1,\ldots, S_k} =1_{S_1}1_{S_2}\cdots 1_{S_k}+\sum_{j=1}^{k-2} \mathbf{e}_j(\mu_{S_1,S_2},\ldots, \mu_{S_k,S_1}), $$
	where $e_j$ is the $j$th elementary symmetric function.
\end{prop}
\begin{proof}
	In the following computation, we shall abuse notation and write $\mathbf{1}$ for the convolution product of characteristic functions of any subcollection of the subspaces $\lbrack S_i\rbrack$ for $i=1,\ldots, k$.
	By Proposition \ref{prop: open plates properties} we have $\mu_{S_1,S_2}\cdots \mu_{S_k,S_1}=0$; but $\mathbf{1}+\mu_{S_1,S_2} = \lbrack\lbrack S_1,S_2\rbrack\rbrack$, hence
	\begin{eqnarray*}
		\lbrack\lbrack 12\cdots n\rbrack\rbrack & = & \lbrack\lbrack S_1,S_2\rbrack\rbrack \bullet \lbrack\lbrack S_2,S_3\rbrack\rbrack \bullet \cdots \bullet \lbrack\lbrack S_k,S_1\rbrack\rbrack\\
		& = & ( \mathbf{1}+\mu_{S_1,S_2})( \mathbf{1}+\mu_{S_2,S_3})\cdots ( \mathbf{1}+\mu_{S_k,S_1})\\
		& = &  \mathbf{1}+\sum_{j=1}^{k-1} \mathbf{e}_j(\mu_{S_1,S_2},\ldots, \mu_{S_k,S_1})\\
		& = & \mathbf{1}+\sum_{j=1}^{k-2} \mathbf{e}_j(\mu_{S_1,S_2},\ldots, \mu_{S_k,S_1}) + \mathbf{e}_{k-1}(\mu_{S_1,S_2},\ldots, \mu_{S_k,S_1}),
	\end{eqnarray*}                                                                                                    the 
	where $\mathbf{e}_j$ is the $j$th elementary symmetric function, and $\mathbf{e}_{k-1}(\mu_{S_1,S_2},\ldots, \mu_{S_k,S_1})$ is the characteristic function of the complement of the blade $((S_1,\ldots, S_k))$.  The first equality follows since $\lbrack\lbrack S_1,S_2\rbrack\rbrack\bullet\cdots \bullet \lbrack\lbrack S_k,S_1\rbrack\rbrack = \lbrack\lbrack S_1\cup\cdots\cup S_k\rbrack\rbrack$.
	
	Consequently we have
	$$\Gamma_{S_1,\ldots, S_k} = \mathbf{1}+\sum_{j=1}^{k-2} \mathbf{e}_j(\mu_{S_1,S_2},\ldots, \mu_{S_k,S_1}).$$
\end{proof}

\begin{thm}\label{thm: blade flag factorization}
	The characteristic function of the blade $\Gamma_{S_1,\ldots, S_k}$
	labeled by a standard ordered set partition $(S_1,\ldots,S_k)$ of $\{1,\ldots, n\}$ admits the ``flag'' factorization
	$$\Gamma_{S_1,\ldots, S_k} = \gamma_{S_1,S_2,S_3}\gamma_{S_1,S_3,S_4}\cdots \gamma_{S_1,S_{k-1},S_k}.$$
\end{thm}

\begin{proof}
	To improve readability, let us temporarily adopt the notation $\mu_{S_i,S_j} = \mu_{i,j}$ and $\gamma_{S_i,S_j,S_k} = \gamma_{i,j,k}$, and 
	$$\mu_{S_1,S_2}\mu_{S_2,S_3} \cdots \mu_{S_{k-1},S_k} = \mu_{1,2,\ldots, k}.$$
	We shall induct on $k$, using the identities
	$$\Gamma_{S_1,\ldots, S_k} = \lbrack\lbrack S_1\cup\cdots\cup S_k\rbrack\rbrack -\left((\mu_{S_1,S_2}\cdots \mu_{S_{k-1},S_k})+(\mu_{S_2,S_3}\cdots \mu_{S_{k},S_1})+\cdots + (\mu_{S_k,S_1}\cdots \mu_{S_{k-2},S_{k-1}})\right),$$
	that is 
	$$\Gamma_{S_1,\ldots, S_k} = 1_{S_1\cup\cdots\cup S_k} -\left( \mu_{12\cdots k}+\mu_{23\cdots k1} +\cdots + \mu_{k1\cdots (k-1)}\right).$$
	For the base case we have 
	$$(\Gamma_{1,2,3} = )\gamma_{1,2,3}  = 1_{123} - (\mu_{1,2}\mu_{2,3}+\mu_{2,3}\mu_{3,1}+\mu_{3,1}\mu_{1,2})=1_{1(k-1)(k)} - \mu_{1,k-1,k}-\mu_{k-1,k,1}- \mu_{k,1,k-1}.$$
	Computing in the flag triangulation gives
	\begin{eqnarray*}
		& & \gamma_{1,2,3}\gamma_{1,3,4}\cdots \gamma_{1,k-2,k-1}\gamma_{1,k-1,k} \\
		& = & (\gamma_{1,2,3}\gamma_{1,3,4}\cdots \gamma_{1,k-2,k-1})\gamma_{1,k-1,k}\\
		& = & (1_{12\cdots (k-1)} - \mu_{1,2,\ldots, k-1}-\mu_{2,3,\ldots,1}-\cdots -\mu_{k-1,1,\ldots, k-2})(1_{1(k-1)(k)} - \mu_{1,k-1,k}-\mu_{k-1,k,1}- \mu_{k,1,k-1})\\
		& = & 1_{12\cdots k} \\
		& - & 1_{12\cdots (k-1)}(\mu_{1,k-1,k}+\mu_{k-1,k,1}+ \mu_{k,1,k-1})\\
		& - & 1_{1(k-1)(k)}( \mu_{1,2,\ldots, k-1}+\mu_{2,3,\ldots,1}+\cdots +\mu_{k-1,1,\ldots, k-2})\\
		& + &  (\mu_{1,k-1,k}+\mu_{k-1,k,1}+ \mu_{k,1,k-1})( \mu_{1,2,\ldots, k-1}+\mu_{2,3,\ldots,1}+\cdots +\mu_{k-1,1,\ldots, k-2})
	\end{eqnarray*}
	We aim to prove that lines 2 and 3 are identically zero:
	$$1_{12\cdots (k-1)}(\mu_{1,k-1,k}+\mu_{k-1,k,1}+ \mu_{k,1,k-1})=0$$
	and
	$$1_{1(k-1)(k)}( \mu_{1,2,\ldots, k-1}+\mu_{2,3,\ldots,1}+\cdots +\mu_{k-1,1,\ldots, k-2})=0.$$
	
	We first establish some essential (but easily verified) identities.  First, we have
	$$\mu_{p,q}\mu_{i,\ldots, j} = -\mu_{i,\ldots, j}$$
	for any $i\le p<q\le j$, as can be checked geometrically by dualizing with $\star$, in which case the convolution $\bullet$ becomes the pointwise product of characteristic functions.  In the dual, $\mu_{p,q}^\star$ takes the value $-1$ on an open half space containing the support of $\mu_{i,i+1,\ldots, p,\ldots, q,\ldots, i-1}^\star$, (and is zero on the complement), hence the pointwise product acts by switching the sign of $\mu_{i,i+1,\ldots, p,\ldots, q,\ldots, i-1}^\star$.
	
	Whenever $p,q\in\{1,\ldots, j\}$ we have
	$$1_{12\cdots j} \mu_{p,q} = 1_{12\cdots j}(1_{pq} - \lbrack\lbrack p,q\rbrack\rbrack) = 1_{12\cdots j} - 1_{12\cdots j} = 0.$$
	
	Now on the other hand, if $p\in \{1,2,\ldots, j\}$ but $q\not\in \{1,2,\ldots, j\}$ then we have
	$$1_{12\cdots j}\mu_{p,q} =1_{12\cdots j}(\lbrack\lbrack p,q\rbrack\rbrack-1)= \lbrack\lbrack 12\cdots j,q\rbrack\rbrack - 1_{12\cdots j}.$$
	Moreover, using 
	$$0= \mu_{S_1,S_2} \mu_{S_2,S_1} = (\lbrack\lbrack S_1,S_2\rbrack\rbrack-1_{S_1}1_{S_2})(\lbrack\lbrack S_2,S_1\rbrack\rbrack-1_{S_1}1_{S_2})$$
	it follows that 
	\begin{eqnarray*}
		&&1_{12\cdots (k-1)}(\mu_{1,k-1}\mu_{k-1,k} + \mu_{k-1,k} \mu_{k,1} + \mu_{k,1}\mu_{1,k-1}) \\
		& = & 0 \cdot \mu_{k-1,k} \\
		& + & (\lbrack\lbrack 12\cdots k-1,k\rbrack\rbrack - 1_{12\cdots k-1})(\lbrack\lbrack k,12\cdots k-1\rbrack\rbrack - 1_{12\cdots k-1})\\
		& + & \mu_{k,1}\cdot 0\\
		& = & 0.
	\end{eqnarray*}
	
	This shows that the second line vanishes.  Proving that for the third line we have
	$$1_{1(k-1)(k)}( \mu_{1,2,\ldots, k-1}+\mu_{2,3,\ldots,1}+\cdots +\mu_{k-1,1,\ldots, k-2})=0$$
	is similar and we omit the computation.
	
	It remains to prove the identity 
	$$(\mu_{1,k-1,k}+\mu_{k-1,k,1}+ \mu_{k,1,k-1})( \mu_{1,2,\ldots, k-1}+\mu_{2,3,\ldots,1}+\cdots +\mu_{k-1,1,\ldots, k-2})$$
	$$=-\left( \mu_{12\cdots k}+\mu_{23\cdots k1} +\cdots + \mu_{k1\cdots (k-1)}\right).$$
	
	Indeed, while
	\begin{eqnarray*}
		& & \mu_{1,k-1}\mu_{k-1,k}( \mu_{1,2,\ldots, k-1}+\mu_{2,3,\ldots,1}+\cdots +\mu_{k-1,1,\ldots, k-2})\\
		& = & (-\mu_{1,\ldots, k-1}\mu_{k-1,k}) + (-\mu_{2,3,\ldots,k-1, 1}\mu_{1,k-1}+\cdots)\\
		& = & (-\mu_{1,\ldots, k-1}\mu_{k-1,k}) + (0+\cdots+0)\\
		& = & -\mu_{1,2,\ldots, k}
	\end{eqnarray*}
	and
	\begin{eqnarray*}
		\mu_{k,1}\mu_{1,k-1}( \mu_{1,2,\ldots, k-1}+\mu_{2,3,\ldots,k-1,1}+\cdots +\mu_{k-1,1,\ldots, k-2}) & = & -(\mu_{k,1,\ldots, k-1}),
	\end{eqnarray*}
	we have the remaining $k-2$ nonzero contributions from $\mu_{k-1,k}\mu_{k,1}$:
	\begin{eqnarray*}
		& & \mu_{k-1,k}\mu_{k,1}( \mu_{1,2,\ldots, k-1}+\mu_{2,3,\ldots,k-1,1}+\cdots +\mu_{k-1,1,\ldots, k-2})\\
		& = & -(\mu_{1,2,\ldots, k-1}\mu_{k-1,k}\mu_{k,1})+\mu_{2,3,\ldots,k-2,k-1}(\mu_{k-1,1}\mu_{k-1,k}\mu_{k,1})+\cdots +(\mu_{k-1,1}\mu_{k-1,k}\mu_{k,1})\mu_{1,\ldots, k-2}\\
		& = & -( \mu_{2,3,\ldots, k,1} +\mu_{3,\ldots, k,1,2}+\cdots + \mu_{k-1,1,\ldots, k-2} ),
	\end{eqnarray*}
	and we finally obtain
	$$\gamma_{1,2,3}\gamma_{1,3,4}\cdots \gamma_{1,k-2,k-1}\gamma_{1,k-1,k} = 1_{12\cdots k} -\left( \mu_{1,2,\ldots, k}+\mu_{2,3,\ldots, k,1} +\cdots + \mu_{k,1,\ldots, (k-1)}\right),$$
	which, after substituting back in the notation $i\mapsto S_i$, becomes exactly $\Gamma_{S_1,\ldots, S_k}$, as desired.
\end{proof}

\section{Blades from triangulations, and factorization independence}

Choose a cyclic (counterclockwise, say) order on the vertices of a polygon with vertices labeled by the (standard) ordered set partition $\mathbf{S}=(S_1,\ldots, S_k)$ of $\{1,\ldots, n\}$, where we assume that $1\in S_1$.  Let 
$$\mathcal{T} = \{((S_{a_1},S_{b_1},S_{c_1})),\ldots, ((S_{a_{k-2}},S_{b_{k-2}},S_{c_{k-2}}))\}$$
be any set of (cyclically oriented) triangles forming a triangulation of the $k$-gon, labeled such that each $(a_i,b_i,c_i)$ satisfies $a_i<b_i<c_i$ or $b_i<c_i<a_i$ or $c_i<a_i<b_i$. 

\begin{defn}
	We shall say that $(S_{a_i},S_{b_i},S_{c_i})$ is a cyclic subword of $(S_1,\ldots, S_k)$ if it satisfies the above condition, $a_i<b_i<c_i$ or $b_i<c_i<a_i$ or $c_i<a_i<b_i$. 
\end{defn}

\begin{prop}\label{prop: triangulation independence complete blade}
	The product 
	$$\Gamma_{\mathcal{T}} = \gamma_{S_{a_1},S_{b_1},S_{c_1}}\gamma_{S_{a_2},S_{b_2},S_{c_2}}\cdots \gamma_{S_{a_{k-2}},S_{b_{k-2}},S_{c_{k-2}}}$$
	is independent of the triangulation $\mathcal{T}$.
\end{prop}

\begin{proof}
		As any two triangulations of an $n$-gon are related by a sequence of flips, as in Figures \ref{fig:squaretriangulationchange0}, replacing the pair $\{(S_i,S_j,S_k),(S_i,S_k,S_\ell)\}$ with $\{(S_i,S_j,S_\ell),(S_j,S_k,S_\ell)\}$, it suffices to verify directly that $\gamma_{S_i,S_j,S_k}\gamma_{S_i,S_k,S_\ell} = \gamma_{S_i,S_j,S_\ell}\gamma_{S_j,S_k,S_\ell}$.  Abbreviating $\mu_{S_iS_j}$ as $\mu_{ij}$ and $\gamma_{S_i,S_j,S_k}$ as $\gamma_{ijk}$, and all products of the characteristic functions of subspaces, $\lbrack\lbrack S_\ell\rbrack\rbrack$, with 1, we have
	\begin{eqnarray*}
		\gamma_{ijk}\gamma_{ik\ell} & = & \left(1+\mu_{ij}+\mu_{jk}+\mu_{ki}\right)\left(1+\mu_{ik}+\mu_{k\ell}+\mu_{\ell i}\right)\\
		& = & 1+(\mu_{ij}+\mu_{jk} + \mu_{k \ell} + \mu_{\ell i}) + (\mu_{ki} + \mu_{ik}) \\
		& + & (\mu_{ij}+\mu_{jk})\mu_{ik} + \mu_{ki} (\mu_{k\ell} + \mu_{\ell i}) +  (\mu_{ij}\mu_{k\ell} + \mu_{ij}\mu_{\ell i} + \mu_{jk}\mu_{k\ell} + \mu_{jk} \mu_{\ell i}) + \mu_{ik}\mu_{ki} 
	\end{eqnarray*}
	Using the triangulation identity for a three-block (open) plate, $(\mu_{ij}+\mu_{jk})\mu_{ik} =\mu_{ij}\mu_{jk} -\mu_{ik}$ and $\mu_{ki} (\mu_{k\ell} + \mu_{\ell i}) = \mu_{k\ell}\mu_{\ell i} -\mu_{k i} $, as well as $\mu_{ik}\mu_{ki}=0$, after cancellation we obtain
	\begin{eqnarray*}
		\gamma_{ijk}\gamma_{ik\ell} & = & 1+(\mu_{ij}+\mu_{jk} + \mu_{k \ell} + \mu_{\ell i}) +  (\mu_{ij}\mu_{jk}+\mu_{ij}\mu_{k\ell}+\mu_{ij}\mu_{\ell i} + \mu_{jk}\mu_{k\ell} + \mu_{jk}\mu_{\ell i} + \mu_{k\ell} \mu_{\ell i}).
	\end{eqnarray*}
	Performing an analogous computation for $\gamma_{ij\ell}\gamma_{jk\ell}$ yields the same result.
\end{proof}

\begin{figure}[h!]
	\centering
	\includegraphics[width=0.65\linewidth]{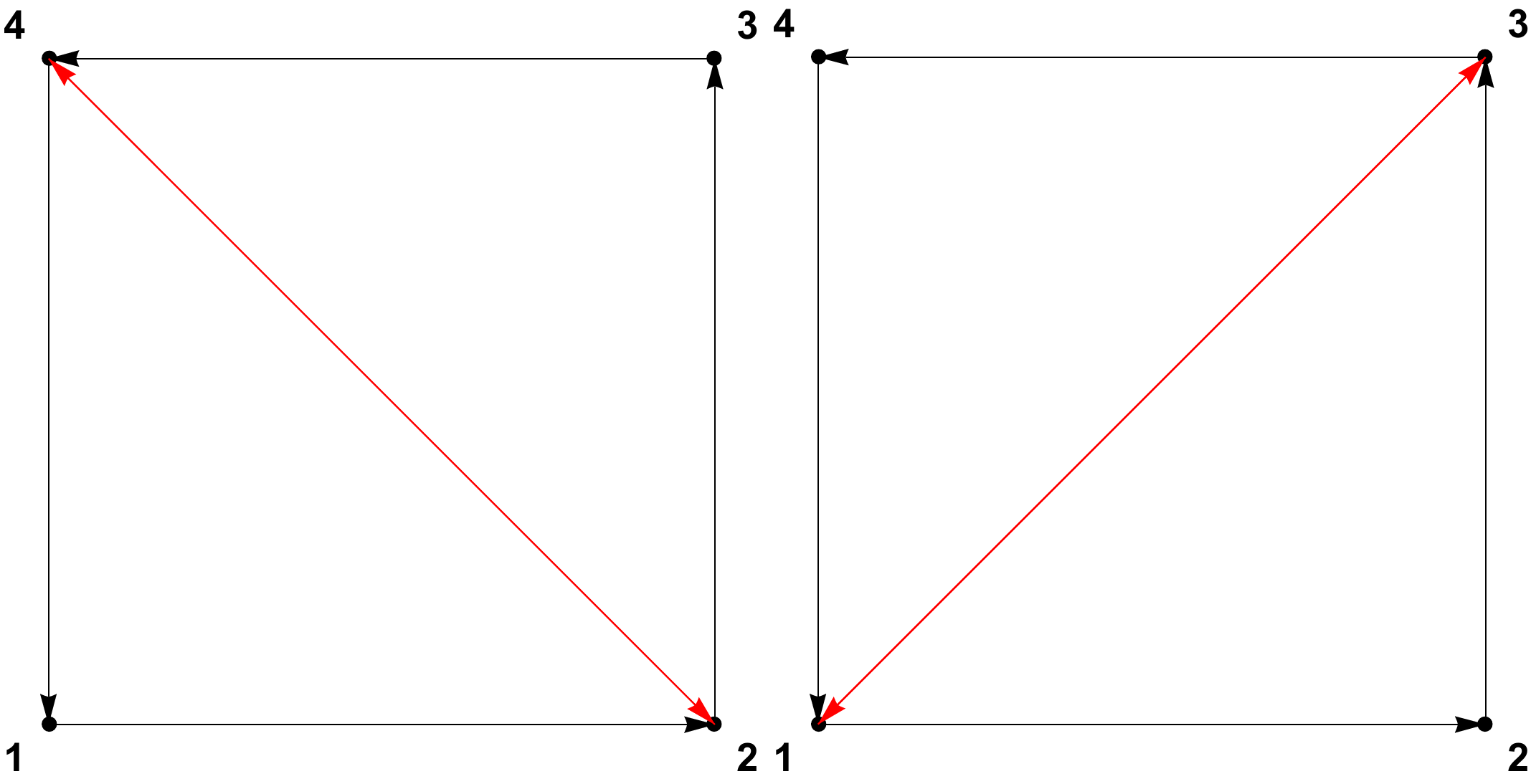}
	\caption{Independence of triangulation of the polygon for Proposition \ref{prop: triangulation independence complete blade} $\{(1,2,4),(2,3,4)\}\leftrightarrow \{(1,2,3),(1,3,4)\}.$}
	\label{fig:squaretriangulationchange0}
\end{figure}

\begin{rem}\label{rem: blade triangulation independent}
	Proposition \ref{prop: triangulation independence complete blade} justifies our usual omission of the triangulation, using instead the notation $\Gamma_{S_1,\ldots, S_k}$ for the characteristic function of the blade labeled by the (standard) ordered set partition $(S_1,\ldots, S_k)$.
\end{rem}
From the ring-theoretic identities for characteristic functions of blades in Theorem \ref{thm: blade flag factorization} and Proposition \ref{prop: triangulation independence complete blade}, we derive the corresponding set-theoretic identity for blades themselves, using Minkowski sums.

\begin{cor}\label{cor: blade Minkowski sum decomposition}
	For a blade $((S_1,\ldots, S_k))$ labeled by a standard ordered set partition $(S_1,\ldots, S_k)$, and any triangulation $\{(S_{a_1},S_{b_1},S_{c_1}),\ldots, (S_{a_{k-2}},S_{b_{k-2}},S_{c_{k-2}})\}$ of the cyclically-oriented polygon with vertices labeled $S_1,\ldots, S_k$, where $a_i<b_i<c_i$ for all $i$, we have
	$$((S_1,\ldots, S_k)) = ((S_{a_1},S_{b_1},S_{c_1}))\oplus\cdots \oplus ((S_{a_{k-2}},S_{b_{k-2}},S_{c_{k-2}})).$$
\end{cor}

\begin{proof}
	By Theorem \ref{thm: blade flag factorization} and Proposition \ref{prop: triangulation independence complete blade} we have the identity of piecewise constant functions 
	$$\Gamma_{S_1,\ldots, S_k} = \Gamma_{S_{a_1},S_{b_1},S_{c_1}}  \cdots  \Gamma_{S_{a_{k-2}},S_{b_{k-2}},S_{c_{k-2}}}.$$
	Using the homomorphism property of the convolution product from Theorem \ref{thm: valuation and convolution}, we express this as
	$$\lbrack ((S_1,\ldots, S_k))\rbrack =\Gamma_{S_1,\ldots, S_k} = \lbrack ((S_{a_1},S_{b_1},S_{c_1})) \oplus \cdots \oplus ((S_{a_{k-2}},S_{b_{k-2}},S_{c_{k-2}}))\rbrack,$$
	which is a relation of ($\{0,1\}$-valued) characteristic functions which holds identically.  In particular the preimages of 1 on both sides are the same, hence
	$$ ((S_1,\ldots, S_k))= ((S_{a_1},S_{b_1},S_{c_1})) \oplus \cdots \oplus ((S_{a_{k-2}},S_{b_{k-2}},S_{c_{k-2}})),$$
	proving that a blade $((S_1,\ldots, S_k))$ can be expressed (non-uniquely) as a Minkowski sum of tripods $((S_{a_t},S_{b_t},S_{c_t}))$. 
\end{proof}

\section{Quasishuffle expansions for simplicial cones}
	To derive the formula used in the proof of Theorem \ref{thm: canonical basis cyclic sum plate}, we need two results from \cite{EarlyCanonicalBasis}.  We recall these shortly here.

	\begin{defn}
		Given $\ell\ge 1$ ordered set partitions $\mathbf{S}_1,\ldots, \mathbf{S}_\ell$ of subsets $J_1,\ldots, J_\ell$ of $\{1,\ldots, n\}$, an ordered set partition $\mathbf{T}$ of $\{1,\ldots, n\}$ is a \textit{quasishuffle} of $\mathbf{S}_1,\ldots,\mathbf{S}_\ell$ if each block of $\mathbf{S}$ is a union of at most one block each from respectively $\mathbf{S}_1,\ldots, \mathbf{S}_\ell$, and such that if $S_{i,a}$ appears before $S_{i,b}$ in $\mathbf{S}_i$, then $S_{i,a}$ appears (in a subset of a block) strictly before $S_{i,b}$ in $\mathbf{S}$.
	\end{defn}
	Note that we are not requiring that $\mathbf{S}_1,\ldots, \mathbf{S}_\ell$ be supported on disjoint subsets of $\{1,\ldots, n\}$.
	
	For example, $\mathbf{T} = (135,24,67,8)$ is a quasishuffle of the pair $\mathbf{S}_1 = (15,67,8)$ and $\mathbf{S}_2=(3,24,8)$.  
	
	Define $\mathbf{S}_3 = (1,3,67),\ \ \mathbf{S}_4 = (5,24,8)$ and $\mathbf{S}_3' = (13,67)$.
	Then $\mathbf{T}$ is \textit{not} a quasishuffle of the pair $\mathbf{S}_3 ,\mathbf{S}_4$ (here the labels $1$ and $3$ are in different blocks of $\mathbf{S}_3$ but are in the same block in $\mathbf{T}$), but it \textit{is} a quasishuffle of the pair $\mathbf{S}_3',\mathbf{S}_4$.

	In Lemma \ref{lem: Weyl chamber decomposition shuffle} we recall a result which was used in \cite{EarlyCanonicalBasis} from the theory of $(P,\omega)$-partitions to decompose the characteristic function of a union of closed Weyl chambers into an alternating sum of partially closed Weyl chambers in a canonical way that depends on descent positions, with respect to the natural order $(1,\ldots, n)$.  Following \cite{EarlyCanonicalBasis} by translating Proposition 4.6.10 of \cite{Stanley EC1} we obtain Lemma \ref{lem: Weyl chamber decomposition shuffle}.

\begin{lem}\label{lem: Weyl chamber decomposition shuffle}
	We have the decomposition into disjoint sets
	$$V_0^n=\sqcup_{\sigma\in \symm_n} C_\sigma,$$
	where each $C_\sigma$ is the partially open Weyl chamber defined by 
	$$x_{\sigma_i}\ge x_{\sigma_{i+1}}\ \text{ if } \sigma_{i}<\sigma_{i+1}$$
	and
	$$x_{\sigma_i}>x_{\sigma_{i+1}}\ \text{ if } \sigma_{i}>\sigma_{i+1}.$$
	The characteristic function of $C_\sigma$, with $\sigma=(\sigma_1,\ldots, \sigma_n)$ given, is
	$$\lbrack C_\sigma\rbrack=\sum_{\pi}(-1)^{n-\operatorname{len}(\pi)}\lbrack\pi^\star\rbrack,$$
	where the sum is over the set of ordered set partitions $(S_1,\ldots, S_k)$ with blocks
	$$S_i = \{\sigma_p:d_{i-1}<p\le d_i\},$$
	where 
	$$\{d_1<\cdots<d_{k-1}\}=\{i: \sigma_i <\sigma_{i+1}\}$$
	and we define $d_0=0$ and $d_k=n$.
\end{lem}

As in \cite{EarlyCanonicalBasis}, Lemma \ref{lem: Weyl chamber decomposition shuffle}, together with applications of duality, gives the expansion of a Minkowski sum of cones labeled by a set of disjoint ordered set partitions $\mathbf{S}_1,\ldots, \mathbf{S}_\ell$.

	\begin{cor}\label{cor: shuffleLump Identity}
	Given $\ell$ disjoint ordered set partitions
	$$\mathbf{S}_1=(S_{1},S_{2},\ldots, S_{k_1}),\mathbf{S}_2=(S_{k_1+1},S_{k_1+2},\ldots, S_{k_1+k_2}),\ldots, \mathbf{S}_{\ell}=(S_{k_1+\cdots+k_{l-1}+1},\ldots, S_{k_1+\cdots+k_{\ell}})$$
	such that $\bigsqcup_{i=1}^{k_1+\cdots+k_l}S_i=\{1,\ldots, n\}$, then we have the identity for characteristic functions 
	
	$$\lbrack\lbrack\mathbf{S}_1\rbrack\rbrack\bullet \cdots\bullet \lbrack\lbrack \mathbf{S}_{\ell}\rbrack\rbrack= \sum_\pi(-1)^{m-\operatorname{len}(\pi)}\lbrack\mathbf{S}\rbrack,$$
	where $m=k_1+\cdots+k_\ell$ and $\mathbf{S}$ runs over all quasishuffles of $\mathbf{S}_1,\ldots, \mathbf{S}_\ell$
\end{cor}

	\begin{cor}\label{cor:treeSetpartition}
		Let $\{S_1,\ldots, S_k\}$ be a collection of disjoint nonempty subsets of $\{1,\ldots, n\}$.  Let $\mathcal{T}= \{(i_1,j_1),\ldots, (i_{k-1},j_{k-1})\}$ be a directed tree, where $i_a,j_a\in\{1,\ldots, k\}$ with $i_a\not=j_a$.   We have
	$$\lbrack\lbrack S_{i_1},S_{j_1}\rbrack\rbrack\bullet\cdots\bullet \lbrack\lbrack S_{i_{k-1}},S_{j_{k-1}}\rbrack\rbrack=\sum_{\pi:p_{i_a} < p_{j_a}}(-1)^{k-\operatorname{len}(\pi)}\lbrack\pi\rbrack.$$
\end{cor}

\section{Graduated functions of blades}

\begin{defn}
	
	To each ordered set partition $\mathbf{S}=(S_1,\ldots, S_k)$ we assign a piecewise-constant function
	$$\lbrack (S_1,\ldots, S_k)\rbrack = \lbrack\lbrack S_1,S_2,\ldots, S_k\rbrack\rbrack + \lbrack\lbrack S_2,\ldots, S_k,S_1\rbrack\rbrack+\cdots +\lbrack\lbrack S_k,S_1,\ldots, S_{k-1}\rbrack\rbrack.$$
	Denote by $\hat{\mathcal{B}}^n$ the $\mathbb{Q}$-linear span of all such functions, as $(S_1,\ldots, S_k)$ varies over all standard ordered set partitions of $\{1,\ldots, n\}$.
\end{defn}
As the $k$ cyclic block rotations of the standard plate are intersecting, the set of possible values of the sum of their characteristic functions coincides with the set $\{1,2,\ldots, k\}$ and therefore $\lbrack (S_1,\ldots, S_k)\rbrack$ is \textit{not} itself a characteristic function, see Figure \ref{fig:blade-inverse-image-3-coords-2}.  This justifies our new term \textit{graduated.}

\begin{figure}[h!]
	\centering
	\includegraphics[width=0.3\linewidth]{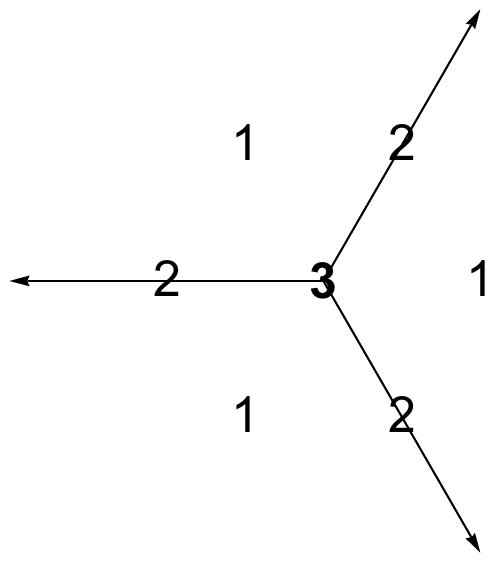}
	\caption{Level sets of the function $$\lbrack (1,2,3)\rbrack = 3 +2(\mu_{1,2} + \mu_{2,3} + \mu_{3,1})+ (\mu_{1,2,3}+\mu_{2,3,1} + \mu_{3,1,2}).$$}
	\label{fig:blade-inverse-image-3-coords-2}
\end{figure}

\begin{prop}\label{prop: linear independence cyclic plate sum}
	The set 
	$$\{\lbrack (S_1,\ldots, S_k)\rbrack:\ (S_1,\ldots, S_k) \text{ is a standard ordered set partition of }\{1,\ldots, n\}\}$$ 
	is linearly independent.
\end{prop}
\begin{proof}
	Following Proposition $12$ of \cite{EarlyCanonicalBasis}, noting that the duals of the cones $\lbrack\lbrack S_1,\ldots, S_k\rbrack\rbrack$ labeled by ordered set partitions are faces of Weyl chambers and therefore have disjoint interiors, using the involutive property for the duality valuation $\star$ on convex cones, it follows that the set of functions $\lbrack (S_1,\ldots, S_k)\rbrack$, which are cyclic sums over block rotations of plates, is linearly independent as well, since each plate occurs in precisely one of the cyclic rotations of a unique standard ordered set partition.
\end{proof}

Recall from \cite{EarlyCanonicalBasis} that an ordered set partition $(S_1,\ldots, S_k)$ of the set $\{1,\ldots, n\}$ is \textit{standard} if $S_1$ contains the minimal element in $S_1\cup\cdots\cup S_k$.  A (standard) composite ordered set partition is a set of disjoint ordered set partitions $\{\mathbf{S}_1,\ldots, \mathbf{S}_\ell\}$ of a set $\{1,\ldots, n\}$, if each $\mathbf{S}_i$ is a (standard) ordered set partition of a subset $T_i$ with $\{T_1,\ldots T_\ell\}$ an (unordered) set partition of $\{1,\ldots, n\}$.

\begin{thm}\label{thm: canonical basis cyclic sum plate}
	Let $T$ be a nonempty subset of $\{1,\ldots, n\}$ and let $\mathbf{S}_1,\ldots, \mathbf{S}_\ell$ be a (standard) composite ordered set partition of the set $\{1,\ldots, n\}\setminus T$.  Denote by $m= 1+\vert \mathbf{S}_1\vert + \cdots + \vert \mathbf{S}_\ell\vert$ the total number of blocks in the ordered set partitions $\mathbf{S}_1,\ldots,\mathbf{S}_\ell$, together with the one block $T$.
	
	Then we have
	\begin{eqnarray*}
		\lbrack (T,\mathbf{S}_1)\rbrack\bullet \cdots \bullet \lbrack  (T,\mathbf{S}_\ell)\rbrack & = & \sum_{\mathbf{U}}(-1)^{m-\vert\mathbf{U}\vert }\lbrack (\mathbf{U})\rbrack,
	\end{eqnarray*}
	where the sum is over all standard ordered set partitions $\mathbf{U}=(U_1,\ldots, U_k)$, say, such that $(U_1,\ldots, U_k)$ is a quasi-shuffle of the ordered set partitions
	$$(T,\mathbf{S}_1),\ldots, (T,\mathbf{S}_\ell),$$
	so that in particular $T\subseteq U_1$.
\end{thm}

\begin{proof}
	In the case when $T=\{1,\ldots, n\}$ there is nothing to prove.  So assume that $T$ is nonempty and proper.

	We shall apply Corollary \ref{cor:treeSetpartition} above to express each characteristic function in the expansion of 
	\begin{eqnarray*}
		\lbrack (T,\mathbf{S}_1)\rbrack\bullet \cdots \bullet \lbrack  (T,\mathbf{S}_\ell)\rbrack
	\end{eqnarray*}
	as a signed sum of characteristic functions labeled by ordered set partitions of $\{1,\ldots, n\}$.  By definition we have 
	\begin{eqnarray*}
		\lbrack (T,\mathbf{S}_1)\rbrack\bullet \cdots \bullet \lbrack  (T,\mathbf{S}_\ell)\rbrack & = & \left(\sum_{j}\lbrack\lbrack T,\mathbf{S}_1\rbrack\rbrack^{(j)}\right)\bullet \cdots \bullet \left(\sum_{j}\lbrack\lbrack T,\mathbf{S}_\ell\rbrack\rbrack^{(j)}\right),
	\end{eqnarray*}
	where the superscript $(j)$ denotes the $j^\text{th}$ block rotation,
	$$\lbrack\lbrack V_1,\ldots, V_a\rbrack\rbrack^{(j)} = \lbrack\lbrack V_j,V_{j+1},\ldots, V_{j-1}\rbrack\rbrack$$
	for $(V_1,\ldots, V_a)$ any ordered set partition.  This expands as 
	\begin{eqnarray*}
		\lbrack (T,\mathbf{S}_1)\rbrack\bullet \cdots \bullet \lbrack  (T,\mathbf{S}_\ell)\rbrack & = & \sum_{j_1,\ldots, j_\ell}\left(\lbrack\lbrack T,\mathbf{S}_1\rbrack\rbrack^{(j_1)}\bullet \cdots \bullet \lbrack\lbrack T,\mathbf{S}_\ell\rbrack\rbrack^{(j_\ell)}\right),
	\end{eqnarray*}
	where we remark now that $\lbrack\lbrack T,\mathbf{S}_i\rbrack\rbrack^{(j_i)}$ is standard only when $j_i=1$.  Applying Corollary \ref{cor:treeSetpartition} to expand each summand as a signed sum of characteristic functions of plates, we obtain
	
	\begin{eqnarray*}
		\lbrack\lbrack T,\mathbf{S}_1\rbrack\rbrack^{(j_1)}\bullet \cdots \bullet \lbrack\lbrack T,\mathbf{S}_\ell\rbrack\rbrack^{(j_\ell)} & = & \sum_{\mathbf{U}^{(j_1,\ldots, j_\ell)}}(-1)^{m-\vert \mathbf{U}^{(j_1,\ldots, j_\ell)}\vert} \lbrack\lbrack \mathbf{U}^{(j_1,\ldots, j_\ell)}\rbrack\rbrack
	\end{eqnarray*}
	and summing both sides over all $(j_1,\ldots, j_\ell)$, after factorization we recover 
	\begin{eqnarray*}
		\lbrack (T,\mathbf{S}_1)\rbrack\bullet \cdots \bullet \lbrack  (T,\mathbf{S}_\ell)\rbrack & = & \sum_{(j_1,\ldots, j_\ell)}\lbrack\lbrack T,\mathbf{S}_1\rbrack\rbrack^{(j_1)}\bullet \cdots \bullet \lbrack\lbrack T,\mathbf{S}_\ell\rbrack\rbrack^{(j_\ell)}\\
		& = & \sum_{(j_1,\ldots, j_\ell)} \sum_{\mathbf{U}^{(j_1,\ldots, j_\ell)}}(-1)^{m-\vert \mathbf{U}^{(j_1,\ldots, j_\ell)}\vert} \lbrack\lbrack \mathbf{U}^{(j_1,\ldots, j_\ell)}\rbrack\rbrack\\
		& = &\sum_{\mathbf{U}}(-1)^{m-\vert\mathbf{U}\vert }\lbrack (\mathbf{U})\rbrack.
	\end{eqnarray*}
\end{proof}
\begin{rem}
	It is tempting to note in Theorem \ref{thm: canonical basis cyclic sum plate} the similarity in the case when $T$ is empty and all blocks $(S_1,\ldots, S_n)$ are singlets, with the so-called \textit{Generalized BDDK relations for the full one-loop Parke-Taylor factors} from \cite{Gomez2018}.  See also Example \ref{example: higher codim blade shuffle} below.
\end{rem}

\begin{example}\label{example: higher codim blade shuffle}
	By Lemma \ref{thm: canonical basis cyclic sum plate}, the characteristic function $\Gamma_{1,2,3}\Gamma_{1,4,5}$ has the following expansion:
	\begin{eqnarray*}
		& & \lbrack (1,2,3)\rbrack \bullet \lbrack (1,4,5)\rbrack \\
		& = &  \lbrack(1,2,3,4,5)\rbrack+\lbrack(1,2,4,3,5)\rbrack+\lbrack(1,2,4,5,3)\rbrack+\lbrack(1,4,2,3,5)\rbrack+\lbrack(1,4,2,5,3)\rbrack+\lbrack(1,4,5,2,3)\rbrack\\
		& - & \lbrack(1,2,4,35)\rbrack-\lbrack(1,2,34,5)\rbrack-\lbrack(1,4,2,35)\rbrack-\lbrack(1,4,25,3)\rbrack-\lbrack(1,24,3,5)\rbrack - \lbrack(1,24,5,3)\rbrack\\
		& + & \lbrack(1,24,35)\rbrack.
	\end{eqnarray*}
	In the functional representation $\lbrack\lbrack i,j\rbrack\rbrack \mapsto \frac{1}{1-x_i/x_j}$ (See \cite{EarlyCanonicalBasis} for well-definedness) we have
	$$\lbrack (1,2,3,4,5)\rbrack \mapsto \frac{x_2 x_3 x_4 x_5}{\left(x_1-x_2\right) \left(x_2-x_3\right) \left(x_3-x_4\right) \left(x_4-x_5\right)}+\frac{ x_3 x_4 x_5x_1}{\left(x_2-x_3\right) \left(x_3-x_4\right) \left(x_4-x_5\right) \left(x_5-x_1\right)}+\cdots $$
	$$=\frac{x_3 x_4 x_5 x_1^2+x_2 x_3 x_4^2 x_1+x_2 x_3^2 x_5 x_1+x_2^2 x_4 x_5 x_1-5 x_2 x_3 x_4 x_5 x_1+x_2 x_3 x_4 x_5^2}{\left(x_1-x_2\right) \left(x_2-x_3\right) \left(x_3-x_4\right)  \left(x_4-x_5\right)\left(x_5-x_1\right)},$$
	and one can verify that by partial fraction identities, modulo non-pointed cones (these are labeled by ordered set partitions having a block of size bigger than 1) we have
	\begin{eqnarray*}
		& & \lbrack (1,2,3)\rbrack \bullet \lbrack (1,4,5)\rbrack\\
		&  = &  \lbrack(1,2,3,4,5)\rbrack+\lbrack(1,2,4,3,5)\rbrack+\lbrack(1,2,4,5,3)\rbrack+\lbrack(1,4,2,3,5)\rbrack+\lbrack(1,4,2,5,3)\rbrack+\lbrack(1,4,5,2,3)\rbrack\\
		& \mapsto  & \frac{x_3 x_4 x_5 x_1^2+x_2 x_3 x_4^2 x_1+x_2 x_3^2 x_5 x_1+x_2^2 x_4 x_5 x_1-5 x_2 x_3 x_4 x_5 x_1+x_2 x_3 x_4 x_5^2}{\left(x_1-x_2\right) \left(x_2-x_3\right) \left(x_3-x_4\right)  \left(x_4-x_5\right)\left(x_5-x_1\right)} + \cdots  \\
		&  = & \frac{\left(x_3 x_1^2+x_2^2 x_1-3 x_2 x_3 x_1+x_2 x_3^2\right) \left(x_5 x_1^2+x_4^2 x_1-3 x_4 x_5 x_1+x_4 x_5^2\right)}{\left(x_1-x_2\right) \left(x_2-x_3\right) \left(x_3-x_1\right)\left(x_1-x_4\right) \left(x_4-x_5\right)\left(x_5-x_1\right) }.
	\end{eqnarray*}
	Substituting $x_i= e^{-\varepsilon y_i}$ and series expanding in $\varepsilon$, truncating at order $\mathcal{O}(\varepsilon^1)$ we obtain a sum of elementary symmetric functions:
	\begin{eqnarray*}
		&&\lbrack (1,2,3,4,5)\rbrack \mapsto \sum_{j=0}^3\frac{1}{(4-j)!}e_j\left(\frac{1}{y_1-y_2},\ldots, \frac{1}{y_5-y_1}\right)\varepsilon^{-j}+\mathcal{O}(\varepsilon),
	\end{eqnarray*}
	where $\mathcal{O}(\varepsilon)$ contains powers of $\varepsilon^p$ for $p\ge 1$, and where the right-hand side of the shuffle identity becomes
	\begin{eqnarray*}
		&& \frac{\left(-y_1^2+y_2 y_1+y_3 y_1-y_2^2-y_3^2+y_2 y_3\right) \left(-y_1^2+y_4 y_1+y_5 y_1-y_4^2-y_5^2+y_4 y_5\right)}{\left(y_1-y_2\right) \left(y_2-y_3\right) \left(y_3-y_1\right) \left(y_1-y_4\right) \left(y_4-y_5\right) \left(y_5-y_1\right)}\varepsilon^{-2}\\
		& + &  \frac{1}{2}\left(\frac{1}{y_1-y_2}+\frac{1}{y_2-y_3}+\frac{1}{y_3-y_1}+\frac{1}{y_1-y_4}+\frac{1}{y_4-y_5}+\frac{1}{y_5-y_1}\right)\varepsilon^{-1}\\
		&& +\frac{1}{4} +\mathcal{O}(\varepsilon).
	\end{eqnarray*}
	Warning: in the expansions of the functional representation built from
	$$\lbrack\lbrack i,j\rbrack \mapsto \frac{1}{1-x_i/x_j}$$
	above, we have drastically truncated an infinite series; as we have seen that elementary symmetric functions appear in the expansion of the characteristic function of a blade, it is not surprising that it appears for generating functions.  However for larger $n$ the structure of the expansion is considerably more complex and interesting, due to the nature of the well-known function
	$$\frac{1}{1-e^{-z}},$$
	see for example \cite{Berline}.
	We leave such questions to future work.
	
	Note that in the functional representation of type $\lbrack\lbrack i,j\rbrack\rbrack \mapsto \frac{1}{x_i-x_j}$, since characteristic functions of \textit{both} non-pointed cones \textit{and} higher codimension cones  are in the kernel of the Laplace transform valuation, we have 
	$$\lbrack (1,2,3,4,5)\rbrack \mapsto \frac{1}{(x_1-x_2)(x_2-x_3)\cdots (x_4-x_5)}+\frac{1}{(x_2-x_3)\cdots (x_4-x_5)(x_5-x_1)}+\cdots=0.$$
\end{example}
\begin{rem}
	Consider now one of the other functional representations mentioned in \cite{EarlyCanonicalBasis}, 
	\begin{eqnarray}\label{eq: bigshuffle higher codim blade}
	\lbrack\lbrack S_1,\ldots, S_k\rbrack &\mapsto &\prod_{i=1}^k\frac{1}{1-\prod_{j\in S_i}x_j},
	\end{eqnarray}
	for example
	$$
	\lbrack\lbrack 1,2,3,4,5\rbrack\rbrack \mapsto \frac{1}{\left(1-x_1\right) \left(1-x_1 x_2\right) \left(1-x_1 x_2 x_3\right) \left(1-x_1 x_2 x_3 x_4\right) \left(1-x_1 x_2 x_3 x_4 x_5\right)}.$$
	The functional representation of Equation \eqref{eq: bigshuffle higher codim blade} is quite complicated and the numerator for the expansion of $\lbrack (1,2,3)\rbrack \bullet \lbrack (1,4,5)\rbrack$ appears not to factor nicely into the product of two irreducible polynomials similar to 
	$$\frac{\left(x_3 x_1^2+x_2^2 x_1-3 x_2 x_3 x_1+x_2 x_3^2\right) \left(x_5 x_1^2+x_4^2 x_1-3 x_4 x_5 x_1+x_4 x_5^2\right)}{\left(x_1-x_2\right) \left(x_2-x_3\right) \left(x_3-x_1\right)\left(x_1-x_4\right) \left(x_4-x_5\right)\left(x_5-x_1\right) },$$
	but rather the numerator is a monstrous polynomial with 14781 monomials which likely doesn't factor at all. The complexity here likely arises because the blades $((1,2,3))$ and $((1,4,5))$ are not in orthogonal subspaces.  On the other hand, the expansion using Equation \eqref{eq: bigshuffle higher codim blade} of for example $\lbrack (1,2,3)\rbrack \bullet \lbrack(4,5)\rbrack $ and $\lbrack (1,2,3)\rbrack \bullet \lbrack (4,5,6)\rbrack$ both can be seen to have numerators which factor as products of two irreducible polynomials.
	
	See Examples 47 and 48 in \cite{EarlyCanonicalBasis} for work with this functional representation for generalized permutohedral cones that are encoded by directed trees.
	
	The leading order term, the coefficient of $\varepsilon^{-5}$ in its analogous approximation using $x_i = e^{-\varepsilon y_i}$,
	$$\lbrack\lbrack 1,2,3,4,5\rbrack\rbrack \mapsto \frac{1}{y_1 \left(y_1+y_2\right) \left(y_1+y_2+y_3\right) \left(y_1+y_2+y_3+y_4\right) \left(y_1+y_2+y_3+y_4+y_5\right)}\varepsilon^{-5}+\mathcal{O}(\varepsilon^{-4}),$$
	is also quite interesting, but it also appears to lack a meaningful factorization property for convolutions involving non-orthogonal subspaces.
\end{rem}
\begin{example}\label{example: U(1) decoupling}
Recall the definition of the Parke-Taylor factor
$$PT(i_1,\ldots, i_n) = \frac{1}{(x_{i_1}-x_{i_2})(x_{i_2} - x_{i_3})\cdots (x_{i_n}-x_{i_1})},$$
where $x_1,\ldots, x_n$ are complex variables, see \cite{ParkeTaylor}.  There is a so-called $U(1)$-decoupling identity,
$$PT(i_1,i_2\ldots,i_n) + PT(i_1,i_3\ldots,i_n,i_2)  + \cdots +PT(i_1,i_n,i_2,\ldots,i_{n-1})=0,$$
which has an analog for blades.  Let us illustrate what it looks like in the first nontrivial case.

In the functional representation 
$$\lbrack\lbrack S_1,\ldots, S_k\rbrack\rbrack \mapsto \prod_{i=1}^k\frac{1}{1-\prod_{j\in S_1\cup\cdots\cup S_i}x_j},$$
for example
$$
\lbrack\lbrack 1,23,4\rbrack\rbrack \mapsto \frac{1}{\left(1-x_1\right) \left(1-x_1 x_2 x_3\right) \left(1-x_1 x_2 x_3 x_4\right)},$$
we have
\begin{eqnarray*}
	\lbrack (2,3,4)\rbrack & \mapsto &  \frac{1}{\left(1-x_4\right) \left(1-x_2 x_4\right) \left(1-x_2 x_3 x_4\right)}+\frac{1}{\left(1-x_3\right) \left(1-x_3 x_4\right) \left(1-x_2 x_3 x_4\right)}\\
	& + & \frac{1}{\left(1-x_2\right) \left(1-x_2 x_3\right) \left(1-x_2 x_3 x_4\right)} \\
	& = & \frac{1}{\left(1-x_2\right) \left(1-x_2 x_3\right) \left(1-x_4\right)}+\frac{1}{\left(1-x_3\right) \left(1-x_4\right) \left(1-x_4 x_2\right)}\\
	& + & \frac{1}{\left(1-x_2\right) \left(1-x_3\right) \left(1-x_3 x_4\right)}+\frac{1}{1-x_2 x_3 x_4}\\
	& - & \frac{1}{\left(1-x_2\right) \left(1-x_3\right) \left(1-x_4\right)}.
\end{eqnarray*}

Recall that 
\begin{eqnarray*}
	\lbrack (1,2,3,4)\rbrack & = & \lbrack\lbrack 1,2,3,4\rbrack\rbrack + \lbrack\lbrack 2,3,4,1\rbrack\rbrack + \lbrack\lbrack 3,4,1,2\rbrack\rbrack + \lbrack\lbrack 4,1,2,3\rbrack\rbrack.
\end{eqnarray*}
Summing over all standard ordered set partitions which contain $(1),(2,3,4)$ as cyclic subwords, after many partial fraction identities we obtain the same, but multiplied by $\frac{1}{1-x_1}$:
\begin{eqnarray*}
	&&\lbrack(1,2,3,4)\rbrack + \lbrack(1,3,4,2)\rbrack + \lbrack(1,4,2,3)\rbrack - \lbrack(12,3,4)\rbrack - \lbrack(13,4,2)\rbrack - \lbrack(14,2,3)\rbrack\\
	& \mapsto & \frac{1}{(1-x_1)(1-x_2)(1-x_2x_3)(1-x_4)} +  \frac{1}{(1-x_1)(1-x_3)(1-x_4)(1-x_4x_2)}\\
	& + & \frac{1}{(1-x_1)(1-x_3)(1-x_3x_4)(1-x_2)} + \frac{1}{(1-x_1)(1-x_2x_3x_4)}\\
	& - &\frac{1}{(1-x_1)\left(1-x_2\right) \left(1-x_3\right) \left(1-x_4\right)}.
\end{eqnarray*}
\end{example}

\begin{defn}
	We call an ordered set partition $(S_1,S_2,\ldots, S_k)$ of $\{1,\ldots, n\}$ \textit{2-standard} if both $(S_1,S_2,\ldots ,S_k)$ and $(S_2,\ldots, S_k)$ are standard ordered set partitions of respectively $\{1,\ldots, n\}$ and $\{1,\ldots, n\}\setminus S_1$.
	
	The set $\{(T,\mathbf{S}_1), \ldots,(T,\mathbf{S}_\ell)\}$ is a 2-standard composite ordered set partition of $\{1,\ldots, n\}$ if each $(T,\mathbf{S}_i)$ is 2-standard.
\end{defn}
In other words, for an ordered set partition of $\{1,\ldots, n\}$ the condition is that $S_1$ contains $1$, and $S_2$ contains the minimal element in the complement of $S_1$ in $\{1,\ldots, n\}$.  For example, $(1,2,4,3,5)$ and $(15,2,4,3)$ are 2-standard, but $(125,4,3)$ is not.

In what follows, after proving in Theorem \ref{thm: cyclic sum canonical basis} that it spans and is linearly independent, we shall call the set
$$\{\lbrack (T,\mathbf{S}_1)\rbrack\bullet \cdots \bullet \lbrack  (T,\mathbf{S}_\ell)\rbrack:  \text{ each }(T,\mathbf{S}_i)\text{ is a 2-standard ordered set partition}\}.$$
the \textit{canonical} basis for $\hat{\mathcal{B}}^n$.  If the set $\{(T,\mathbf{S}_1),\ldots, (T,\mathbf{S}_\ell)\}$ contains only 2-standard ordered set partitions, then it shall be called a \textit{blade-canonical} composite ordered set partition.

Following the procedure of \cite{EarlyCanonicalBasis}, we define an endomorphism on $\hat{\mathcal{B}}^n$, taking the given basis of functions $\lbrack(S_1,\ldots S_k)\rbrack$ labeled by standard ordered set partitions to the canonical one, labeled by 2-standard composite ordered set partitions.  This will prove linear independence for the canonical basis.

We now define a bijection $\mathcal{U}_{\mathcal{B}}$ between standard ordered set partitions and 2-standard composite ordered set partitions.  It is closely related to the bijection from \cite{EarlyCanonicalBasis} between ordered set partitions and standard composite ordered set partitions.

Let $(S_1,\ldots, S_k)$ be a standard (i.e., $1\in S_1$) ordered set partition of $\{1,\ldots, n\}$.  Setting $T=S_1$, we apply the bijection in \cite{EarlyCanonicalBasis} to the ordered set partition $(S_2,\ldots, S_k)$ of $\{1,\ldots, n\}\setminus S_1$.  Namely, let us first recall the algorithm from Proposition 25 in \cite{EarlyCanonicalBasis}.  Define a subset 
$$\{i_1,\ldots, i_m\}\subseteq \{2,\ldots, k\}$$
with $i_1>i_2>\cdots>i_m$, such that $S_{i_1}$ contains the smallest label in the set $S_{i_1}\cup S_{i_1+1}\cup \cdots\cup S_k$, and in general $S_{i_p}$ contains the smallest label in the set $S_{i_p}\cup S_{i_p+1}\cup\cdots \cup S_{i_{p-1}-1}$.  

Now set 
$$\mathcal{U}_{\mathcal{B}}(S_1,\ldots, S_k) = \{(S_1,\mathbf{U}_{1}),(S_1,\mathbf{U}_{2})\ldots, (S_1,\mathbf{U}_{m})\},$$
where $\mathbf{U}_{a} = (S_{i_a},S_{i_{a}+1},\ldots, S_{i_{a-1}-1})$.

Recall from \cite{EarlyCanonicalBasis} the lexicographic partial order on ordered set partitions: given two ordered set partitions $(S_1,\ldots, S_k)$ and $(T_1,\ldots, T_\ell)$ of $\{1,\ldots, n\}$, define position vectors $(p_1,\ldots, p_n)$ and $(q_1,\ldots, q_n)$ whenever $i\in S_{p_i}$ and $j\in T_{q_j}$. We say that in the lexicographic order $(S_1,\ldots, S_k) \prec (T_1,\ldots, T_\ell)$ if the sequence $(p_1,\ldots, p_n)$ is lexicographically smaller than $(q_1,\ldots, q_n)$.

\begin{rem}
	It is easy to see that when $1\in S_1$ then the ordered set partition $(S_1,\ldots, S_k)$ of $\{1,\ldots, n\}$ is lexicographically minimal among its cyclic block rotations.  
\end{rem}
\begin{example}
	For standard ordered set partitions we have, in increasing lexicographic order, standard ordered set partitions together with their position vectors,
	\begin{eqnarray*}
		(12345) & \Leftrightarrow & \mathbf{p} = (1,1,1,1,1)\\
		(1,234,5,67) & \Leftrightarrow &\mathbf{p} = (1,2,2,2,3,4,4)\\
		(1,23,4,5,67) & \Leftrightarrow &\mathbf{p} = (1,2,2,3,3,4,4)\\
		(1,3,24,5,67) & \Leftrightarrow &\mathbf{p} = (1,3,2,3,3,4,4)\\
		(1,5,4,3,2) & \Leftrightarrow &\mathbf{p} = (1,5,4,3,2)\\
		(1\ 5,11,4,12,10,2,8\ 9,3,6\ 7) & \Leftrightarrow & \mathbf{p} = (1,6,8,3,1,9,9,7,7,5,2,4)
	\end{eqnarray*}
	and 
	\begin{eqnarray*}
		\mathcal{U}_{\mathcal{B}}(12345) &  =&  \{(12345)\}\\
		\mathcal{U}_\mathcal{B}(1,234,5,67)  & = &  \{(1,234,5,67)\}\\
		\mathcal{U}_{\mathcal{B}}(1,3,24,5,67) & = & \{(1,24,5,67),(1,3)\}\\
		\mathcal{U}_{\mathcal{B}}(1,5,4,3,2) & =&  \{(1,2),(1,3),(1,4),(1,5)\}\\
\mathcal{U}_\mathcal{B}(1\ 5,11,4,12,10,2,8\ 9,3,6\ 7) & =&  \{(1\ 5,2,8\ 9,3,6\ 7),(1\ 5,4,12,10),(1\ 5,11)\}.
	\end{eqnarray*}
	
\end{example}

Theorem \ref{thm: cyclic sum canonical basis} finally establishes linear independence for the canonical basis for graduated functions of blades; the main component of proof, lexicographic maximality of $\mathbf{S}$ among the quasi-shuffles of $\mathcal{U}_\mathcal{B}(\mathbf{S})$, is illustrated in Figure \ref{fig:5-7-2019-canonical-basis-permutations}.

\begin{thm}\label{thm: cyclic sum canonical basis}
	The linear map induced by the bijection $\mathcal{U}_\mathcal{B}$ is invertible.  The set 
	$$\{\lbrack (T,\mathbf{U}_1)\rbrack\bullet \cdots \bullet \lbrack  (T,\mathbf{U}_\ell)\rbrack:  \text{ each }(T,\mathbf{U}_i)\text{ is a 2-standard ordered set partition}\}.$$
	is linearly independent (and hence is a basis) for the space $\hat{\mathcal{B}}^n$.
\end{thm}

\begin{proof}
	The proof that the bijection $\mathcal{U}_\mathcal{B}$ is upper-unitriangular with respect to the lexicographic order follows the same procedure which was used in the proof of Theorem 28 in \cite{EarlyCanonicalBasis}.  
	
	Let us suppose for clarity that we are given an ordered set partition of $\{1,\ldots, n\}$ with only singlet blocks, $\mathbf{S} = (b_1,\ldots, b_n)$ with $b_1=1$; the case when blocks are present which are of size $\ge2$ will then follow immediately.
	
	Then we obtain standard ordered set partitions (so, all blocks are singlets and the minimal entry of each $C_i$ is in the first position) $C_1,\ldots, C_m$ such that 	
	$$\mathcal{U}_\mathcal{B}(\mathbf{S}) = \{(1,C_1),(1,C_2),\ldots, (1,C_m)\},$$
	so that $\mathbf{S}$ is recovered as the concatenation $(1,C_m,C_{m-1},\ldots, C_1)$.  Evidently, in the set of quasi-shuffles of $\mathcal{U}_\mathcal{B}(\mathbf{S})$, the concatenation $(1,C_1,C_2,\ldots, C_m)$ (where the label $p_2=2$, with small minima pushed to the left and large minima pushed to the right) is lexicographically smallest, and in particular, the concatenation $(1,C_m,C_{m-1},\ldots, C_1)=\mathbf{S}$ is lexicographically largest (with small minima pushed to the right and large minima pushed to the left, so for example the index $2$ is pushed as far right as possible and hence $p_2$ as large as possible).  A similar argument implies the same for ordered set partitions with non-singlet blocks.
	
	Thus, if $(S_1,\ldots, S_k)$ is a standard ordered set partition of $\{1,\ldots, n\}$ and we apply Theorem \ref{thm: canonical basis cyclic sum plate} to expand $\mathcal{U}_\mathcal{B}\left(\lbrack (S_1,\ldots, S_k)\rbrack\right)$, then every blade in the expansion is labeled by a standard ordered set partition that is lexicographically strictly smaller than $(S_1,\ldots, S_k)$; since $\lbrack (S_1,\ldots, S_k)\rbrack$ itself is a quasi-shuffle of $\mathcal{U}_\mathcal{B}\left(\lbrack (S_1,\ldots, S_k)\rbrack\right)$, it follows that the matrix for $\mathcal{U}_\mathcal{B}$ is upper unitriangular, hence invertible.
\end{proof}

\begin{example}
	For $\mathbf{S}=(1\ 9,12,7,11,10,5,8,6\ 13,2,4,3\ 14)$ we have
	$$\mathcal{U}_\mathcal{B}(\mathbf{S}) = \{(1\ 9,2,4,3\ 14),(1\ 9,5,8,6\ 13),(1\ 9,7,11,10),(1\ 9,12)\},$$
	so 
	$$\mathcal{U}_\mathcal{B}\lbrack\mathbf{S}\rbrack = \lbrack (1\ 9,2,4,3\ 14)\rbrack \bullet \lbrack (1\ 9,5,8,6\ 13)\rbrack \bullet \lbrack(1\ 9,7,11,10)\rbrack \bullet \lbrack(1\ 9,12)\rbrack.$$
	The minimal quasi-shuffle has the position vector 
	$$\mathbf{p}=(1,2,4,3,5,7,8,6,1,10,9,11,7,4)$$
	while the maximal quasi-shuffle has the position vector 
	$$ \mathbf{p} = (1,9,11,10,6,8,3,7,1,5,4,2,8,11),$$
	\begin{figure}[h!]
		\centering
		\includegraphics[width=0.85\linewidth]{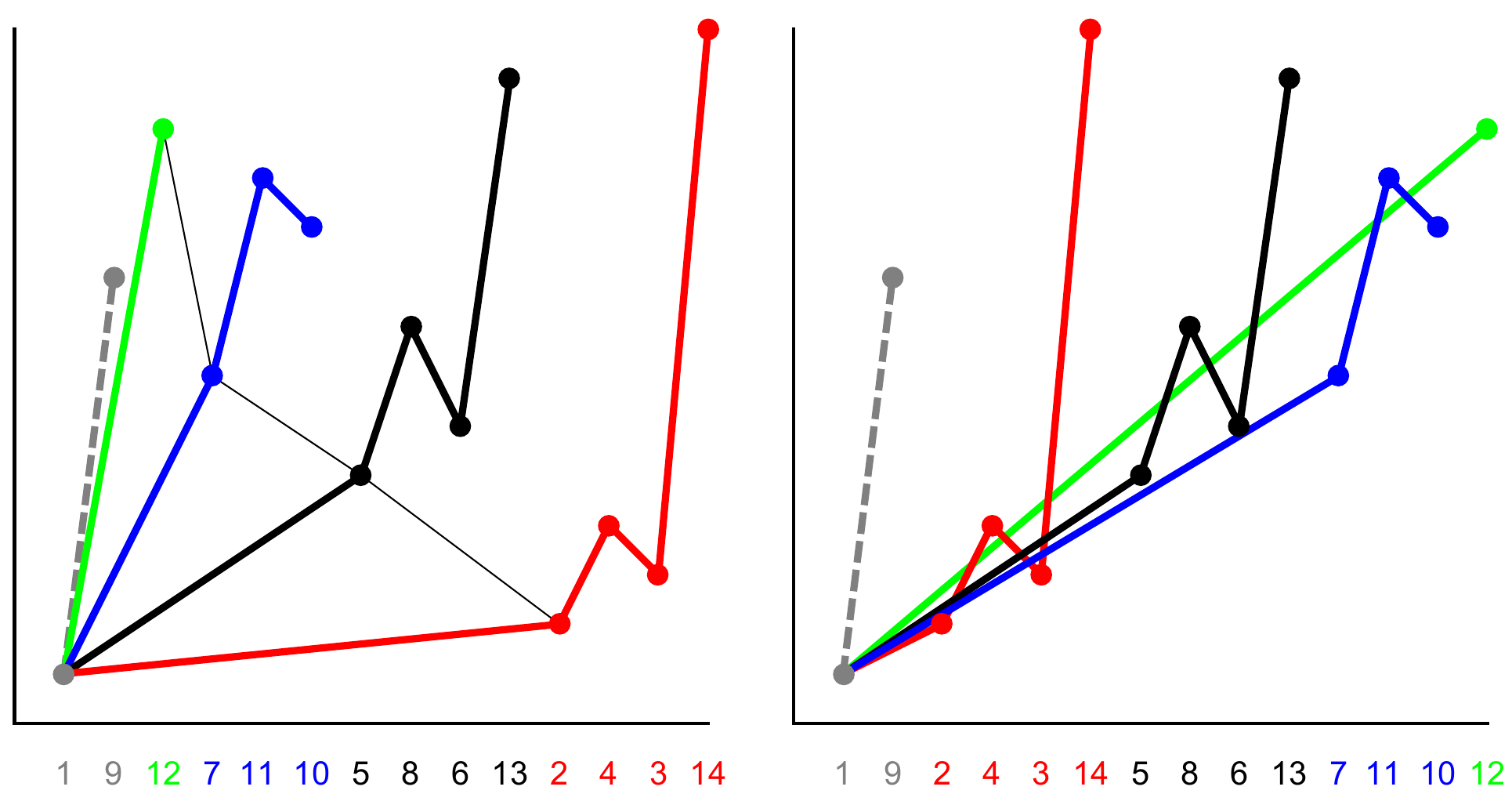}
		\caption{The extremal (quasi-)shuffles of the set obtained from the bijection, $\mathcal{U}_\mathcal{B}(1\ 9,12,7,11,10,5,8,6\ 13,2,4,3\ 14)$. Left, lexicographically maximal; right, lexicographically minimal.  Here the blocks have been flattened with their respective labels ordered increasingly.  Descending minima are connected by a thin line (left).  Best viewed in color.}
		\label{fig:5-7-2019-canonical-basis-permutations}
	\end{figure}
	
\end{example}

	\section{Canonical blades, graduated and characteristic: enumeration and a conjecture}\label{section: canonical basis enumeration}

Recall that 
$$\gamma_{i,j,k} = 1+ \mu_{i,j} + \mu_{j,k} + \mu_{k,i}=1+ (\lbrack\lbrack i,j\rbrack\rbrack-1) + (\lbrack\lbrack j,k\rbrack\rbrack-1) + (\lbrack\lbrack k,i\rbrack\rbrack-1)$$
and that $1_{ab}$ is the characteristic function of the one-dimensional subspace
$$\{t(e_a-e_b):t\in\mathbb{R}\}.$$

Note: in Theorem \ref{thm: canonical basis cyclic sum plate} we established the linear independence of the set of canonical \textit{graduated} blades $\lbrack (S_1,\ldots, S_k)\rbrack$.  In what follows we give the set of \textit{canonical} characteristic functions of blades $\Gamma_{S_1,\ldots, S_k}$.  We conjecture that these are linearly independent and span the same space as the graduated blades; however the proofs of these two assertions is left to future work.

\begin{example}
	We give the standard ordered set partition parametrization of the canonical basis.  Note in what follows that the characteristic function labeled by any set partitions with only singlets is the identity, $\lbrack(j)\rbrack = 1$.
	
	The canonical basis is graded by the dimension of the blade.
		
\begin{center}
	
\begin{tabular}{|c|c|c|c|}
	\hline 
	$\mathbf{S}$ &$\mathcal{U}_\mathcal{B}(\mathbf{S})$ & $\mathcal{U}_\mathcal{B}(\lbrack (\mathbf{S})\rbrack)$ & $\Gamma_\mathbf{S}$\\ 
\hline
(1,3,2) & \{(1,2),(1,3)\} & 1&1\\
\hline 
(1,2,3)& \{(1,2,3)\} & \lbrack (1,2,3)\rbrack &$\Gamma_{1,2,3} = \gamma_{1,2,3}$\\  
(12,3)& \{12,3\} & \lbrack (12,3)\rbrack& $1_{12}$\\  
(1,23)& \{1,23\} & \lbrack (1,23)\rbrack & $1_{23}$\\
(13,2) & \{13,2\} & \lbrack 13,2\rbrack& $ 1_{13}$\\
\hline
(123) & \{(123)\} & \lbrack (123)\rbrack& $1_{123}$\\
\hline
\end{tabular} 	
\end{center}

Above for instance 
$$\mathcal{U}_\mathcal{B}(\lbrack (1,3,2)\rbrack) = \lbrack (1,2,3)\rbrack + \lbrack (1,3,2)\rbrack - \lbrack (12)\rbrack - \lbrack (13)\rbrack - \lbrack (23)\rbrack + \lbrack (123)\rbrack = 1,$$
which is supported on the (zero-dimensional) subspace, the origin itself.

We list the canonical graduated and characteristic functions for $n=4$ below.

\begin{center}
	
	\begin{tabular}{|c|c|c|c|c|}
		\hline 
		$\mathbf{S}$ &$\mathcal{U}_\mathcal{B}(\mathbf{S})$ & $\mathcal{U}_\mathcal{B}(\lbrack (\mathbf{S})\rbrack)$ & $\Gamma_\mathbf{S}$\\ 
		\hline
		(1,4,3,2) & \{(1,2),(1,3),(1,4)\} & 1&1\\
		\hline
		(1,4,2,3) & \{(1,2,3),(1,4)\} & \lbrack (1,2,3)\rbrack&$\Gamma_{1,2,3}$\\
		(1,3,2,4) & \{(1,2,4),(1,3)\} & \lbrack (1,2,4)\rbrack &$\Gamma_{1,2,4}$\\
		(1,3,4,2) & \{(1,3,4),(1,2)\} & \lbrack (1,3,4)\rbrack& $ \Gamma_{1,3,4}$\\
		(12,4,3)  & \{(12,3),(12,4)\} & $\lbrack (12)\rbrack$ & $ 1_{12}$\\
		(13,4,2)  & \{(13,2),(13,4)\} & $\lbrack (13)\rbrack$ & $ 1_{13}$\\
		(14,3,2)  & \{(14,2),(14,3)\} & \lbrack (14)\rbrack & $ 1_{14}$\\
		(1,4,23)  & \{(1,4),(1,23)\} & \lbrack (23)\rbrack & $ 1_{23}$\\
		(1,3,24)  & \{(1,3),(1,24)\} & \lbrack (24)\rbrack & $ 1_{24}$\\
		(1,34,2)  & \{(1,34),(1,2)\} & \lbrack (34)\rbrack & $ 1_{34}$\\
		\hline 
		(1,2,3,4)& \{(1,2,3,4)\} & \lbrack (1,2,3,4)\rbrack &$\Gamma_{1,2,3,4} = \gamma_{1,2,3}\gamma_{1,3,4}$\\  
		(1,2,4,3)& \{(1,2,4,3)\} & \lbrack (1,2,4,3)\rbrack &$\Gamma_{1,2,4,3} = \gamma_{1,2,4}\gamma_{1,4,3}$\\  
		(1,2,34)& \{(1,2,34)\} & \lbrack (1,2,34)\rbrack &$\Gamma_{1,2,34} = \gamma_{1,2,3}1_{34}$\\
		(1,23,4)& \{(1,23,4)\} & \lbrack (1,23,4)\rbrack &$\Gamma_{1,23,4} = \gamma_{1,2,4}1_{23}$\\  
		(12,3,4)& \{(12,3,4)\} & \lbrack (12,3,4)\rbrack &$\Gamma_{12,3,4} = \gamma_{1,3,4}1_{12}$\\
		(1,24,3)& \{(1,24,3)\} & \lbrack (1,24,3)\rbrack &$\Gamma_{1,24,3} = \gamma_{1,2,3}1_{24}$\\
		(13,2,4)& \{(13,2,4)\} & \lbrack (13,2,4)\rbrack &$\Gamma_{13,2,4} = \gamma_{1,2,4}1_{13}$\\   
		(14,2,3)& \{(14,2,3)\} & \lbrack (14,2,3)\rbrack &$\Gamma_{14,2,3} = \gamma_{1,2,3}1_{14}$\\
		(12,34) & \{(12,34)\} & \lbrack (12,34) \rbrack & $1_{12}1_{34}$\\
		(13,24) & \{(13,24)\} & \lbrack (13,24) \rbrack & $1_{13}1_{24}$\\
		(14,23) & \{(14,23)\} & \lbrack (14,23) \rbrack & $1_{14}1_{23}$\\
		(123,4) & \{(123,4)\} & \lbrack (123,4) \rbrack & $1_{123}$\\
		(124,3) & \{(124,3)\} & \lbrack (124,3) \rbrack & $1_{124}$\\
		(134,2) & \{(134,2)\} & \lbrack (134,2) \rbrack & $1_{134}$\\
		(1,234) & \{(1,234)\} & \lbrack (1,234) \rbrack & $1_{234}$\\
		\hline
		(1234) & \{(1234)\} & \lbrack (1234)\rbrack& $1_{1234}$\\
		\hline
	\end{tabular} 	
\end{center}
Here for the expansion we have for instance
\begin{eqnarray*}
	\mathcal{U}_\mathcal{B}(\lbrack 14,3,2\rbrack) & = & (\lbrack 14,2,3\rbrack+ \lbrack 14,3,2\rbrack - \lbrack (124,3)\rbrack - \lbrack (124,3)\rbrack - \lbrack 134,2\rbrack + \lbrack (1234)\rbrack) \\
	& = & 1_{14}(\lbrack 1,2,3\rbrack+ \lbrack 1,3,2\rbrack - \lbrack (12,3)\rbrack - \lbrack (12,3)\rbrack - \lbrack 13,2\rbrack + \lbrack (123)\rbrack),
\end{eqnarray*}
where in the last line we have used the convolution product to factor out the characteristic function of the line $\mathbb{R}(e_1-e_4)$.

	Note that the total count is $1+9+15+1 = 26$, the necklace number for $n=4$, which counts the number of standard ordered set partitions of $\{1,2,3,4\}$.
	
	For $n=5$ the first 2-standard composite ordered set partitions appear.  The new characteristic functions are
	$$\Gamma_{1,2,3}\Gamma_{1,4,5},\ \Gamma_{1,2,4}\Gamma_{1,3,5},\ \Gamma_{1,2,5}\Gamma_{1,3,4},$$
	and the corresponding graduated functions of blades are (this time making the convolution product $\bullet$ explicit)
	\begin{eqnarray*}
			\mathcal{U}_\mathcal{B}(\lbrack (1,4,5,2,3)) & = & \lbrack (1,2,3)\rbrack \bullet \lbrack (1,4,5)\rbrack\\
			\mathcal{U}_\mathcal{B}(\lbrack (1,3,5,2,4)) & = & \lbrack (1,2,4)\rbrack \bullet \lbrack (1,3,5)\rbrack\\
			\mathcal{U}_\mathcal{B}(\lbrack (1,3,4,2,5)) & = & \lbrack (1,2,5)\rbrack \bullet \lbrack (1,3,4)\rbrack.
		\end{eqnarray*}

\end{example}

\begin{rem}
	In Corollary 31 of \cite{EarlyCanonicalBasis} we found for the count of number $T_{n,k}$ of characteristic functions of plates of dimension $n-k$ in the canonical plate basis, the formula
$$T_{n,k}=\sum_{i=1}^n S(n,i)s(i,k-1).$$
Rows $n=1,\ldots, 6$ are given below.
$$
\begin{array}{ccccccc}
1 &  &  &  &  &  &  \\
2 & 1 &  &  &  &  &  \\
6 & 6 & 1 &  &  &  &  \\
26 & 36 & 12 & 1 &  &  &  \\
150 & 250 & 120 & 20 & 1 &  &  \\
1082 & 2040 & 1230 & 300 & 30 & 1 &  \\
\end{array}
$$
\end{rem}

By counting elements in the canonical basis of $\hat{\mathcal{B}}^n$, it is easy to obtain the formula in Proposition \ref{prop: graded dimension blades}, by shifting one of the indices from the formula for the canonical plate basis, reflecting the 2-standardness condition.

\begin{prop}\label{prop: graded dimension blades}
	The number of blades of dimension $(n-k)$ in the canonical basis of graduated blades is
	$$T^B_{n,k}=\sum_{i=1}^n S(n,i)s(i-1,k-1),$$
	where $S(n,i)$ is the Stirling number of the second kind, and $s(i,k)$ is the (unsigned) Stirling number of the first kind. 
\end{prop}

For $n=2,3,\ldots, 7$ we have
$$\begin{array}{cccccc}
1 & \text{} & \text{} & \text{} & \text{} & \text{} \\
1 & 1 & \text{} & \text{} & \text{} & \text{} \\
1 & 4 & 1 & \text{} & \text{} & \text{} \\
1 & 15 & 9 & 1 & \text{} & \text{} \\
1 & 66 & 66 & 16 & 1 & \text{} \\
1 & 365 & 500 & 190 & 25 & 1 \\
\end{array}$$
Note that the rows sum (correctly) to the necklace numbers, which count ordered set partitions up to cyclic block rotation.

	The numbers that are easy to remember are as follows: the left column counts the occurrence of the whole space $1_{12\cdots n}$.  For the rightmost two diagonals, there is always the characteristic function of the unique point at the origin, and there are $(n-1)^2$ linearly-independent characteristic functions of blades each of dimension 1, corresponding to the $\binom{n-1}{2}$ tripod generators $\gamma_{ijk}$ as in Figure \ref{fig:blade3coordinates0} together with the $\binom{n}{2}$ characteristic functions of the one-dimensional subspaces, $1_{ab}$.

The generating functions for the first few diagonals of the triangle in Proposition \ref{prop: graded dimension blades} are
$$\frac{1}{1-x},\frac{x+1}{(1-x)^3},\frac{x^2+10 x+1}{(1-x)^5},\frac{x^3+59 x^2+59 x+1}{(1-x)^7},\frac{x^4+356 x^3+966 x^2+356 x+1}{(1-x)^9},$$
$$\frac{x^5+2517 x^4+12602 x^3+12602 x^2+2517 x+1}{(1-x)^{11}},$$
$$\frac{x^6+21246 x^5+161967 x^4+298852 x^3+161967 x^2+21246 x+1}{(1-x)^{13}},$$
and the coefficients of the numerators sum to the sequence 
$$1, 2, 12, 120, 1680, 30240,$$
which appears to be given by O.E.I.S. sequence A001813, $a(m) = \frac{(2m)!}{m!}$, \cite{oeis}.  Clearly, a combinatorial proof and explanation for this formula would be highly desirable.

A vector of real numbers $(b_1,\ldots, b_m)$ is symmetric if $b_i = b_{m+1-i}$.  It is unimodal if it there exists an index $i$ such that $b_1\le b_2\le\cdots\le b_i\ge b_{i+1} \ge \cdots \ge b_m$.

The following conjecture has been verified numerically in Mathematica through the 24th diagonal.

\begin{conjecture}\label{conjecture: higher codim blades generating function}
	The coefficients of the polynomial numerators of the generating functions for the diagonals of the array $T^B_{n,k}$ are 
	\begin{enumerate}
		\item symmetric,
		\item unimodal, and 
		\item they sum to the sequence $a(m) = \frac{(2m)!}{m!}$.
	\end{enumerate}
\end{conjecture}

\begin{rem}
	After formulating this conjecture, Donghyun Kim shared with us his proofs of (1) and (3); interestingly, the formulas involve the Eulerian numbers of the second kind, O.E.I.S. A008517.  See Appendix \ref{sec: Donghyun Kim proof}.
\end{rem}

The sequence $a(m) = \frac{(2m)!}{m!}$ itself is already suggestive; for example, it can expressed as $a(m)=(m+1)! \left(\frac{(2m)!}{(m+1)!(m!)}\right)$, which is $(m+1)!$ times the Catalan number $C_m$, suggesting the possibility of a relation to \textit{labeled} (rooted) binary trees (as noted in the comments in sequence A001813), with some additional statistic explaining the numerator coefficients for the generating series.  It seems plausible to ask about the possibility of a directly geometric and/or ring-theoretic interpretation of the sequence $a(m)$ and the numerator coefficients in terms of blades or other geometric or physical objects.  We leave this question to future research.

	\section{The graded cohomology ring of a configuration space}\label{sec: graded cohomology ring}

In what follows, we connect with joint work with V. Reiner in \cite{EarlyReiner} on the cohomology ring $H^\star (X_n)$, where $X_n$ is the configuration space of $n$ distinct points in $SU(2)$ modulo the diagonal action of $SU(2)$.  The presentation of the cohomology ring in Definition \ref{defn: cohomology ring SU2} suggests that this ring is analogous to a graded analog of the space of characteristic functions of blades labeled by nondegenerate ordered set partitions of $\{1,\ldots, n\}$, having $n$ singleton blocks.  We leave the detailed exploration of the connection to future work.

\begin{defn}\label{defn: cohomology ring SU2}
	Let $\mathcal{U}^n$ be the commutative algebra on $\binom{n}{2}$ generators $u_{ij}  =-u_{ji}$ with $i\not=j$, subject to the relations
	$$u_{ij}^2=0$$
	and
	$$u_{ij}u_{jk}+u_{jk}u_{ki}+u_{ki}u_{ij}=0.$$
\end{defn}

Define $v_{ijk} = u_{ij} + u_{jk} + u_{ki}$, and denote by $\mathcal{V}^n$ the subalgebra of $\mathcal{U}^n$ generated by the $v_{ijk}$.

\begin{thm}[\cite{EarlyReiner}, Theorem 3]
	We have an isomorphism of graded rings 
	$$\mathcal{V}^n\simeq H^\star(X_n).$$
\end{thm}

In the algebra $\mathcal{V}^n$ we have the very useful relation of Proposition \ref{prop:triplesRelation}.  It is interesting to compare its relative simplicity to the complexity of the proof of Theorem \ref{thm: blade flag factorization}.

\begin{prop}\label{prop:triplesRelation}
	In $\mathcal{U}^n$ we have the relation 
	$$(1+u_{i_1,i_2})\cdots(1+u_{i_k,i_1})
	= (1+v_{i_1,i_2,i_3})\cdots (1+v_{i_1,i_{k-1},i_{k}}),$$
	for any sequence $(i_1,\ldots, i_k)$ selected from $\{1,\ldots, n\}$ with $k\ge 2$ ($k=2$ being trivial),
	and additionally the identities
	\begin{eqnarray*}
		v_{ijk} & = & v_{jki}=-v_{ikj}\\
		v_{ijk}^2 & = & 0.
	\end{eqnarray*}
\end{prop}

\begin{proof}
	As $u_{ij}^2=0$ and 
	$$u_{ij}u_{jk}+u_{jk}u_{ki}+u_{ki}u_{ij}=0,$$
we have
	$$v_{ijk}^2 = u_{ij}^2+u_{jk}^2 + u_{ki}^2 +2\left(u_{ij}u_{jk}+u_{jk}u_{ki}+u_{ki}u_{ij}\right)=0.$$
	Now let us prove the first assertion by induction. 
	
	Since $u_{i,j}^2 =0$ we have 
	$$(1+u_{ij})(1+u_{ji}) = (1-u_{ij}^2)=1.$$
	
	Then,
	\begin{eqnarray*}
		(1+u_{i_1i_2})\cdots (1+u_{i_ki_1}) & = & (1+u_{i_1i_2})(1+u_{i_2i_3})\cdots (1+u_{i_{k-1}i_1})(1+u_{i_{k-1}i_k})(1-u_{i_{k-1}i_1})(1+u_{i_{k}i_1})\\
		& = & (1+u_{i_1i_2})(1+u_{i_2i_3})\cdots (1+u_{i_{k-1}i_1})(1+u_{i_1i_{k-1}}+u_{i_{k-1}i_k}+u_{i_ki_{1}})\\
		& = & (1+u_{i_1i_2})(1+u_{i_2i_3})\cdots (1+u_{i_{k-1}i_1})(1+v_{1,k-1,k})\\
		& = & (1+v_{i_1i_2i_3})(1+v_{i_1i_3i_4})\cdots (1+v_{i_1i_{k-1}i_k}).
	\end{eqnarray*}
	Here we have used the induction step, $u_{ij}^2=0$ and the Jacobi identity to express 
	\begin{eqnarray*}
		(1+u_{i_{k-1}i_k})(1-u_{i_{k-1}i_1})(1-u_{i_{k}i_1})& = & (1+u_{1i_{k-1}})(1+u_{i_{k-1}i_k})(1+u_{i_{k}i_1})\\
		& = & (1+u_{i_{1}i_{k-1}}+u_{i_{k-1}i_k}+u_{i_ki_1})\\
		& = & 1+v_{1i_{k-1}i_k}.
	\end{eqnarray*}
\end{proof}
Note that the identity is unchanged if we deform with a formal parameter $t$ to keep track of degree: replace $(1+u_{ij}) \mapsto (1+tu_{ij})$ and accordingly $(1+v_{ijk})\mapsto (1+tv_{ijk})$.

Then Proposition \ref{prop:triplesRelation} becomes
\begin{eqnarray*}
	(1+tu_{i_1i_2})(1+tu_{i_2i_3})\cdots (1+tu_{i_ki_1}) & = & (1+tv_{i_1i_2i_3})(1+tv_{i_1i_3i_4})\cdots (1+tv_{i_1i_{k-1}i_k}).
\end{eqnarray*}

\begin{cor}\label{cor:cyclicSumProdVanish}
	We have
	$$u_{i_1,i_2}\cdots u_{i_{k-1},i_k} + u_{i_2,i_3}\cdots u_{i_{k},i_1} + \cdots +u_{i_{k},i_1}\cdots u_{i_{k-2},i_{k-1}}=0$$
	and
	$$u_{i_1,i_2}\cdots u_{i_k,i_1}=0,$$
	for any sequence $\{i_1,\ldots, i_k\}$ of elements in $\{1,\ldots, n\}$.
\end{cor}
\begin{proof}
	We need to check that the coefficients of $t^k$ and $t^{k-1}$ in the following expression are zero:
	\begin{eqnarray*}
		(1+tu_{i_1i_2})(1+tu_{i_2i_3})\cdots (1+tu_{i_ki_1}).
	\end{eqnarray*}
	But by Proposition \ref{prop:triplesRelation} this equals
	\begin{eqnarray*}
		(1+tv_{i_1i_2i_3})(1+tv_{i_1i_3i_4})\cdots (1+tv_{i_1i_{k-1}i_k})
	\end{eqnarray*}
	and the highest power of $t$ which appears with nonzero coefficient is $t^{k-2}$.
\end{proof}

It is interesting to compare the elegance of exponentiation in Example \ref{example: Vn calculation n=4} with the relative complexity of the same formally equivalent identity in Proposition \ref{prop: triangulation independence complete blade}:
$$(\gamma_{1,2,3} \gamma_{1,3,4} = \gamma_{1,2,4}\gamma_{2,3,4})\ \ \Leftrightarrow (v_{123}v_{134} = v_{124} v_{234}).$$

\begin{example}\label{example: Vn calculation n=4}
	In $\mathcal{U}^n$ we have the relations
	$$\exp\left(v_{123}\right)\exp\left(v_{134}\right)=\exp\left(v_{123}+v_{134}\right)=\exp\left(u_{12} + u_{23} + u_{34} + u_{41}\right),$$
	hence the full exponential is an invariant of the boundary 1-skeleton of the polygon!  See also Appendix \ref{sec:Leading singularity} for a more general construction involving triangulations of polygons.  Continuing,
\begin{eqnarray*}
	&  & \exp\left(u_{12} + u_{23} + u_{34} + u_{41}\right)\\
	& = & 1+\left(u_{12} + u_{23} + u_{34} + u_{41}\right) + \frac{1}{2}\left(u_{12} + u_{23} + u_{34} + u_{41}\right)^2 + \frac{1}{3!}\left(u_{12} + u_{23} + u_{34} + u_{41}\right)^3+\cdots \\
	& = & 1+\left(u_{12} + u_{23} + u_{34} + u_{41}\right) + \left(u_{12}u_{23} + u_{12} u_{34} + u_{12} u_{41} + u_{23} u_{34} + u_{23} u_{41} + u_{34} u_{41}\right),
\end{eqnarray*}
where the degree 3 term vanishes by direct computation, or by Proposition \ref{prop:triplesRelation}.  The sum is in termwise bijection with the expression for $\Gamma_{i,j,k,\ell}$ from Proposition \ref{prop: triangulation independence complete blade} and in particular the expression for $\Gamma_{1,2,3,4}$ as a sum of elementary symmetric functions, as depicted in Figure \ref{fig:bladesstickframed3a}. 
\end{example}

By way of a specialization of the canonical basis from Theorem \ref{thm: cyclic sum canonical basis} to ordered set partitions having only singleton blocks, one would hope to have the following graded basis for the cohomology ring of the configuration space of points in $SU(2)$ modulo the diagonal action, from \cite{EarlyReiner}.  Indeed, using the so-called \textit{nbc} basis \cite{OrlikTerao}, in \cite{EarlyReiner} it was shown that one has the following analog of the canonical basis for graduated functions of blades.

\begin{prop}[\cite{EarlyReiner}]\label{prop:canonical Basis algebra}
		Writing
	each cycle $C$ of $w$ uniquely as 
	$C=(c_1 c_2  \cdots c_\ell)$ with convention $c_1=\min\{ c_1, c_2,\ldots, c_\ell\}$,
	then $(\mathcal{U}^{n-1})_j$ and $(\mathcal{V}^n)_j$ have the bases respectively
	$$
	\prod_{\text{cycles }C\text{ of }w} 
	u_{ c_1, c_2} \,\, u_{c_2, c_3} \cdots u_{c_{\ell-1} c_\ell}
	\qquad \text{ and }
	\prod_{\text{cycles }C\text{ of }w} 
	v_{ c_1, c_2, n} \,\, v_{c_2, c_3,n} \cdots v_{c_{\ell-1} c_\ell,n},
	$$
	where $w$ runs through all permutations in $\symm_{n-1}$ with $n-1-j$ cycles.
\end{prop}

\section{Configuration space of points on the circle}\label{sec: configuration space circle}

Let us denote by 
$$\mathcal{O}^n = U(1)^n\slash U(1)$$
the configuration space of $n$ points on the unit circle $U(1) = \{z\in \mathbb{C}: \vert z\vert=1\}$, modulo simultaneous rotation.  This could naturally be considered to be embedded on the diagonal in $U(n)$.  In this section we record a convenient labeling and parameterization of the components of $\mathcal{O}^n$.  We give parametrizations of the closed components of $\mathcal{O}^n$, leaving further exploration to future work.

Denote by $\Delta_1^n = \{x\in \lbrack 0,1\rbrack^n: \sum_{i=1}^n x_i=1\}$ the unit simplex.

\begin{defn}
	Let $\mathbf{S}=(S_1,\ldots, S_k)$ be an ordered set partition of $\{1,\ldots, n\}$.  For $x\in \Delta_1^n$, define
	\begin{eqnarray*}
		\varphi_{\mathbf{S}}(x_1,\ldots, x_n) & = & e^{2\pi i x_{S_1\cdots S_k}}e_{S_1} + e^{2\pi i x_{S_2\cdots S_k}}e_{S_2} + \cdots + e^{2\pi i x_{S_k}}e_{S_k}\\
		& = & e_{S_1} + e^{2\pi i x_{S_2\cdots S_k}}e_{S_2} + \cdots + e^{2\pi i x_{S_k}}e_{S_k},
	\end{eqnarray*}
	and denote the equivalence class of $\varphi_{\mathbf{S}}(x_1,\ldots, x_n)$ modulo simultaneous rotation by $U(1)$ by $\overline{\varphi}_{\mathbf{S}}(x_1,\ldots, x_n)$.
	We further define $\lbrack S_1,\ldots, S_k\rbrack^\circledast=\overline{\varphi}_{\mathbf{S}}(\Delta_1^n)$, and as usual denote by $\lbrack \lbrack S_1,\ldots, S_k\rbrack\rbrack^\circledast$ the characteristic function of $\lbrack S_1,\ldots, S_k\rbrack^\circledast$.

\end{defn}

Then it is easy to see that the image $\varphi_{\mathbf{S}}(\Delta_1^n)\subset \mathcal{O}^n$ fills out the closure of the unique cyclic order of points on the circle, including all possible additional collisions, which is oriented counterclockwise (i.e. with increasing angle) as 
$$x_{(S_1)}\leftarrow x_{(S_2)}\leftarrow\cdots\leftarrow x_{(S_k)}\leftarrow x_{(S_1)},$$
where we recall the shorthand notation $x_{(S)}$ which stands for $(x_{i_1}=\cdots =x_{i_s})$ for $S=\{i_1,\ldots, i_s\}$.

In Proposition \ref{prop: linear orders glued2} we show that the set of composite maps $\overline{\varphi}_\mathcal{S}:\Delta_1^n\rightarrow \mathcal{O}^n$ identifies all subsets of $\mathcal{O}^n$ which are labeled by cyclic rotations of the same ordered set partition.  
\begin{prop}\label{prop: linear orders glued2}
	For any ordered set partition $(S_1,\ldots, S_k)$ of $\{1,\ldots, n\}$ we have invariance under cyclic block rotation:
	$$\lbrack S_1,S_2,\ldots, S_k\rbrack^\circledast = \lbrack S_2,S_3,\ldots, S_k,S_1\rbrack^\circledast.$$
\end{prop}

\begin{proof}
	Let $\mathbf{S}=(S_1,\ldots, S_k)$ be an ordered set partition of $\{1,\ldots, n\}$, and let us denote $\mathbf{S}^{(j)} = (S_j,S_{j+1},\ldots, S_{j-1})$.  It suffices to check pointwise that 
	$\varphi_{\mathbf{S}^{(1)}}(x)$ and $\varphi_{\mathbf{S}^{(2)}}(x)$ differ only by a phase.
	
	For each $(x_1,\ldots, x_n)\in \Delta_1^n$ we have
	\begin{eqnarray*}
		\varphi_{(S_1,\ldots, S_k)}(x_1,\ldots, x_n) & = & e_{S_1} + e^{2\pi i x_{S_2\cdots S_k}}e_{S_2} + e^{2\pi i x_{S_3\cdots S_k}}e_{S_2}+ \cdots e^{2\pi i x_{S_k}}e_{S_k}\\
		& = & e^{2\pi i x_{S_2\cdots S_k}}\left(e^{2\pi i x_{S_1}}e_{S_1} + e_{S_2} +e^{2\pi i x_{S_3\cdots S_kS_1}}e_{S_3}+ \cdots + e^{2\pi i x_{S_kS_1}}e_{S_k}\right)\\
		& = & e^{2\pi i x_{S_2\cdots S_k}}\left( e_{S_2} +e^{2\pi i x_{S_3\cdots S_kS_1}}e_{S_3}+ \cdots + e^{2\pi i x_{S_kS_1}}e_{S_k}+e^{2\pi i x_{S_1}}e_{S_1}\right)\\
		& = & e^{2\pi i x_{S_2\cdots S_k}}\varphi_{(S_2,\ldots, S_k,S_1)}(x_1,\ldots, x_n),
	\end{eqnarray*}
	which differs from $\varphi_{(S_2,\ldots, S_k,S_1)}(x_1,\ldots, x_n)$ only by the phase $e^{2\pi i x_{S_2\cdots S_k}}$, hence we have pointwise 
	$$	\overline{\varphi}_{(S_1,\ldots, S_k)}(x_1,\ldots, x_n)=\overline{\varphi}_{(S_2,\ldots, S_k,S_1)}(x_1,\ldots, x_n),$$
	and since $(S_1,\ldots, S_k)$ was arbitrary, the identity follows.
\end{proof}

It is natural for $\mathcal{O}^n$ to enumerate the set of degenerate cyclic orders of $n$ particles on the circle,
$$\left\{\lbrack S_1,\ldots, S_k\rbrack^\circledast: (S_1,\ldots, S_k) \text{ is a standard ordered set partition of } \{1,\ldots, n\}\right\},$$
by the number of blocks: the entry $T_{n,k}$ in the number triangle counts the number of cyclic configurations of $n$ colliding particles on the circle having $k$ distinct positions.  One easily finds O.E.I.S. sequence A028246, see also A053440:
$$T_{n,k} = S(n,k)(k-1)!,$$
where $S(n,k)$ is the Stirling number of the second kind.  The number triangle begins with
$$
\begin{array}{ccccccc}
1 & &  &  &  &  & \\
1 & 1 & & &  &  & \\
1 & 3 & 2 &  & &  &  \\
1 & 7 & 12 & 6 &  & & \\
1 & 15 & 50 & 60 & 24 &  &  \\
1 & 31 & 180 & 390 & 360 & 120 &\\
\end{array}
$$
Here the rows total to the necklace numbers, as in Proposition \ref{prop: graded dimension blades}.

\section{Concluding remarks}\label{sec: concluding remarks}
In this paper, we have studied a new factorization property for permutohedral honeycomb tessellations: using ring theoretic calculations with characteristic functions of blades, we proved that honeycomb tessellations are locally Minkowski sums of 2-dimensional permutohedral honeycombs.  We have also established a certain \textit{canonical} basis for graduated functions of blades.

We have only scratched the surface of what seems to be a vast subject.  Some comments are in order.

\begin{enumerate}
	\item For graduated functions of blades we prove linear independence for the canonical basis (in Theorem \ref{thm: cyclic sum canonical basis}), but we do not establish any factorization property.  Indeed, it turns out that due to the larger integer multiplicities on shared faces of the cyclic sum the simple factorization property does not hold for graduated blades.
	\item On the other hand, for characteristic functions of blades, which are $\{0,1\}$-valued, we prove the factorization property.   We leave all linear independence proofs to future work. 
	\item Further, we expect, but it was beyond the scope of the paper to prove, that graduated and characteristic functions of blades and span the same space.
		
	\vspace{.1in}
		\begin{center}
		\begin{tabular}{|c|c|c|}
			\hline
			& Basis? & Simple factorizations into tripods? \\ 
			\hline 
			$\{0,1\}$-valued, standard & Expected & \checkmark \\ 
			\hline 
			$\{0,1\}$-valued canonical & Expected & \checkmark \\ 
			\hline 
			$\bar{n}$-valued, standard & \checkmark & No \\ 
			\hline 
			$\bar{n}$-valued, canonical & \checkmark & No\\ 
			\hline 
		\end{tabular} 
	\end{center}

	\vspace{.1in}	
	\item Proving the closed formula for the general straightening relations for both characteristic and graduated functions of blades is beyond the scope of the present work we defer the exposition to future work.  Toric posets may be a relevant for this investigation (and more generally), see \cite{ToricOrders}.  
	
	However, modulo both characteristic functions of non-pointed and higher codimension cones the solution is manageable, particularly for the top-degree component.  In fact in this case it turns out that the straightening relations are combinatorially identical to the \textit{Kleiss-Kuijf} relations for the Parke-Taylor factors, see \cite{KK}, and have already been formulated graph-theoretically in Section 3.2 of \cite{Nonplanar}, which we can see by way of the isomorphism
	$$\Gamma_{i_1,\ldots, i_n}\mapsto  PT(i_1,\ldots, i_n),$$
	where 
	$$PT(i_1,\ldots, i_n) = \frac{1}{(x_{i_1}-x_{i_2})(x_{i_2} - x_{i_3})\cdots (x_{i_n}-x_{i_1})},$$
	$x_1,\ldots, x_n$ being complex variables.  Seeing that this is a homomorphism (in fact it is an isomorphism) is not difficult, but it partially involves structures studied in \cite{DeltaAlgebra2019} and we leave the proof to future work.  See also Example \ref{example: U(1) decoupling} for an analog of the $U(1)$-decoupling identity, for a functional representation of graduated functions of blades.  For related straightening relations, see the so-called canonicalization of pseudoinvariants in \cite{MafraSchlotterer2014}. 
	\item  As one can derive by simply counting multiplicities obtained by summing the characteristic functions of the blade $((1,2,3))$, see Figure \ref{fig:blade3coordinates0}, and its mirror image $((1,3,2))$, in dimension $\le 1$, the fundamental blade relations for characteristic functions take the form
	$$\Gamma_{1,2,3} + \Gamma_{1,3,2} = 1_{12} + 1_{23} + 1_{31} -1_{1}1_{2}1_{3}.$$
	On the other hand, for graduated functions (see Figure \ref{fig:blade-inverse-image-3-coords-2}) on $V_0^3$ the analogous expression takes the form
	$$\lbrack (1,2,3)\rbrack + \lbrack (1,3,2) \rbrack = \lbrack (1,23)\rbrack+ \lbrack (12,3)\rbrack + \lbrack (13,2)\rbrack - 1_{123} + 1_{1}1_{2}1_{3}.$$

	\begin{enumerate}
		\item Modulo characteristic functions of cones of \textit{codimension} $\ge 2$ in $V_0^3$, the fundamental blade relations take the form respectively	
		$$\Gamma_{1,2,3} + \Gamma_{1,3,2} = 1_{12} + 1_{23} + 1_{31}$$
		and
		$$\lbrack (1,2,3)\rbrack + \lbrack (1,3,2) \rbrack = \lbrack (1,23)\rbrack+ \lbrack (12,3)\rbrack + \lbrack (13,2)\rbrack - 1_{123}.$$
		\item Modulo characteristic functions of non-pointed cones, the fundamental blade relations take the form respectively
		$$\Gamma_{1,2,3} + \Gamma_{1,3,2} = -1_{1}1_{2}1_{3}.$$
		and
		$$\lbrack (1,2,3)\rbrack + \lbrack (1,3,2) \rbrack = 1_{1}1_{2}1_{3}.$$
		\item Modulo both we have antisymmetry, respectively
		$$\Gamma_{1,2,3} + \Gamma_{1,3,2} =0$$
		and
		$$\lbrack (1,2,3)\rbrack + \lbrack (1,3,2) \rbrack =0.$$
		One could compare these with the antisymmetry of the generator $v_{ijk}$ in the cohomology ring of Section \ref{sec: graded cohomology ring}.
	\end{enumerate}
	
\end{enumerate}

\vspace{.2in}

\section{Acknowledgements}
Many conversations with many people have contributed to this work.  We are grateful in particular to Adrian Ocneanu for many discussions during our graduate study at Penn State about permutohedral plates and blades and related topics and for encouraging us to explore and to develop our own approach.  We thank Freddy Cachazo, Darij Grinberg, Tamas Kalman, Donghyun Kim, Carlos Mafra, Sebastian Mizera, William Norledge, Oliver Schlotterer and Guoliang Wang for stimulating discussions.  We thank Nima Arkani-Hamed for discussions and for suggesting \cite{Nonplanar}, Pavel Etingof for pointing out \cite{EtingofVarchenko,Felder}, Alexander Postnikov for pointing out the connection to the Catalan matroid polytopes, and Victor Reiner for fruitful collaboration on the related paper \cite{EarlyReiner}.  We thank Massachusetts Institute of Technology for excellent working conditions while this paper was written.

\newpage

\appendix

\section{Proof by Donghyun Kim of generating function symmetry}\label{sec: Donghyun Kim proof}
Denote by $E_{k,\ell}$ the $k^\text{th}$ Eulerian number of the second kind.

\begin{proof}[Proof of Conjecture \ref{conjecture: higher codim blades generating function}, (1) and (3)]
	We show that the coefficients of the polynomial numerators of the generating functions for the diagonals of the array $T^{B}_{n,k}$ are symmetric, and they sum to the sequence $a(j) =\frac{(2j)!}{j!}$.

 We first derive a nonrecursive formula for the coefficients.
	
	The (unsigned ) Stirling number of the first kind and the Stirling number of the second kind have the following polynomial expression (See \cite{Graham Eulerian number reference}, p. 271, Equations 6.43 and 6.44),
	$$s(n,n-k)=\sum_{l=0}^{k-1}E_{k,l}\binom{n+l}{2k}$$
	$$S(n,n-k)=\sum_{l=0}^{k-1}E_{k,l}\binom{n+k-l-1}{2k}.$$
	
	Using those polynomial expressions, we have 
\begin{small}
		\begin{eqnarray*}
	&&T^{B}_{n,n-j}\\
	 &= &\sum_{k=0}^{j} S(n,n-j+k)s(n-j+k-1,n-j-1) \\
		&= &  \frac{1}{2}\left(\sum_{k=0}^{j} S(n,n-j+k)s(n-j+k-1,n-j-1) +S(n,n-k)s(n-k-1,n-j-1)\right)\\
		&=& \frac{1}{2}\left(\sum_{k=0}^{j} \left(\sum_{l=0}^{j-k-1}E_{j-k,l}\binom{n+j-k-l-1}{2(j-k)}\right)\left(\sum_{l=0}^{k-1}E_{k,l}\binom{n-j+k+l-1}{2k}\right)\right. \\
		&+&\left. \left(\sum_{l=0}^{k-1}E_{k,l}\binom{n+k-l-1}{2k}\right)\left(\sum_{l=0}^{j-k-1}E_{j-k,l}\binom{n-k+1-1}{2(j-k)}\right)\right) \\
		&=&\sum_{l_1,l_2}E_{j-k,l_1}E_{k,l_2}\left(\binom{n+j-k-l_1-1}{2(j-k)}\binom{n-j+k+l_2-1}{2k}+\binom{n+k-l_2-1}{2k}\binom{n-k-1+l_1}{2(j-k)}\right)
	\end{eqnarray*}
\end{small}	
	We will show that the coefficients of the polynomial numerators of the generating functions for 
\begin{small}
		$$\left(\binom{n+j-k-l_1-1}{2(j-k)}\binom{n-j+k+l_2-1}{2k}+\binom{n+k-l_2-1}{2k}\binom{n-k-1+l_1}{2(j-k)}\right)$$ are symmetric. 
\end{small}
	
	There exist constants $A_m$ such that  
	\begin{small}
		$$\binom{n+j-k-l_1-1}{2(j-k)}\binom{n-j+k+l_2-1}{2k}=\sum_{m} A_m \binom{n+m}{2j}.$$
	\end{small}
	This is an identity as a polynomial in $n$, so plugging in $(-n+j+1)$ for $n$, we have, 
\begin{small}
		$$\binom{-n+2j-k-l_1}{2(j-k)}\binom{-n+k+l_2}{2k}=\sum_{m} A_m \binom{-n+m+j+1}{2j}.$$
\end{small}
	
	This is equivalent to 
\begin{small}
		$$\binom{n-k+l_1-1}{2(j-k)}\binom{n+k-l_2-1}{2k}=\sum_{m} A_m \binom{n+j-2-m}{2j}.$$
\end{small}
	So we have, 
\begin{small}
		\begin{equation*} 
	\begin{split}
	&\binom{n+j-k-l_1-1}{2(j-k)}\binom{n-j+k+l_2-1}{2k}+\binom{n+k-l_2-1}{2k}\binom{n-k-1+l_1}{2(j-k)} \\&=\sum_{m} A_m \binom{n+m}{2j}+\sum_{m} A_m \binom{n+j-2-m}{2j}\\
	& =\sum_{m} (A_m+A_{j-2-m}) \binom{n+m}{2j}.
	\end{split}
	\end{equation*}
\end{small}
	
	Thus the polynomial numerators of the generating functions for
	$$\left(\binom{n+j-k-l_1-1}{2(j-k)}\binom{n-j+k+l_2-1}{2k}+\binom{n+k-l_2-1}{2k}\binom{n-k-1+l_1}{2(j-k)}\right)$$
	have the form
	$$(A_0+A_{j-2})+(A_1+A_{j-3})x+\cdots+(A_{j-3}+A_1)x^{j-3}+(A_{j-2}+A_0)x^{j-2},$$
	which has symmetric coefficients. 
	
	Note that $T^{B}_{n,n-j}$ is a degree $2j$ polynomial in $n$; let us denote by $c_{2j}$ the coefficient of $x^{2j}$. Then the sum of the coefficients of polynomial numerator of the generating function for $T^{B}_{n,n-j}$ is $(2j)! c_{2j}$.
	
	As the Eulerian numbers of the second kind sum up to $(2k-1)!!$ for each $k$, the coefficients of the leading term for $S(n,n-k)$ and $s(n,n-k)$ are both $\frac{(2k-1)!!}{(2k)!}$. So we have, 
	
	$$c_{2j}=\sum_{k=0}^{j}\frac{(2(j-k)-1)!!}{(2(j-k))!}\frac{(2k-1)!!}{(2k)!}=\frac{1}{(2j)!}\sum_{k=0}^{j}\binom{2j}{2k}(2(j-k)-1)!!(2k-1)!!.$$
	
	Now $\sum_{k=0}^{j}\binom{2j}{2k}(2(j-k)-1)!!(2k-1)!!$ counts the number of bicolored (black and white) perfect matching on $2j$ elements, as we first pick the $2k$ elements which we will match with black edges and $(2(j-k)-1)!!(2k-1)!!$ will count number of perfect matchings for each color. So we have, 
	
	$$(2j)!c_{2j}=2^{j}(2j-1)!!=\frac{(2j)!}{j!}.$$
	
\end{proof}

\section{Symmetries of a leading singularity}\label{sec:Leading singularity}

In Theorem \ref{thm: cyclic sum canonical basis} we saw that factorizations of characteristic functions of blades correspond to the triangulations of a cyclically-oriented polygon.  It is natural to ask about products which do \textit{not} correspond to triangulations, but rather to graphs embedded in higher dimensional objects; indeed, it turns out that this case is similar to something in the scattering amplitudes literature known as a \textit{leading singularity}.  Further, the examples and discussion which follows suggests a new class of identifications for non-planar on-shell diagrams beyond the well-known square move which deserves further study, see \cite{DeltaAlgebra2019} where the study was initiated.

While our computations rely on relations which hold in the cohomology ring $\mathcal{U}^n$, namely $v_{ijk} = -v_{ikj}$, $v_{ijk}^2=0$ and $v_{ijk}v_{ik\ell}+v_{ik\ell}v_{i\ell j} + v_{i\ell j} v_{ijk}=0$, it makes sense to ask whether the leading singularities deform manageably when the formal generator $v_{ijk}$ is ``replaced'' with the characteristic function $\gamma_{i,j,k}$.  The practical difficulty is that the above relations on the $v_{ijk}$'s are degenerations of the relations on the $\gamma_{i,j,k}$, so there are many more moving parts.  For example, the relation $v_{123} + v_{132}=0$ now deforms to the fundamental identity
$$\gamma_{1,2,3} + \gamma_{1,3,2} = 1_{12} + 1_{23} + 1_{31} - 1.$$
However, in the algebra $\mathcal{V}^n$ the relations are very simple, and due to the nilpotence of the generators $v_{ijk}^2=0$ one has the exponential map.

\begin{figure}[h!]
	\centering
	\includegraphics[width=0.7\linewidth]{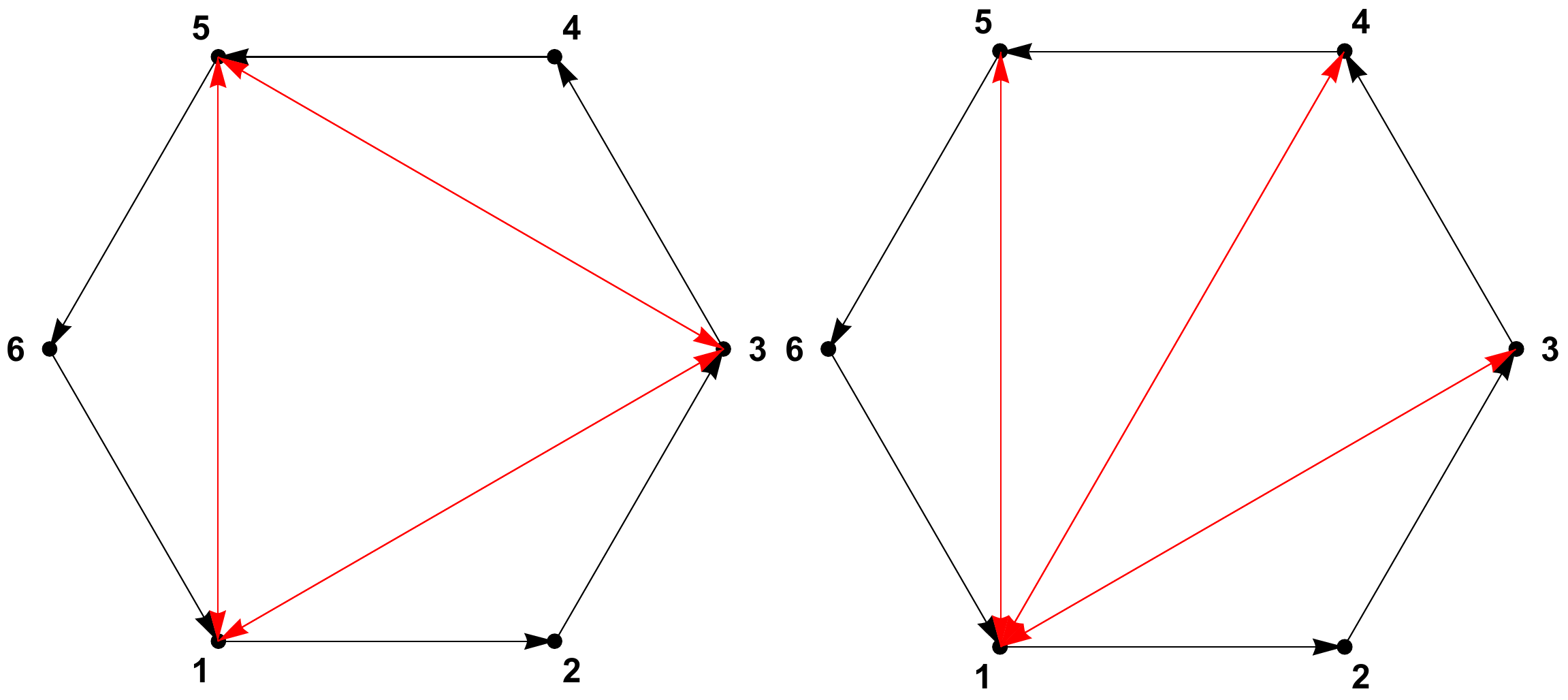}
	\caption{Triangulation change: $$\{(1,2,3),(3,4,5),(5,6,1),(1,3,5)\} \Leftrightarrow \{(1,2,3),(1,3,4),(1,4,5),(1,5,6)\}.$$
	Internal arrows overlap in opposite directions and cancel (triangulation succeeds), see Example \ref{example: 6-particle standard}.}
	\label{fig:triangulationchange6}
\end{figure}

\begin{example}\label{example: 6-particle standard}
	The simplest example of a triangulation of a polygon that is not a flag occurs for a hexagon at $n=6$, see Figure \ref{fig:triangulationchange6}.  This central triangulation corresponds to the coefficient of $t^4$ in the following expression:
	\begin{eqnarray*}
		v_{123} v_{345}v_{561}v_{135} & = & \text{coeff}_{t^{4}}\left((1+tv_{123})(1+tv_{345})(1+tv_{561})(1+tv_{135})\right)\\
		& = &\text{coeff}_{t^{4}}((1+tv_{123})(1+tu_{13})(1+tv_{345})(1+tu_{35})(1+tv_{561})(1+tu_{15}))\\
		& = & \text{coeff}_{t^{4}}\left((1+tu_{12})(1+tu_{23})(1+tu_{34})(1+tu_{45})(1+tu_{56})(1+tu_{61})\right)\\
		& = & \text{coeff}_{t^4}\left((1+tv_{123})(1+tv_{134})(1+tv_{145})(1+tv_{156})\right)\\
		& = & v_{123}v_{134}v_{145}v_{156},
	\end{eqnarray*}
	having used 
	$$(1+tv_{abc}) = (1+tu_{ab})(1+tu_{bc})(1+tu_{ca}),$$
	or, directly, the flip move $v_{135}v_{345} = v_{134}v_{145}$, as also holds for characteristic functions of blades, $\gamma_{135}\gamma_{345} = \gamma_{134}\gamma_{145}$.
	On the other hand, we have the interesting product which does \textit{not} correspond to a triangulation, but which \textit{does} have meaning as a certain \textit{leading singularity}, as seen in \cite{Nonplanar}.  In the derivation we shall make use of the identities $\exp\left(tu_{ij}\right) = 1+tu_{ij}$ and $\exp\left(tv_{ijk}\right) = 1+tv_{ijk}$, as $u_{ij}^2=0$ and $v_{ijk}^2=0$.  
	
	Then, for the set of triples $\{(1,2,3),(3,4,5),(5,6,1),(2,6,4)\}$ we have
\begin{eqnarray*}
	&&(1+v_{123}) (1+v_{345})(1+v_{561})(1+v_{264}) =  \exp\left((v_{123}+v_{345}+v_{561}+v_{264})\right)\\
	& = & \exp\left(u_{12}+u_{23}+u_{31}+u_{34}+u_{45}+u_{53}+u_{56}+u_{61}+u_{15}+u_{26}+u_{64}+u_{42}\right)\\
	& = & \exp\left(u_{12}+u_{23}+u_{34}+u_{56}+u_{61}\right)\exp\left(v_{153}+v_{264}\right).
\end{eqnarray*}
	\begin{figure}[h!]
		\centering
		\includegraphics[width=0.45\linewidth]{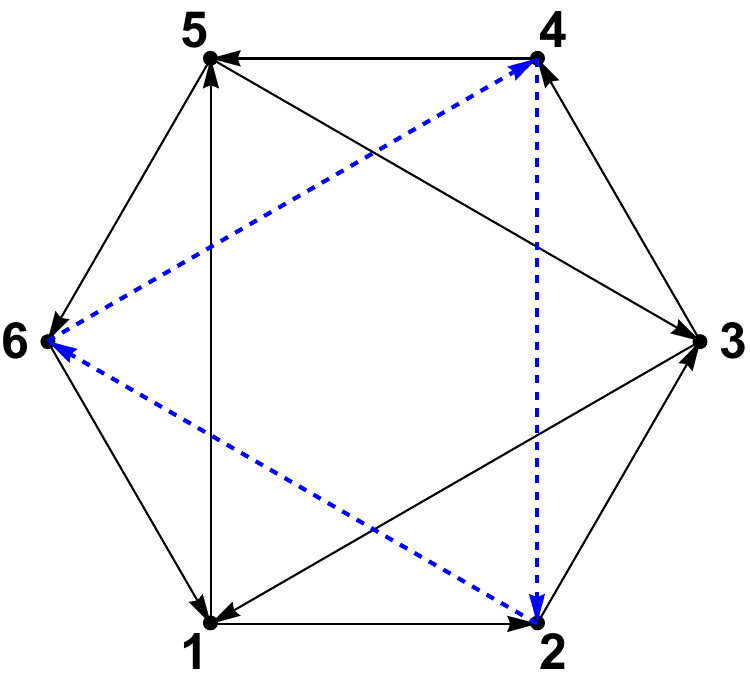}
		\caption{No internal edge cancellation occurs for the leading singularity of Example \ref{example: 6-particle standard}: $\{(1,2,3),(3,4,5),(5,6,1),(2,6,4)\}$.  For further discussion see \cite{DeltaAlgebra2019}.}
		\label{fig:hexagonorientedtriangulationnonplanar}
	\end{figure}
		Here each $u_{ij}$ corresponds to an oriented edge, as in Figure \ref{fig:hexagonorientedtriangulationnonplanar}.
\end{example}
	It is easy to check that the expansion here is invariant under rotation by the cycle $(123456)$, which would not be obvious from the factorization into triples.  This suggests that the existence of new kinds of moves, coming from higher-dimensional structures, see \cite{DeltaAlgebra2019}.

\section{Combinatorial scattering equations and balanced graphs}\label{sec: combinatorial scattering equations}

In this somewhat speculative section, let us look toward scattering amplitudes for inspiration and future work, specifically toward the one-loop worldsheet functions from \cite{NewBCJ}.  The story began in \cite{Cachazo Song Yuan}, where the \textit{scattering equations} were introduced.  These are a system of $n$ highly nonlinear equations in the $n$ complex variables $\sigma_1,\ldots, \sigma_n$:
$$\sum_{b\not=a} G_{ab}=0$$
for each $a=1,\ldots, n$, where we define $G_{ab}=\frac{s_{ab}}{\sigma_a-\sigma_b}.$  

Above the numerators $s_{ab} = s_{ba}$ for $1\le a<b\le n$ are known as the generalized Mandelstam invariants; for the present purposes, we may assume that they are complex numbers.  Remark that we regard these only as combinatorial objects.  We shall not here think about the usual linear fractional action of the gauge group $SL_2$ on the variables $\sigma_a$.

Remark that the functions $G_{ab}$ are obtained as limits of Kronecker-Eisenstein series, see \cite{Kronecker} and \cite{Zagier}.  In the context of scattering amplitudes, a good starting point would be \cite{MafraSchlotterer} and the references therein.  The same functions arise in elliptic solutions to the Classical Dynamical Yang-Baxter equation, see \cite{Felder} and \cite{EtingofVarchenko} and the references therein.  

The purely combinatorial approach which motivates this section is inspired in part by Appendix A of \cite{NewBCJ}, where the \textit{one-loop} analog of the scattering equations were in effect considered to generate an ideal in a certain ring of one-loop worldsheet functions.  That is, one considers the ring generated formally by the functions $G_{ab}$ and mods out by the ideal generated by the one-loop scattering equations, see \cite{LoopParkeTaylor}).

Our aim here is to study a combinatorial interpretation of the quotient ring from \cite{NewBCJ}: it turns out that this quotient ring is combinatorially temptingly close to (a filtered analog of) $\mathcal{V}^n$, and thus to blades.  It is interesting to note that a combinatorially similar quotient ring was studied in \cite{MoseleyProudfoot}, where it was conjectured to be isomorphic to the cohomology ring of the configuration space $X_n$ of $n$ distinct points in $SU(2)$, modulo the simultaneous action of $SU(2)$.  See Section \ref{sec: graded cohomology ring} above.

As we saw in Example \ref{example: Vn calculation n=4}, the exponential map can be used to efficiently encode the combinatorial structure of the $k$-skeleta of a blade uniquely from its 1-skeleton.  But is there a criterion to determine when a given polynomial is in the subalgebra $\mathcal{V}^n$ of $\mathcal{U}^n$?  We give a partial answer which consists of a set of linear relations on the coefficients of the argument of the exponential map $\exp:(\mathcal{U}^n)_1\rightarrow\mathcal{U}^n$, which restricts nicely as $\exp:(\mathcal{V}^n)_1\rightarrow\mathcal{V}^n$.  Clearly any monomial in the $v_{ijk}$ is the leading order coefficient of an exponential map, suggesting the possibility that the result could be extended.

\begin{prop}\label{prop:Combinatorial scattering equations}
	Let constants $m_{ij}\in\mathbb{C}$, $1\le (i\not=j)\le n$ be given; define $\alpha_{ij} = m_{ij}-m_{ji}$.  Then, the product 
	$$\prod_{1\le (i\not=j)\le n}(1+u_{ij})^{m_{ij}}\in\mathcal{U}^n$$
	is in the subalgebra $\mathcal{V}^n$ if and only if the constants $\alpha_{ij}$ satisfy what we call the \emph{combinatorial scattering equations},
	\begin{eqnarray*}
		\alpha_{12} + \alpha_{13} + \cdots + \alpha_{1n} & = & 0\\
		\alpha_{21} +\alpha_{23} + \cdots + \alpha_{2n} & = & 0\\
		\alpha_{31} +\alpha_{32} + \cdots + \alpha_{3n} & = & 0\\
		& \vdots &\\
		\alpha_{n1} + \alpha_{n2} + \cdots + \alpha_{n\ n-1} &  = &0. 
	\end{eqnarray*}
\end{prop}

\begin{proof}
	First note that $u_{ij}^2=0$ implies the identity
	\begin{eqnarray*}
		\prod_{1\le (i\not=j)\le n}(1+u_{ij})^{m_{ij}} & = & \prod_{1\le i<j\le n}\exp\left(\alpha_{ij}u_{ij}\right)\\
		& = & \exp\left(\sum_{1\le i<j\le n} \alpha_{ij} u_{ij}\right).
	\end{eqnarray*}	
	Now, the linear span of all the $u_{ij}$'s decomposes into a direct sum of two irreducible symmetric group representations: 
	$$(\mathcal{U}^n)_{(1)}\simeq V_{(n-1,1)} \oplus V_{(n-2,1,1)},$$
	with spanning sets respectively
	$$\left\{z_i: i=1,\ldots, n\right\}\ \text{ and } \left\{v_{ijk} = u_{ij} + u_{jk} + u_{ki}: 1\le i<j<k\le n\right\},$$
	and bases
	$$\left\{z_i: i=1,\ldots, n-1\right\}\ \text{ and } \left\{v_{1jk} = u_{1i} + u_{jk} + u_{k1}: 2\le j<k\le n\right\},$$
	say, where we define $z_i=\sum_{j\not=i}u_{ij}.$
	
	It is easy to see that the combinatorial scattering equations express the condition on the coefficients $\alpha_{ij}$ for a linear combination
	$$\sum_{1\le i<j\le n} \alpha_{ij}u_{ij}$$
	to be in $V_{(n-2,1,1)}$.  Indeed, in light of the decomposition
	$$u_{ij} = \frac{1}{n}\left(z_i-z_j\right) + \frac{1}{n}\left(\sum_{k\not=i,j}v_{i,j,k}\right)\in V_{(n,1,1)}\oplus V_{(n-2,1,1)},$$
	we have
	\begin{eqnarray*}
		\sum_{1\le i<j\le n} \alpha_{ij}u_{ij} & = & \frac{1}{n}\sum_{1\le i<j\le n}\alpha_{ij} \left(z_i-z_j + \sum_{k\not= i,j} v_{ijk}\right)\\
		& = & \frac{1}{n} \sum_{i=1}^n\left(\sum_{k\not=i} \alpha_{i,k}\right)z_i +  \frac{1}{n}\sum_{1\le i<j\le n}\alpha_{ij} \left(\sum_{k\not= i,j} v_{ijk}\right),
	\end{eqnarray*}
	where we have used $\alpha_{ij} = -\alpha_{ji}$.  Evidently, this is in $\mathcal{V}^n$ if and only if all coefficients in the first term,
	$$\frac{1}{n} \left(\sum_{k\not=i} \alpha_{ik}\right),$$
	vanish for each $i$.  This set coincides with exactly the combinatorial scattering equations for the $\alpha_{ij}$.
\end{proof}

Given any element 
$$C\in \mathcal{L}_{\mathcal{U}^n}=\{-1,0,1\}^{\binom{n}{2}},$$
we have
$$\exp\left(\sum_{1\le i<j\le n}C_{ij}u_{ij}\right)\in\mathcal{U}^n.$$
In particular, any canonical basis element for $\mathcal{U}^n$, (or $\mathcal{V}^n$), is obtained as the coefficient of the highest power of $t$ of $\exp\left(t\sum_{(i,j)\in C}u_{ij}\right)\in\mathcal{U}^n$
for some $C\in\{-1,0,1\}^{\binom{n}{2}}$.  It seems tempting to interpret the combinatorial scattering equations as describing the restriction (of $\exp$) to an $\binom{n-1}{2}$-dimensional subspace of $(\mathcal{U}^n)_1$.  In particular, for integer-valued weights, we obtain inside the subspace an $\binom{n-1}{2}$-dimensional sublattice $\mathcal{L}_{\mathcal{V}^n}\subset \mathcal{L}_{\mathcal{U}^n}$.  Thus, of particular interest for scattering amplitudes are elements $C\in \mathcal{L}_{\mathcal{V}^n}$ such that $\sum_{(i,j)\in C}u_{ij} = \sum_{a=1}^{n-2} v_{r_as_at_a}$ for some list of triples $\{(r_1,s_1,t_1),\ldots, (r_{n-2},s_{n-2},t_{n-2})\}$.  

There appears to be some graph-theoretic machinery at hand.  We now prove that to each balanced weighted graph (in the sense of \cite{LamPostnikov}) there exists an element of the algebra $\mathcal{V}^n$.

Let $\mathcal{G}^n$ be an unoriented graph on $n$ vertices $1,\ldots, n$.  Let us choose edge orientations $(i_1\rightarrow j_1),\ldots, (i_\ell\rightarrow j_\ell)$ such that $i_a<j_a$ for all $a=1,\ldots, \ell$.

\begin{defn}
	A graph $\mathcal{G}^n$, with each edge $i\rightarrow j$ equipped with a flow $m_{ij}$ (which could be in $\mathbb{C}$) from $i$ to $j$, for $1\le i\not =j\le n$, is said to be  \textit{balanced} provided that the net flux at every vertex is zero.  That is, for each $a\in\{1,\ldots, n\}$, the total flow entering $a$ is the same as the total flow leaving $a$:
	$$\sum_{b:\ b\not=a} m_{ab} = \sum_{b:\ b\not=a} m_{ba},$$
	or
	$$\sum_{b:\ b\not=a} (m_{ab}-m_{ba}) = 0.$$
\end{defn}
We recover immediately at a graph theoretic form of the combinatorial scattering equations.

\begin{cor}
	Let $\mathcal{G}^n$ be a graph on vertices $\{1,\ldots, n\}$, with flow $m_{ij}$ on the edge $i\rightarrow j$ for all distinct $i,j\in\{1,\ldots, n\}$.
	
	If $\mathcal{G}$ is balanced then the product
	$$\prod_{1\le (i\not=j)\le n}(1+u_{ij})^{m_{ij}}$$
	is in the subalgebra $\mathcal{V}^n$.
\end{cor}

\begin{proof}Suppose that $\mathcal{G}^n$ is balanced.  Then for the product
	$$\prod_{1\le (i\not=j)\le n}(1+u_{ij})^{m_{ij}}=\prod_{1\le i<j\le n}(1+(m_{ij}-m_{ji}) u_{ij}) = \exp\left(\sum_{1\le i<j\le n} (m_{ij}-m_{ji})u_{ij}\right),$$
	at each vertex $a=1,\ldots, n$ we have the total flux
	$$\sum_{b:\ b\not=a} (m_{ab}-m_{ba}) = \sum_{b:\ b\not=a} \alpha_{ab}=0.$$
	These are the combinatorial scattering equations with coefficients $\alpha_{ab} = m_{ab} - m_{ba}$.
\end{proof}

Of course, one can easily see that the same will be true for the graph of any leading singularity.

\begin{example}
	Assigning the weight $m_{ij}=+1$ to each directed edge $(i,j)$, then the graphs in Figures \ref{fig:triangulationchange6} and \ref{fig:hexagonorientedtriangulationnonplanar} are all balanced!
\end{example}
\section{Connecting with the $\Delta$-algebra}
Fix an integer $n\ge 3$.

Denote by $W^{n+n}$ the ($2n$-dimensional) the complex vector space which has basis the (anticommuting) Grassmann variables 
$$\theta_1,\ldots, \theta_n,\chi_1,\ldots, \chi_n.$$
In this section we show that $\mathcal{V}^n$ is isomorphic to the $\Delta$-algebra from \cite{DeltaAlgebra2019}, which is an $n$-parameter subalgebra of the Grassmann algebra $\mathcal{G}^n$ generated by the variables $\theta_1,\ldots, \theta_n,\chi_1,\ldots,\chi_n$.

Define 
$$W_0^{n+n} = \left\{\sum_{i=1}^n (\alpha_i\theta_i) + \sum_{i=1}^n (\beta_i\chi_i)\in W^{n+n}: \sum_{i=1}^n \alpha_i =0,\text{ and } \sum_{i=1}^n\beta_i =0\right\}.$$

For $$(x_1,\ldots, x_n)\in \{x\in\mathbb{C}^n: x_i\not=x_j\text{ for } i\not=j \},$$
define 
$$u_{ij} = \frac{\theta_{ij}\chi_{ij}}{x_{ij}} = \frac{(\theta_{i}-\theta_j)(\chi_i-\chi_j)}{x_i-x_j}.$$
Let $\mathfrak{U}^n \subset\mathcal{G}^n$ be the subalgebra generated by the elements $u_{ij}$ for $1\le i<j\le n$.
\begin{rem}\label{rem: top form}
	There is a unique (up to a scalar multiple) antisymmetric tensor of degree $2n-2$ on $W_0^n$, given by
	$$\prod_{\ell=1}^{n-1}(\theta_{\ell n}\chi_{\ell n}) = \prod_{\ell=1}^{n-1}(\theta_{\ell} - \theta_n)(\chi_{\ell}-\chi_n) = \prod_{\ell=1}^{n-1}(\theta_{\ell} - \theta_{\ell+1})(\chi_{\ell}-\chi_{\ell+1}),$$
	where $\theta_{ij} = \theta_i-\theta_j$ and $\chi_{ij} = \chi_i-\chi_j$.
	
	Note that while the \textit{factorization} is however not unique, the product $\prod_{\ell=1}^{n-1}(\theta_{\ell n}\chi_{\ell n})$ is \textit{symmetric} under simultaneous label permutations, that it is invariant with respect to the diagonal group embedding  $\symm_n\hookrightarrow \symm_n\times\symm_n$.
\end{rem}
\begin{cor}
	The algebra $\mathfrak{U}^n$ is a (faithful) representation of $\mathcal{U}^n$.
\end{cor}

\begin{proof}
	It is a now standard result, see the work of Orlik-Terao \cite{OrlikTerao}, that the (square-free part of the) algebra of reciprocals of linear forms $\frac{1}{x_i-x_j}$ has basis the nbc basis, as well as the \textit{canonical} from \cite{EarlyReiner}, see also Proposition \ref{prop:canonical Basis algebra} above.  It remains only to note that the relations from Definition \ref{defn: cohomology ring SU2} for $\mathcal{U}^n$ also hold for $\mathfrak{U}^n$.  Indeed, clearly
	$$u_{ij}^2 = \frac{(\theta_{ij}\chi_{ij})(\theta_{ij}\chi_{ij})}{x_{ij}^2}=0$$
	due to antisymmetry.  Commutativity of the $u_{ij}$'s follows since each $u_{ij}$ has even Grassmann degree. Moreover, by Remark \ref{rem: top form} we have
	$$u_{ij}u_{jk} + u_{jk} u_{ki} + u_{ki}u_{ij} = \left(\theta_{ij}\chi_{ij}\theta_{jk}\chi_{jk}\right)\left(\frac{1}{x_{ij}}\frac{1}{x_{jk}} + \frac{1}{x_{jk}}\frac{1}{x_{ki}}+\frac{1}{x_{ki}}\frac{1}{x_{ij}}\right)=0,$$
	where the last equality is a well-known (and easily verified) rational function identity.
	
\end{proof}

\end{document}